\documentclass[reqno]{amsart}

\pdfoutput=1

\usepackage[normalem]{ulem}

\usepackage[alphabetic]{amsrefs}
\usepackage{amsmath}
\usepackage{amsfonts}
\usepackage{amssymb}
\usepackage{enumerate}
\usepackage{amstext}
\usepackage{amsbsy}
\usepackage{amsopn}
\usepackage{bbm,amsthm}
\usepackage{amscd}
\usepackage[pdftex]{color}
\usepackage[pdftex]{graphicx}
\usepackage[all]{xy}
\usepackage{mathtools}
\usepackage{todonotes}

\usepackage{tikz}
\usetikzlibrary{intersections,positioning,patterns,arrows,decorations.markings}

\usepackage{todonotes}

\usepackage{amsxtra}
\usepackage{upref}
\usepackage{epstopdf}
\usepackage{graphicx,color}
\usepackage{hyperref}
\usepackage{longtable}
\usepackage[francais, english]{babel}

\makeatletter
\newsavebox{\@brx}
\newcommand{\llangle}[1][]{\savebox{\@brx}{\(\m@th{#1\langle}\)}%
  \mathopen{\copy\@brx\kern-0.5\wd\@brx\usebox{\@brx}}}
\newcommand{\rrangle}[1][]{\savebox{\@brx}{\(\m@th{#1\rangle}\)}%
  \mathclose{\copy\@brx\kern-0.5\wd\@brx\usebox{\@brx}}}
\makeatother

\DeclareFontFamily{OML}{rsfs}{\skewchar\font'177}
\DeclareFontShape{OML}{rsfs}{m}{n}{ <5> <6> rsfs5 <7> <8> <9> rsfs7
  <10> <10.95> <12> <14.4> <17.28> <20.74> <24.88> rsfs10 }{}
\DeclareMathAlphabet{\mathfs}{OML}{rsfs}{m}{n}

\newtheorem{theorem}{Theorem}
\newtheorem{lemma}[theorem]{Lemma}
\newtheorem{proposition}[theorem]{Proposition}
\newtheorem{corollary}[theorem]{Corollary}

\theoremstyle{definition}

\newtheorem{question}{Question}
\newtheorem{ltheorem}{Theorem}

\newtheorem{innercustomgeneric}{\customgenericname}
\providecommand{\customgenericname}{}
\newcommand{\newcustomtheorem}[2]{%
  \newenvironment{#1}[1]
  {%
   \renewcommand\customgenericname{#2}%
   \renewcommand\theinnercustomgeneric{##1}%
   \innercustomgeneric
  }
  {\endinnercustomgeneric}
}

\newcustomtheorem{customthm}{Theorem}
\newcustomtheorem{customprop}{Proposition}

\theoremstyle{remark}
\newtheorem{remark}[theorem]{\bf Remark}

\numberwithin{equation}{section}
\numberwithin{theorem}{section}

\newcommand{\intav}[1]{\mathchoice {\mathop{\vrule width 6pt height 3 pt depth  -2.5pt
\kern -8pt \intop}\nolimits_{\kern -6pt#1}} {\mathop{\vrule width
5pt height 3  pt depth -2.6pt \kern -6pt \intop}\nolimits_{#1}}
{\mathop{\vrule width 5pt height 3 pt depth -2.6pt \kern -6pt
\intop}\nolimits_{#1}} {\mathop{\vrule width 5pt height 3 pt depth
-2.6pt \kern -6pt \intop}\nolimits_{#1}}}

\newcommand{\intavl}[1]{\mathchoice {\mathop{\vrule width 6pt height 3 pt depth  -2.5pt
\kern -8pt \intop}\limits_{\kern -6pt#1}} {\mathop{\vrule width 5pt
height 3  pt depth -2.6pt \kern -6pt \intop}\nolimits_{#1}}
{\mathop{\vrule width 5pt height 3 pt depth -2.6pt \kern -6pt
\intop}\nolimits_{#1}} {\mathop{\vrule width 5pt height 3 pt depth
-2.6pt \kern -6pt \intop}\nolimits_{#1}}}

\newcommand{\norm}[1]{{\left\| #1 \right\|}}

\newcommand{\abs}[1]{\lvert#1\rvert}

\newcommand{\vertiii}[1]{{\left\vert\kern-0.2ex\left\vert\kern-0.2ex\left\vert #1 
    \right\vert\kern-0.2ex\right\vert\kern-0.2ex\right\vert}}

\newcommand{\hf}{\widehat{f}}
\newcommand{\hx}{\widehat{x}}
\newcommand{\hy}{\widehat{y}}

\newcommand{\hM}{\widehat{M}}
\newcommand{\hp}{\widehat{p}}

\newcommand{\hmu}{\widehat{\mu}}

\newcommand{\hSing}{\widehat{\mathfs S}}

\newcommand{\un}{\underline}
\newcommand{\mc}{\mathcal}

\newcommand{\ve}{\varepsilon}
\newcommand{\wt}{\widetilde}
\newcommand{\wh}{\widehat}
\newcommand{\vt}{\vartheta}
\newcommand{\vf}{\varphi}

\newcommand{\R}{\mathbb{R}}
\newcommand{\N}{\mathbb{N}}

\newcommand{\Z}{\mathbb{Z}}

\newcommand{\whS}{\widehat{\Sigma}}
\newcommand{\whs}{\widehat{\sigma}}

\def\deg{\operatorname{deg}}

\renewcommand{\exp}[1]{{\rm exp}_{#1}}

\newcommand{\Sas}{d_{\rm Sas}}
\newcommand{\inj}{{\rm inj}}
\newcommand{\nuh}{{\rm NUH}}
\newcommand{\loc}{{\rm loc}}
\newcommand{\musrb}{\mu_{\mbox{\tiny Leb}}}

\begin{document}

\title[Measures of maximal entropy for non-uniformly hyperbolic maps]{Measures of maximal entropy for\\ non-uniformly hyperbolic maps}

\author{Yuri Lima,  Davi Obata and Mauricio Poletti}

\address{Instituto de Matemática e Estatística, Universidade de São Paulo, Rua do Matão, 1010, Cidade Universitária, 05508-090. São Paulo -- SP, Brazil}
\email{yurilima@gmail.com}
\address{Department of Mathematics, Brigham Young University, Provo, Utah, 84602, USA}
\email{davi.obata@mathematics.byu.edu}
\address{Departamento de Matem\'atica, Universidade Federal do Cear\'a (UFC), Campus do Pici,
Bloco 914, CEP 60455-760. Fortaleza -- CE, Brasil}
\email{mpoletti@mat.ufc.br}

\date{\today}
\thanks{
YL was supported by 
CNPq/MCTI/FNDCT project 406750/2021-1,
FUNCAP grant UNI-0210-00288.01.00/23, Instituto Serrapilheira grant
``Jangada Din\^{a}mica: Impulsionando Sistemas Din\^{a}micos na Regi\~{a}o Nordeste'', and
FAPESP grant number 2025/11400-7.
DO was partially supported by the National Science Foundation under Grant No. \ DMS-2349380. 
MP was partially supported by CAPES-Finance Code 001, Instituto Serrapilheira grant number Serra-R-2211-41879 and FUNCAP grant AJC 06/2022. 
This study was supported by the Fulbright-ABC Jacob Palis Award for Affiliated Members
of the Brazilian Academy of Sciences.
}

\begin{abstract}
For $C^{1+}$ maps, possibly non-invertible and with singularities, we prove that each homoclinic class
of an ergodic adapted hyperbolic measure carries at most one adapted hyperbolic measure of maximal entropy. 
We then apply this to study the finiteness/uniqueness
of such measures in several different settings: finite horizon dispersing billiards, codimension one partially hyperbolic endomorphisms with ``large'' entropy, robustly non-uniformly hyperbolic volume-preserving endomorphisms as in 
Andersson-Carrasco-Saghin (2025), and Viana maps (1997).
\end{abstract}

\maketitle

\tableofcontents

\section{Introduction}\label{Section-introduction}

In the theory of dynamical systems, entropy quantifies the level of chaos within a system. Among the various definitions of entropy, we focus on {\em topological entropy}, which measures the exponential growth rate of distinguishable trajectories as the measurement precision approaches zero, and {\em measure-theoretic entropy}
-- also known as {\em Kolmogorov-Sina{\u\i} entropy} -- which assesses the information-theoretic complexity of an invariant probability measure. These two concepts are connected, in well-behaved contexts, via the variational principle, which states that the topological entropy equals the supremum of the measure-theoretic entropies 
across all invariant probability measures.

A measure that is invariant and whose metric entropy equals the topological entropy is called a 
{\em measure of maximal entropy}. This measure encapsulates significant dynamical information, and understanding it provides insight into many statistical properties of the system. Since the 1970s, a key question in the field has been the existence, finiteness, and uniqueness of such measures of maximal entropy.

Newhouse proved the existence of measures of maximal entropy for $C^\infty$ maps \cite{Newhouse-Entropy}.
In a recent work,  Buzzi, Crovisier and Sarig proved that $C^{\infty}$ surface diffeomorphisms have a finite number of ergodic measures of maximal entropy, and a unique one if the system is transitive \cite{BCS-Annals}.

A key component of \cite{BCS-Annals} is a criterion that guarantees that two ergodic measures
 of maximal entropy coincide, which is obtained using the symbolic dynamics constructed by Sarig~\cite{Sarig-JAMS}. More specifically, the authors prove that for any hyperbolic ergodic measure $\mu$, there exists an
 {\em irreducible} countable topological Markov shift that ``captures'' the behavior of any other 
 ``sufficiently hyperbolic'' measure which is homoclinically related to $\mu$ (see Section \ref{ss.hom.relation} for the definition of homoclinic relation of measures).  One can then use results of countable topological Markov shifts to conclude the existence of at most one measure of maximal entropy which is homoclinically related to $\mu$. 
Since then, this criterion has been used to prove the finiteness/uniqueness of measures of maximal entropy
in various contexts, see e.g. \cite{obata-standard,LP-homoclinic-flow,mongez-pacifico-finite}.

The two main goals of this work are to generalize this criterion for non-invertible maps with singularities, and 
then to apply it in several different settings. We begin stating a non-precise version of the criterion.
We say that a map is of class $C^{1+}$ when it is $C^{1+\beta}$ for some $\beta>0$.

Let $M$ be a smooth manifold with finite diameter, possibly with boundary,
let $\mathfs D\subset M$ be a closed set, and let 
$f:M\to M$ such that $f|_{M\setminus \mathfs D} :M\setminus \mathfs D \to M$ is a $C^{1+}$ map with
critical set $\mathfs C=\{x\in M\setminus \mathfs D:d_xf\text{ is not invertible}\}$.
The singular set of $f$ is defined as $\mathfs S=\mathfs C \cup \mathfs D$, and we assume
it is closed.

The statement of the next theorem requires some terminology, which we will 
informally introduce now and postpone the formal definition for the next sections.
We require $f$ to satisfy some geometrical and dynamical conditions (A1)--(A7),
introduced in Section \ref{subsection-a1a7}. These conditions are satisfied in many 
cases of interest, e.g. if the curvature tensor of $M$ and its derivatives have bounded norms
and $\|df^{\pm 1}\|,\|d^2f^{\pm 1}\|$ grow at most polynomially fast with respect to
the distance to $\mathfs S$. In particular, they are satisfied for finite horizon dispersing billiards.

Among the $f$--invariant probability measures $\mu$, we focus
on the ones that are \emph{adapted} (the function $\log d(x,\mathfs S)\in L^1(\mu)$)
and {\em hyperbolic} (all Lyapunov exponents are non-zero $\mu$--a.e.).
Finally, two adapted hyperbolic measures $\mu,\nu$ are \emph{homoclinically related} if the stable manifold of
$\mu$--a.e. point transversally intersects the unstable manifold of $\nu$-a.e. point and vice-versa,
where invariant manifolds are considered in the sense of Pesin, see Section~\ref{ss.hom.relation}.

\begin{ltheorem}
\label{thm.criterion}
Let $M$ be a smooth manifold with finite diameter, possibly with boundary, and let $f:M\to M$ be a map that
is $C^{1+}$ outside the singular set $\mathfs S$ and that
verifies conditions $(\mathrm{A}1)$--$(\mathrm{A}7)$.
In each homoclinic class of an ergodic adapted hyperbolic measure there exists
at most one adapted hyperbolic measure of maximal entropy.
When it exists, it is isomorphic to a Bernoulli shift times a finite rotation and its
support equals ${\rm HC}(\mathcal O)$ for every hyperbolic periodic orbit
$\mathcal O$ homoclinically related to $\mu$.
\end{ltheorem}

Above, a measure is isomorphic to a Bernoulli shift times a finite rotation if its lift to the natural
extension satisfies this property.

In the sequel, we describe the main results obtained in four different settings on which
we apply Theorem \ref{thm.criterion}: finite horizon dispersing billiards; partially hyperbolic endomorphisms with one dimensional center; the new examples of non-uniformly hyperbolic volume-preserving endomorphisms introduced by Andersson, Carrasco and Saghin \cite{ACS}; and strongly transitive non-uniformly expanding maps with singularities, 
which include Viana maps \cite{Viana-maps}.

\subsection{Finite horizon dispersing billiards}

Consider finitely many pairwise disjoint closed, convex subsets $O_1,\ldots,O_\ell$
of $\mathbb{T}^2 = \mathbb{R}^2 / \mathbb{Z}^2$ such that each boundary $\partial O_i$ is a 
$C^3$ curve with strictly positive curvature. Inside the {\em billiard table} 
$\mathfs T:= \mathbb{T}^2 \setminus (\bigcup_{i=1}^\ell O_i)$,
we consider a particle moving at unit speed in straight lines and performing 
elastical collisions with $\partial \mathfs T$. 
Parameterizing each $\partial O_i$ by arclength $r$ and letting $\vf \in [-\pi/2, \pi/2]$ denote the angle
made by the post-collision velocity vector and the inward normal to $\partial O_i$ at the point of collision, 
we obtain a {\em billiard map} $f:M\to M$ where $M = \bigcup_{i=1}^\ell (\partial O_i \times [-\pi/2, \pi/2])$,
which represents the mechanical law of evolution of collisions of the particle with $\partial\mathfs T$.
A billiard of this type is called a {\em dispersing} or {\em Sina{\u\i} billiard}.
These maps were introduced and first studied extensively by Sina{\u\i} \cite{Sinai-billiards},
and are examples of maps with singularities. More specifically,
letting $\mathfs S_0=\{(r,\vf)\in M:|\vf|=\pi/2\}$ and $\mathfs S=\mathfs S_0\cup f^{-1}(\mathfs S_0)$,
then $f$ is a $C^2$ diffeomorphism from $M\setminus \mathfs S$ onto its image. Among the $f$--invariant
measures, we consider the adapted ones, as defined in Section \ref{subsection-a1a7}.

It is well-known that $f$ preserves a smooth probability measure $\musrb$.
Let $\tau(x)$ be the flight time from $x\in M$ 
to $f(x)$. We assume that $f$ has {\em finite horizon}, i.e. $\sup_{x\in M}\tau(x)<\infty$.
Baladi and Demers introduced an ad hoc definition of topological entropy $h_{\rm top}(f)$ for 
finite horizon dispersing billiard and, using transfer operator methods, proved that if such
billiard satisfies a sparse recurrence condition to the singular set then it has a unique measure of maximal entropy,
and that it is adapted, Bernoulli, hyperbolic and fully supported \cite{Baladi-Demers-MME}. 
We prove the following result. Let $E^u_x$ denote the unstable direction at $x$.

\begin{ltheorem}
\label{thm.billiards}
A finite horizon dispersing billiard has at most one adapted measure of maximal entropy. 
When it exists, it is Bernoulli, hyperbolic and fully supported.
Moreover, $\musrb$ is the unique adapted measure of maximal entropy if and only if
the values $\displaystyle \tfrac{1}{p} \log \|df^p_x|_{E^u_x}\|$ coincide
for every non grazing periodic point $x$ of period $p$, in which case the common value equals
$h_{\rm top}(f)$.
\end{ltheorem}
 
Theorem \ref{thm.billiards} is proved in Section \ref{sec.billiards}.
Observe that it does not require the sparse recurrence condition,
but it does not give the existence of measures of maximal entropy.
Up to our knowledge, this is still an open problem. Baladi and Demers
also proved that, under the sparse recurrence condition, if $\musrb$ is the unique
measure of maximal entropy, then $\displaystyle \tfrac{1}{p} \log \|df^p_x|_{E^u_x}\| = h_{\rm top}(f)$
for every non grazing periodic point $x$ of period $p$. The final 
part of Theorem \ref{thm.billiards} gives a complete characterization of this phenomenon, 
again not requiring the sparse recurrence condition.

Recently, Climenhaga et al constructed examples of finite horizon dispersing billiards with 
{\em non-adapted} measures of positive entropy \cite{CDLZ-24}. We then ask the following question.

\begin{question}
Is there a finite horizon dispersing billiard with a non-adapted measure of maximal entropy?
\end{question}

\subsection{Partially hyperbolic endomorphisms}
Let $M$ be a closed smooth Riemannian manifold of dimension $d$. Given $x\in M$ and $k \in \{1, \ldots, d-1\}$, a $k$--dimensional cone $\mathcal{C}$ is a subset of $T_xM$ defined as follows. There exist a decomposition into subspaces $T_xM = E \oplus F$, where $F$ has dimension $k$, and a constant $\eta>0$ such that $\mathcal{C}$ is the set of $k$--dimensional subspaces which are graphs of linear transformations $L:F \to E$ with norm $\|L\| < \eta$.  A continuous cone field is a choice of a cone for each $x\in M$ such that we can chose $x\mapsto E(x)$, $x\mapsto F(x)$ and $x\mapsto \eta(x)$ to be continuous functions.

Let $f:M \to M$ be an endomorphism for which there are constants $\chi, C>0$, a continuous line field $E^c$,
and a $\mathrm{dim}(M) -1$ dimensional continuous cone field $\mathcal{C}^u$, such that:
\begin{enumerate}[$\circ$]
\item The cone $\mathcal{C}^u$ is forward invariant, that is, $\overline{df_x(\mathcal{C}^u(x))} \subset \mathcal{C}^u(f(x))$, and for any $x\in M$ and unit vector $v\in \mathcal{C}^u(x)$,  
it holds $\|df_x^n v\| \geq C^{-1} e^{\chi n}$. We call such a cone field an \emph{unstable cone field};
\item $\|df^n_x|_{E^c}\| \leq Ce^{- \chi n} \|(df_x^n)^{-1}v\|^{-1}$ for all $x\in M$, $n\geq 0$ and $v\in T_{f^n(x)}M$ with $\|v\|=1$;
\item $\mathrm{dim}(E^c) = 1$.
\end{enumerate}
If $f$ verifies the above conditions, we call $f$ a \emph{codimension one partially hyperbolic endomorphism}. 
Notice that $E^c$ is uniquely defined, and in particular it is invariant by the dynamics. However, $E^u$ does not have to be invariant; therefore, we have an unstable cone.

\'Alvarez and Cantarino proved that every $C^1$ codimension one partially hyperbolic endomorphism 
admits a measure of maximal entropy \cite{alvarez-cantarino-23}. They also obtained a condition
for uniqueness when $M=\mathbb{T}^d$: if $f:\mathbb{T}^d \to \mathbb{T}^d$ is a
$C^1$ codimension one partially hyperbolic endomorphism, 
dynamically coherent, with quasi-isometric foliations
and such that its 
linear part\footnote{The linear part of $f$ is the action of $f$ on the first homology group $\mathbb{Z}^d$,
which is a matrix in $\mathrm{GL}(d, \mathbb{Z})$ and in particular induces a map on $\mathbb{T}^d$.}
is hyperbolic and a factor of $f$, then $f$ has a unique measure of maximal entropy.

Our second application of Theorem \ref{thm.criterion} is the following. Let $\operatorname{deg}(f)$ be
the topological degree of $f$ and $h_{\rm top}(f)$ the topological entropy of $f$. 
Then $h_{\rm top}(f)\geq \log \operatorname{deg}(f)$, see \cite{MP77a}.

\begin{ltheorem}\label{thm-phcodone}
Let $f:M \to M$ be a $C^{1+}$ codimension one partially hyperbolic endomorphism.
If $h_{\rm top}(f)>\log \operatorname{deg}(f)$ then $f$ has finitely many measures of maximal entropy.
Moreover, if for any $C^1$ curve $\gamma$ tangent to the unstable cone
the set $\bigcup_{n\geq 0}f^n(\gamma)$ is dense, then $f$ has a unique measure of maximal entropy.
\end{ltheorem}

Similar results for diffeomorphisms have been obtained recently in
\cite{mongez-pacifico-unique,mongez-pacifico-finite}. 
We note that the strict inequality $h_{\rm top}(f)>\log \operatorname{deg}(f)$ is necessary.
For example, let ${\rm Id}:\mathbb T \to \mathbb T$ be the identity and $g:\mathbb T \to \mathbb T$
be the doubling map $g(x) = 2x \pmod 1$.  The map $f = {\rm Id} \times g$
has topological entropy $h_{\rm top}(f) = \log 2$ and $\delta_x\times {\rm Leb}_{\mathbb T}$ is a measure
of maximal entropy, for every $x\in\mathbb T$.  One can easily adapt this construction to build examples having infinitely many hyperbolic measures of maximal entropy. 

%
%

In contrast to \cite{alvarez-cantarino-23}, Theorem \ref{thm-phcodone} requires higher regularity on $f$,
but the uniqueness criterion does not require conditions on $M$ nor on the linear part of $f$. 
We therefore ask the following questions.

\begin{question}
Is Theorem~\ref{thm-phcodone} true for $f\in C^1$?
\end{question}


\begin{question}
Does a transitive $C^{1+}$ codimension one partially hyperbolic endomorphism
has a unique measure of maximal entropy?
\end{question}

\subsection{Non-uniformly hyperbolic volume-preserving endomorphisms}\label{ss.acs}

Our third application deals with the non-uniformly hyperbolic volume-preserving endomorphisms
on $\mathbb T^2$ recently introduced and studied by Andersson, Carrasco and Saghin \cite{ACS}.
Let $\mathcal{U}$ be the set of $C^1$ volume-preserving endomorphisms
satisfying the following condition: there are $N,c>0$ such that for all $x\in \mathbb{T}^2$ and all
$v\in T_x \mathbb{T}^2$ unitary, it holds
\begin{equation}\label{eq.pastexpansion}
\sum_{f^N(y)=x} \frac{\log \norm{(df^N_y)^{-1}v}}{\abs{\det(df^N_y)}}>c.
\end{equation}
It was proved that for every $f\in\mathcal U$
the volume measure is hyperbolic, with one positive and one negative Lyapunov exponent \cite[Theorem A]{ACS}.
Additionally, this open set contains concrete examples without dominated splitting
in almost all homotopy classes, as we now explain. 
Let $s:\mathbb T\to \R$ be defined by $s(x)=\sin(2\pi x)$, or more generally be a function satisfying some properties (listed in \cite[Section 3]{ACS}), and for each $t\in\R$
consider the {\em shear} $h_t:\mathbb T\to \mathbb T$ defined by $h_t(x,y)=(x,y+ts(x))$.
Let $E=(e_{ij})\in\mathrm{GL}(2,\mathbb{R})$ with integer coefficients such that:
\begin{enumerate}[$\circ$]
\item $E$ is not an homothety;
\item $\pm 1$ is not an eigenvalue of $E$;
\item  $\abs{\det(E)}/{\rm gcd}(e_{11},e_{12},e_{21},e_{22})>4$, where ${\rm gcd}$ is the greatest common divisor.  
\end{enumerate}
Finally, let $f_t=E \circ P\circ h_t \circ P^{-1}$, where $P\in \mathrm{SL}(2,\Z)$, and observe that $f_t$ is a volume-preserving endomorphism isotopic to $E$.
Then $f_t\in\mathcal U$ for all $t$ large enough and $P$ satisfying some conditions \cite[Section 3.1 and 3.4]{ACS}. In this context,
we prove the following result.  

\begin{ltheorem}\label{thm.unique.ACS}
For every $t$ large enough, there is a $C^1$ open set $\mathcal{U}_t\subset\mathcal U$
containing $f_t$ 
such that every $f\in \mathcal{U}_t$ of class $C^{1+}$
has at most one measure of maximal entropy. 
If it exists, this measure is Bernoulli and fully supported.
In particular, every $f\in \mathcal{U}_t$ of class $C^\infty$ has a unique measure of maximal entropy.
\end{ltheorem}

Recall that, in this work, a measure is Bernoulli for $f$ if its lift to the natural
extension is Bernoulli.
The class of matrices $E$ for which $f_t\in \mathcal U$ was recently
extended \cite{Janeiro,ramirez-vivas}.
%
%
%
%
%
%
%

As a consequence of our techniques, we also obtain a criterion for the ergodicity of elements
in $\mathcal U$.

\begin{ltheorem}\label{thm.ergodicity}
If $f\in \mathcal{U}$ is $C^{1+}$ and transitive, then $f$ is ergodic with respect to the Lebesgue measure. 
In particular, if $\pm1$ is not an eigenvalue of the linear part of $f$ then $f$ is stably Bernoulli. 
\end{ltheorem}

We say that $f$ is {\em stably Bernoulli} if every $C^{1+}$ endomorphism $g\in\mathcal U$ close 
to $f$ is Bernoulli. The condition on the linear part in Theorem \ref{thm.ergodicity}
implies transitivity \cite{Andersson-trans}.
 
 The above theorem improves \cite[Theorem C]{ACS}, where the authors prove the existence of stably Bernoulli endomorphisms in a subset of $\mathcal{U}$. The main difference from their theorem to ours is that we do not require ``large'' stable manifolds (see \cite{ACS} for the definition of ``large'').

A natural question is to extend Theorem \ref{thm.unique.ACS} to $\mathcal U$.

\begin{question}
Does every transitive $f\in \mathcal{U}$ admit at most one measure of maximal entropy?
\end{question}

\subsection{Non-uniformly expanding maps with singularities}

As a last application, we consider non-invertible non-uniformly expanding maps.
Let $f$ be a map as in Theorem \ref{thm.criterion}. An $f$--invariant probability
measure is called \emph{expanding} if its Lyapunov exponents are all positive.

Let $a_0\in (1,2)$ be a parameter
such that $t=0$ is pre-periodic for the quadratic map $t\mapsto a_0-t^2$. For fixed $d\geq 2$ and
$\alpha>0$, the associated Viana map is the skew product 
$f=f_{a_0,d,\alpha}:\mathbb S^1\times \mathbb{R}\to \mathbb S^1\times \mathbb{R}$ defined by 
$f(\theta,t)=(d\theta,a_0+\alpha \sin(2\pi \theta)-t^2)$. 
If $\alpha>0$ is small enough then there is a compact interval $I_0\subset (-2,2)$ such that 
$f(\mathbb S^1\times I_0)\subset \textrm{int}(\mathbb S^1\times I_0)$. Hence $f$ has 
an attractor inside $\mathbb S^1\times I_0$, and so we consider the restriction of $f$ to
$\mathbb S^1\times I_0$. Observe that $f$ has a critical
set $\mathbb S^1\times\{0\}$, where $df$ is non-invertible.
These maps were introduced by Viana in  \cite{Viana-maps}, 
where he showed the robust existence of a uniformly positive Lyapunov exponent, 
see also \cite{buzzi-Vmaps}. We let $\mathfs S$
denote this critical set. On a $C^2$ neighborhood of $f$, conditions
(A1)--(A7) are satisfied \cite[Proposition 12.1]{ALP-23}.
In this context, we provide another proof of the uniqueness of the measure of maximal entropy,
whose existence and uniqueness was recently proved~\cite{pinheiro-23}.

\begin{ltheorem}\label{thm.viana.maps}
If $\alpha>0$ is small, there exists a $C^3$ neighborhood $\mathfs U$ of $f$ such that 
every map in $\mathfs U$ has at most one measure of maximal entropy.
When it exists, it is Bernoulli.
\end{ltheorem}

Theorem \ref{thm.viana.maps} improves \cite[Theorem~1.6]{ALP-23}, where it was proved that $f$ has at
most countably many ergodic measures of maximal entropy, each of them Bernoulli up to a period.
In a recent paper, Li studied the uniqueness of equilibrium states of potentials of small variation
for the Viana maps \cite{Li2025}.

\subsection{Organization of the paper}
In Section~\ref{Section-preliminaries} we introduce some preliminaries.
Section~\ref{Section-ALP} begins with the presentation of the technical conditions (A1)--(A7) required on
Theorem~\ref{thm.criterion} and a recast of the main tools and results of \cite{ALP-23}.
Then we prove Theorem~\ref{Thm-Main} and use it to establish Theorem~\ref{thm.criterion}.

Sections \ref{sec.billiards} to \ref{Section-exp} are devoted to applications of 
Theorem~\ref{thm.criterion}. In Section~\ref{sec.billiards} we prove Theorem \ref{thm.billiards},
in Section~\ref{Section-ph} we prove Theorem~\ref{thm-phcodone},
in Section~\ref{Section-ACS} we prove Theorems~\ref{thm.unique.ACS} and \ref{thm.ergodicity},
and in Section~\ref{Section-exp} we prove Theorem \ref{thm.viana.maps}.

The appendix~\ref{appendix} gives the proof of the inclination lemma in our context.

\medskip
\noindent
{\em{Acknowledgements.}} We are thankful to J. Buzzi, S. Crovisier and O. Sarig for helpful conversations on this work, and the anonymous referee for the comments that greatly improved the manuscript.
YL also thanks Universidade Federal do Ceará, for all the support provided during the development
of this work.

\section{Preliminaries}\label{Section-preliminaries}

\subsection{Topological Markov shifts}

Let $\mathfs G=(V,E)$ be an oriented graph, where $V,E$ are the vertex and edge sets.
We denote edges by $v\to w$, and assume that $V$ is countable.

\medskip
\noindent
{\sc Topological Markov shift (TMS):} It is a pair $(\Sigma,\sigma)$
where
$$
\Sigma:=\{\text{$\Z$--indexed paths on $\mathfs G$}\}=
\left\{\un{v}=\{v_n\}_{n\in\Z}\in V^{\Z}:v_n\to v_{n+1}, \forall n\in\Z\right\}
$$
is the symbolic space and $\sigma:\Sigma\to\Sigma$, $[\sigma(\un v)]_n=v_{n+1}$, is the {\em left shift}. 
We endow $\Sigma$ with the distance $d(\un v,\un w):={\rm exp}[-\inf\{|n|\in\Z:v_n\neq w_n\}]$.
The {\em regular set} of $\Sigma$ is
$$
\Sigma^\#:=\left\{\un v\in\Sigma:\exists v,w\in V\text{ s.t. }\begin{array}{l}v_n=v\text{ for infinitely many }n>0\\
v_n=w\text{ for infinitely many }n<0
\end{array}\right\}.
$$

\medskip
We will sometimes omit $\sigma$ from the definition, referring to $\Sigma$ as a TMS.
We only consider TMS that are \emph{locally compact}, i.e.
for all $v\in V$ the number of ingoing edges $u\to v$ and outgoing edges $v\to w$ is finite.

\medskip
\noindent
{\sc Irreducible component:}
If $\Sigma$ is a TMS defined by an oriented graph
$\mathfs{G}=(V,E)$, its \emph{irreducible components} are the subshifts $\Sigma'\subset\Sigma$
defined over maximal subsets $V'\subset V$ satisfying the following condition:
$$\forall v,w\in V',\;\exists \un v\in \Sigma \text{ and } n\geq 1\text{ such that } v_0=v \text{ and } v_n=w.$$

\subsection{Natural extensions}\label{Section-natural-extension}

Most of the discussion in this section is classical, see e.g. \cite{Rohlin-Exactness} or
\cite[\S 3.1]{Aaronson-book}.
Given a map $f:M\to M$, let
$$
\wh M:=\{\wh x=(x_n)_{n\in\Z}:f(x_{n-1})=x_n, \forall n\in\Z\}.
$$
We will write $\wh x=(\ldots,x_{-1};x_0,x_1,\ldots)$ where ; denotes the separation between
the positions $-1$ and 0.
Although $\wh M$ does depend on $f$, we will not write this dependence. Endow $\wh M$ with the distance
$\wh d(\wh x,\wh y):=\sup\{2^nd(x_n,y_n):n\leq 0\}$; then $\wh M$ is
a metric space. As for TMS, 
the definition of $\widehat d$ is not canonical and affects the
H\"older regularity of $\pi$ in Theorem \ref{Thm-ALP}.
For each $n\in\Z$, let $\vartheta_n:\wh M\to M$ be the projection into the $n$--th
coordinate, $\vartheta_n[\wh x]=x_n$. Consider the 
sigma-algebra in $\wh M$ generated by $\{\vartheta_n:n\leq 0\}$, i.e.
the smallest sigma-algebra that makes all $\vartheta_n$, $n\leq 0$, measurable.
We write $\vt=\vt_0$.

\medskip
\noindent
{\sc Natural extension of $f$:} The {\em natural extension} of $f$ is the 
map $\wh f:\wh M\to\wh M$ defined by
$\wh f(\ldots,x_{-1};x_0,\ldots)=(\ldots,x_0;f(x_0),\ldots)$.
It is an invertible map, with inverse $\wh f^{-1}(\ldots,x_{-1};x_0,\ldots)=(\ldots,x_{-2};x_{-1},\ldots)$.

\begin{remark}
Firstly, observe that if $f:M\to M$ is not continuous then $\hM$ is not compact, even when $M$ is. 
Secondly, we will work inside the subset $\hM\setminus\bigcup_{n\in\Z}{\wh f}^n(\vt^{-1}[\mathfs S])$, see Section~\ref{ss.proof.ALP}. Inside this set, the trajectories of the points are always outside the 
singular set and $df_{\vt[\wh x]}$ is an isomorphism. 
\end{remark}

There is a bijection between $f$--invariant and $\wh f$--invariant probability measures,
as follows.

\medskip
\noindent
{\sc Projection of a measure:} If $\wh\mu$ is an $\wh f$--invariant probability measure, then
$\mu=\wh\mu\circ \vt^{-1}$ is an $f$--invariant probability measure.

\medskip
\noindent
{\sc Lift of a measure:} If $\mu$ is an $f$--invariant probability measure,
let $\wh\mu$ be the unique probability measure on $\wh M$ s.t.
$\wh\mu[\{\wh x\in\wh M:x_n\in A\}]=\mu[A]$ for all $A\subset M$ Borel and all $n\leq 0$.

\medskip
The projection and lift procedures are inverse operations, and they preserve ergodicity 
and the Kolmogorov-Sina{\u\i} entropy, see \cite{Rohlin-Exactness}.

Let $N=\bigsqcup_{x\in M}N_x$ be a vector bundle over $M$, and $A:N\to N$ be measurable
s.t. for every $x\in M\setminus \mathfs S$ the restriction $A|_{N_x}$ is a linear
isomorphism $A_x:N_x\to N_{f(x)}$. The map $A$ defines a (possibly non-invertible)
cocycle $(A^{(n)})_{n\geq 0}$ over $f$ by $A^{(n)}_x=A_{f^{n-1}(x)}\cdots A_{f(x)}A_x$ 
for $x\in M\setminus\bigcup_{k\geq 0}f^{-k}(\mathfs S)$, $n\geq 0$.
There is a way of extending $(A^{(n)})_{n\geq 0}$ to an invertible cocycle over $\wh f$.
For $\wh x\in\wh M\setminus \bigcup_{n\in\Z}\wh f^{n}(\vt^{-1}[\mathfs S])$, let $N_{\wh x}:=N_{\vt[\wh x]}$ and let
$\wh N:=\bigsqcup_{\wh M\setminus \bigcup_{n\in\Z}\wh f^{n}(\vt^{-1}[\mathfs S])} N_{\wh x}$, 
a vector bundle over $\wh M\setminus \bigcup_{n\in\Z}\wh f^{n}(\vt^{-1}[\mathfs S])$.
Define the map $\wh A:\wh N\to\wh N$, $\wh A_{\wh x}:=A_{\vt[\wh x]}$.
For $\wh x=(x_n)_{n\in\Z}$ in $\wh M\setminus \bigcup_{n\in\Z}\wh f^{n}(\vt^{-1}[\mathfs S])$, define 
$$
\wh A^{(n)}_{\wh x}:=
\left\{
\begin{array}{ll}
A^{(n)}_{x_0}&,\text{ if }n\geq 0\\
A_{x_{-n}}^{-1}\cdots A_{x_{-2}}^{-1}A_{x_{-1}}^{-1}&,\text{ if }n\leq 0.
\end{array}
\right.
$$
By definition,
$\wh A^{(m+n)}_{\wh x}=\wh A^{(m)}_{\wh f^n(\wh x)}\wh A^{(n)}_{\wh x}$
for all $m,n\in\Z$, hence $(\wh A^{(n)})_{n\in\Z}$
is an invertible cocycle over $\wh f$.

We will use $\wh{TM}$ for the fiber bunde over $\wh M$ induced by $TM$ and $\wh{df}$ for the cocycle induced by the derivative cocycle $df$.

\section{Global symbolic dynamics}\label{Section-ALP}

In this section, we state and prove Theorem \ref{Thm-ALP}, which is a version 
of \cite[Theorem 3.5]{BCS-Annals} for non-invertible maps with singularities.
Theorem \ref{Thm-ALP} is a stronger result than the Main Theorem of \cite{ALP-23},
as we obtain various properties of the coding, notably condition ($\wh{\rm C}$9).
To maintain the same generality of \cite{ALP-23}, we require that the map satisfies
some properties, called (A1)--(A7), which
control the geometry and dynamics near the singular set. After stating the first theorem below,
we recall (A1)--(A7) following \cite{ALP-23} and then prove Theorem \ref{Thm-ALP}. 
Below, we use the same labeling of \cite{BCS-Annals}, which lists
the properties of the coding by (C1) to (C9).

\begin{theorem}\label{Thm-ALP}
Let $M$ be a smooth Riemannian manifold with finite diameter, $f$ a map on $M$,
and assume that $M,f$ satisfy assumptions {\rm (A1)--(A7)}. For all $\chi>0$, there is
a locally compact countable topological Markov shift $(\wh\Sigma,\wh\sigma)$ and
a H\"older continuous map $\wh \pi:\wh\Sigma\to\wh M$ such that $\wh \pi\circ\wh \sigma=\wh f\circ\wh\pi$ and:
\begin{enumerate}[{\rm (C1)}]
\item[{\rm (C1)}] The restriction $\wh\pi\restriction_{\whS^\#}$ is finite-to-one: 
if $\un v\in\whS^\#$ with $v_n=v$ for infinitely many $n>0$ and $v_n=w$ for infinitely many $n<0$,
then $\#\{\un w\in \whS^\#:\wh\pi(\un w)=\pi(\un v)\}$ is bounded by a constant $N(v,w)$.
\item[{\rm ($\wh{\rm C}$2)}] 
\begin{enumerate}[{\rm (a)}]
\item If $\mu$ is adapted and $\chi$--hyperbolic then $\wh\mu[\wh\pi(\whS^\#)]=1$, and
there exists $\nu$ a $\whs$--invariant probability measure on $\whS$
such that $\wh\mu=\nu\circ\wh\pi^{-1}$ and $h_{\nu}(\whs)=h_\mu(f)$.
\item If $\nu$ is $\whs$--invariant probability measure on $\whS$ then
$\wh\mu=\nu\circ\wh\pi^{-1}$ is hyperbolic with $h_{\wh\mu}(\wh f)=h_{\nu}(\whs)$.
\end{enumerate}
\item[{\rm (C3)}] For every $\wh x\in\wh\pi(\whS)$ there is a splitting 
$\wh{TM}_{\wh x}=E^s_{\wh x}\oplus E^u_{\wh x}$ such that:
\begin{enumerate}[i.]
\item[{\rm (i)}] $\limsup\limits_{n\to+\infty}\tfrac{1}{n}\log\|\wh{df}^{(n)}|_{E^s_{\wh x}}\|\leq -\frac{\chi}{2}$;
\item[{\rm (ii)}] $\limsup\limits_{n\to+\infty}\tfrac{1}{n}\log\|\wh{df}^{(-n)}|_{E^u_{\wh x}}\|\leq -\frac{\chi}{2}$.
\end{enumerate}
The maps $\un v\in\whS\mapsto E^{s/u}_{\wh\pi(\un v)}$ are H\"older continuous.
\item[{\rm (C4)}] For every $\wh x\in\wh\pi(\whS)$ there are $C^1$ submanifolds 
$W^{s/u}_{\wh x}\subset M$ passing through $\vt[\wh x]$ such that:
\begin{enumerate}[i.]
\item[{\rm (i)}] $T_{\vt[\wh x]}W^{s}_{\wh x}=E^s_{\wh x}$ and $d(f^n(y),f^n(z))\leq e^{-\frac{\chi}{2}n}$
for all $y,z\in W^{s}_{\wh x}$ and $n\geq 0$;
\item[{\rm (ii)}] $T_{\vt[\wh x]}W^{u}_{\wh x}=E^u_{\wh x}$ and 
$d(f^{-1}_{x_{-n}}\circ \cdots \circ f^{-1}_{x_{-1}}(y),f^{-1}_{x_{-n}}\circ \cdots \circ f^{-1}_{x_{-1}}(z))\leq e^{-\frac{\chi}{2}n}$ for all $y,z\in W^{u}_{\wh x}$ and $n\geq 0$.
\end{enumerate}
\item[{\rm (C5)}] {\sc Bowen property:} There is a symmetric binary relation $\sim$ on the alphabet 
$V$ of $\whS$ such that:
\begin{enumerate}[i.]
\item[{\rm (i)}] $\sim$ is locally finite: for every $v\in V$, it holds $\#\{w\in V:w\sim v\}<\infty$;
\item[{\rm (ii)}] If $\un v,\un w\in\whS^\#$, then $\wh\pi(\un v)=\wh\pi(\un w)$ if and only if
$v_n\sim w_n$ for all $n\in\Z$.
\end{enumerate}
\item[{\rm (C6)}] If $\nu$ is $\wh\sigma$--invariant, then the projection $\wh\mu=\nu\circ\wh\pi^{-1}$
is a $\chi/3$--hyperbolic measure. 
\item[{\rm (C7)}] For every $\chi'>0$, the set of ergodic $\sigma$--invariant measures $\nu$ 
such that $\nu\circ\wh\pi^{-1}$ is $\chi'$--hyperbolic is open in the relative weak--$^*$ topology of 
$\mathbb P_{\rm e}(\whS)$.
\item[{\rm (C8)}] For any relatively compact sequence $\un v^1,\un v^2,\ldots\in\whS^\#$, if
$\un w^1,\un w^2,\ldots\in\whS^\#$ satisfies $\wh\pi(\un v^i)=\wh\pi(\un w^i)$ for all $i\geq 1$
then $\un w^1,\un w^2,\ldots$ is also relatively compact.
\item[{\rm ($\wh{\rm C}$9)}] If $K\subset\wh M$ is a transitive $\wh f$--invariant compact $\chi$--hyperbolic set,
then there is a {\em transitive} $\wh\sigma$--invariant compact set $X\subset\whS$ such that $\wh\pi(X)=K$.
\end{enumerate}
\end{theorem}

Above, $\mathbb P_{\rm e}(\whS)$ is the set of $\whs$--invariant ergodic probability measures. 
Since the work of Sarig \cite{Sarig-JAMS}, there has been intense development on
the construction of codings for non-uniformly hyperbolic systems.
We cite the work of 
Ben Ovadia for diffeomorphisms in any dimension \cite{Ben-Ovadia-2019},
Lima and Matheus for two dimensional non-uniformly hyperbolic billiards \cite{Lima-Matheus},
Lima and Sarig for three dimensional flows without fixed points \cite{Lima-Sarig},
Lima for one-dimensional maps \cite{Lima-IHP},
Araujo, Lima and Poletti for non-invertible maps with singularities in any dimension \cite{ALP-23},
and more recently Buzzi, Crovisier and Lima for three dimensional flows without fixed points \cite{BCL},
improving the result of \cite{Lima-Sarig}.

\begin{remark}
Properties {\rm ($\wh{\rm C}$2)}(b) and (C6) are essentially the same.
The reason we use this notation is to maintain the analogy with the notation of \cite{BCS-Annals},
since we believe it eases the readability and comparison with \cite{BCS-Annals}.

\end{remark}

\subsection{Assumptions (A1)--(A7)}\label{subsection-a1a7}

Let $M$ be a smooth Riemannian manifold with finite diameter.
We fix a closed set $\mathfs D\subset M$, which will denote the set
of {\em discontinuities} of the map $f$. Given $x\in M$, let $T_xM$ denote
the tangent space of $M$ at $x$. For $r>0$, let $B_x[r]\subset T_xM$ denote the open ball with center 0
and radius $r$. Given $x\in M\backslash\mathfs D$, let $\inj(x)$ be the {\em injectivity radius} of $M$ at $x$,
and let $\exp{x}:B_x[\inj(x)]\to M$ be the {\em exponential map} at $x$.

Denote the Sasaki metric on $TM$ by $\Sas(\cdot,\cdot)$. When there is no confusion, we
denote the Sasaki metric on $TB_x[r]$ by the same notation. For $x\in M$ and $r>0$,
let $B(x,r)\subset M$ denote the open ball with center $x$ and radius $r$.
The first two assumptions on $M,f$ are about the exponential maps.

\medskip
\noindent
{\sc Regularity of $\exp{x}$:} $\exists a>1$ s.t. for all
$x\in M\backslash\mathfs D$ there is $d(x,\mathfs D)^a<\mathfrak d(x)<1$
s.t. for $\mathfrak B_x:=B(x,2\mathfrak d(x))$ the following holds:
\begin{enumerate}[ii]
\item[(A1)] If $y\in \mathfrak B_x$ then $\inj(y)\geq 2\mathfrak d(x)$, $\exp{y}^{-1}:\mathfrak B_x\to T_yM$
is a diffeomorphism onto its image, and
$\tfrac{1}{2}(d(x,y)+\|v-P_{y,x}w\|)\leq \Sas(v,w)\leq 2(d(x,y)+\|v-P_{y,x} w\|)$ for all $y\in \mathfrak B_x$ and
$v\in T_xM,w\in T_yM$ s.t. $\|v\|,\|w\|\leq 2\mathfrak d(x)$, where 	
$P_{y,x}:=P_\gamma$ is the parallel transport along the length minimizing geodesic $\gamma$ joining $y$ to $x$.
\item[(A2)] If $y_1,y_2\in \mathfrak B_x$ then
$d(\exp{y_1}v_1,\exp{y_2}v_2)\leq 2\Sas(v_1,v_2)$ for $\|v_1\|$, $\|v_2\|\leq 2\mathfrak d(x)$,
and $\Sas(\exp{y_1}^{-1}z_1,\exp{y_2}^{-1}z_2)\leq 2[d(y_1,y_2)+d(z_1,z_2)]$
for $z_1,z_2\in \mathfrak B_x$ whenever the expression makes sense.
In particular $\|d(\exp{x})_v\|\leq 2$ for $\|v\|\leq 2\mathfrak d(x)$,
and $\|d(\exp{x}^{-1})_y\|\leq 2$ for $y\in \mathfrak B_x$.
\end{enumerate}

\medskip
The next two assumptions are on the regularity of the derivative $d\exp{x}$.
For $x,x'\in\ M\backslash\mathfs D$, let $\mathfs L _{x,x'}:=\{A:T_xM\to T_{x'}M:A\text{ is linear}\}$
and $\mathfs L _x:=\mathfs L_{x,x}$. 
Given $y\in \mathfrak B_x,z\in \mathfrak B_{x'}$ and $A\in \mathfs L_{y,z}$,
let $\widetilde{A}\in\mathfs L_{x,x'}$, $\widetilde{A}:=P_{z,x'} \circ A\circ P_{x,y}$.
The norm $\|\widetilde{A}\|$ does not depend on the choice of $x,x'$.
If $A_i\in\mathfs L_{y_i,z_i}$ then $\|\widetilde{A_1}-\widetilde{A_2}\|$ does
depend on the choice of $x,x'$, but if we change the basepoints $x,x'$ to $w,w'$ then
the respective differences differ by precompositions and postocompositions
with norm of the order of the areas of the geodesic triangles formed by $x,w,y_i$
and by $x',w',z_i$, which will be negligible to our estimates.
Define the map $\tau=\tau_x:\mathfrak B_x\times \mathfrak B_x\to \mathfs L_x$
by $\tau(y,z)=\widetilde{d(\exp{y}^{-1})_z}$, where we use the identification
$T_v(T_{y}M)\cong T_{y}M$ for all $v\in T_yM$.

\medskip
\noindent
{\sc Regularity of $d\exp{x}$:}
\begin{enumerate}[ii]
\item[(A3)] If $y_1,y_2\in \mathfrak B_x$ then
$
\|\widetilde{d(\exp{y_1})_{v_1}}-\widetilde{d(\exp{y_2})_{v_2}}\|
\leq d(x,\mathfs D)^{-a}\Sas(v_1,v_2)
$
for all $\|v_1\|,\|v_2\|\leq 2\mathfrak d(x)$, and 
$\|\tau(y_1,z_1)-\tau(y_2,z_2)\|\leq d(x,\mathfs D)^{-a}[d(y_1,y_2)+d(z_1,z_2)]$
for all $z_1,z_2\in \mathfrak B_x$.
\item[(A4)] If $y_1,y_2\in \mathfrak B_x$ then the map 
$\tau(y_1,\cdot)-\tau(y_2,\cdot):\mathfrak B_x\to \mathfs L_x$
has Lipschitz constant $\leq d(x,\mathfs D)^{-a}d(y_1,y_2)$.
\end{enumerate}

\medskip
Now we discuss the assumptions on the map. Consider a map
$f:M\backslash\mathfs D\to M$. 
We assume that $f$ is differentiable at every point $x\in M\backslash\mathfs D$, and we let
$\mathfs C=\{x\in M\backslash\mathfs D: df_x\text{ is not invertible}\}$ be the {\em critical set} of $f$.
We assume that $\mathfs C$ is a closed subset of $M$.

\medskip
\noindent
{\sc Singular set $\mathfs S$:} The {\em singular set of $f$} is
$\mathfs S:=\mathfs C\cup \mathfs D$. 

\medskip
The singular set $\mathfs S$ is closed. We assume that $f$ satisfies the following properties.

\medskip
\noindent
{\sc Regularity of $f$:} \label{regularity-f}
$\exists \beta>0,\mathfrak K>1$ s.t. for all $x\in M$ with 
$x,f(x)\notin\mathfs S$ there is $\min\{d(x,\mathfs S)^a,$ $d(f(x),\mathfs S)^a\}<\mathfrak r(x)<1$
s.t. for $D_x:=B(x,2\mathfrak r(x))$ and $E_x:=B(f(x),2\mathfrak r(x))$ the following holds:
\begin{enumerate}[.......]
\item[(A5)] The restriction of $f$ to $D_x$ is a diffeomorphism onto its image;
the inverse branch of $f$ taking $f(x)$ to $x$ is a well-defined diffeomorphism from 
$E_x$ onto its image.
%
\item[(A6)] For all $y\in D_x$ it holds $d(x,\mathfs S)^a\leq \|df_y\|\leq d(x,\mathfs S)^{-a}$; for
all $z\in E_x$ it holds $d(x,\mathfs S)^a\leq \|dg_z\|\leq d(x,\mathfs S)^{-a}$,
where $g$ is the inverse branch of $f$ taking $f(x)$ to $x$.
\item[(A7)] For all $y,z\in D_x$ it holds $\|\wt{df_y}-\wt{df_z}\|\leq\mathfrak Kd(y,z)^\beta$;
for all $y,z\in E_x$ it holds $\|\wt{dg_y}-\wt{dg_z}\|\leq\mathfrak Kd(y,z)^\beta$,
where $g$ is the inverse branch of $f$ taking $f(x)$ to $x$.
\end{enumerate}

\medskip
Although technical, conditions (A5)--(A7) hold in most cases of interest, e.g.
if $\|df^{\pm 1}\|,\|d^2f^{\pm 1}\|$ grow at most polynomially fast with respect to
the distance to $\mathfs S$.
In the sequel, we let $f_x^{-1}:E_x\to f_x^{-1}(E_x)$ denote the inverse branch of $f$
taking $f(x)$ to $x$. This notation is used in property (C4) of Theorem \ref{Thm-ALP}.

Finally, we define the measures that we are able to code.

\medskip
\noindent
{\sc Adapted measure:} An $f$--invariant probability
measure $\mu$ on $M$ is called {\em adapted}
if $\log d(x,\mathfs S)\in L^1(\mu)$. A fortiori, $\mu(\mathfs S)=0$.

\medskip
Due to assumption (A6), if $\mu$ is $f$--adapted then the conditions of the
non-invertible version of the Oseledets theorem are satisfied. Therefore, if $\mu$ is $f$--adapted
then for $\mu$--a.e. $x\in M$ the Lyapunov exponent
$\chi(x,v)=\lim_{n\to+\infty}\tfrac{1}{n}\log\|df^n_x(v)\|$ exists for all $v\in T_xM$.
Among the adapted measures, we consider the hyperbolic ones. Fix $\chi>0$.

\medskip
\noindent
{\sc $\chi$--hyperbolic:} An $f$--invariant probability measure $\mu$ is called
{\em $\chi$--hyperbolic} if for $\mu$--a.e. $x\in M$ we have $|\chi(x,v)|>\chi$
for all $v\in T_xM$.

\medskip
In particular, an ergodic measure is hyperbolic in the classical sense if and only if it is 
$\chi$--hyperbolic for some $\chi>0$.
We can similarly define adaptedness and $\chi$--hyperbolicity
for a $\wh f$--invariant probability measure $\wh\mu$. Via the projection/lift bijection
explained in Section \ref{Section-natural-extension}, it is clear that 
$\wh\mu$ is adapted/$\chi$--hyperbolic if and only if $\mu$ is adapted/$\chi$--hyperbolic.
Note that in the statement of Theorem \ref{Thm-ALP} we use these notions both
for $f$ and $\wh f$.

\subsection{Proof of Theorem \ref{Thm-ALP}}\label{ss.proof.ALP}

In order to prove Theorem \ref{Thm-ALP}, we recast the main objects used in the proof
of \cite[Main Theorem]{ALP-23}. Define
$$
\hSing= \bigcup_{n\in\Z}{\wh f}^n(\vt^{-1}[\mathfs S]).
$$
Applying the construction in Section \ref{Section-natural-extension},
let $(\wh{df}^{(n)}_{\wh x})_{n\in\Z}$ be an invertible cocycle defined for
$\wh x=(x_n)_{n\in\Z}\in \wh M\backslash \hSing$ by
$$
\wh{df}^{(n)}_{\wh x}:=
\left\{
\begin{array}{ll}
df^{n}_{x_0}&,\text{ if }n\geq 0\\
(df_{x_{-n}})^{-1}\cdots (df_{x_{-2}})^{-1}(df_{x_{-1}})^{-1}&,\text{ if }n\leq 0.
\end{array}
\right.
$$
As in \cite{BCS-Annals}, we divide the discussion into three steps.

\medskip
\noindent
{\em Step 1: The non-uniformly hyperbolic locus $\nuh^\#$.}

\medskip
We first define a ``weak'' nonuniformly hyperbolic locus ${\rm NUH}$ as follows.

\medskip
\noindent
{\sc The set ${\rm NUH}$:} It is defined as the set of points \label{Def-NUH}
$\wh x\in \wh M\backslash \hSing$ for which there is a splitting
$\wh{TM}_{\wh x}=E^s_{\wh x}\oplus E^u_{\wh x}$ such that:
\begin{enumerate}[(NUH1)]
\item Every $v\in E^s_{\wh x}\backslash\{0\}$ contracts in the future at least $-\chi$ and expands in the past:
$$\limsup_{n\to+\infty}\tfrac{1}{n}\log \|\wh{df}^{(n)} v\|\leq -\chi\ \text{ and } 
\ \liminf_{n\to+\infty}\tfrac{1}{n}\log \|\wh{df}^{(-n)} v\|>0.
$$ 
\item Every $v\in E^u_{\wh x}\backslash\{0\}$ contracts in the past at least $-\chi$ and expands in the future:
$$\limsup_{n\to+\infty}\tfrac{1}{n}\log \|\wh{df}^{(-n)} v\|\leq -\chi\text{ and } 
\liminf_{n\to+\infty}\tfrac{1}{n}\log \|\wh{df}^{(n)}v\|>0.$$
\item \label{Def-NUH3} The parameters $s(\wh x)=\sup\limits_{v\in E^s_{\wh x}\atop{\|v\|=1}}S(\wh x,v)$
and $u(\wh x)=\sup\limits_{w\in E^u_{\wh x}\atop{\|w\|=1}}U(\wh x,w)$ are finite, where:
\begin{align*}
S(\wh x,v)&=\sqrt{2}\left(\sum_{n\geq 0}e^{2n\chi}\|\wh{df}^{(n)} v\|^2\right)^{1/2},\\
U(\wh x,w)&=\sqrt{2}\left(\sum_{n\geq 0}e^{2n\chi}\|\wh{df}^{(-n)}w\|^2\right)^{1/2}.
\end{align*}
\end{enumerate}

For each $\wh x\in{\rm NUH}$, one defines a parameter $Q(\wh x)$ in terms
of the values $S(\wh x,\cdot)$, $U(\wh x,\cdot)$ and
$d(\vt_n[\wh x],\mathfs S)$, $n\in\Z$, see \cite[Section~3.2]{ALP-23} for the precise definition. 
The parameter $Q(\wh x)$ gives the size of a neighborhood of $\vt[\wh x]$ in which the map
$f$ can be represented, in a new system of coordinates (Pesin charts),
as a small peturbation of a hyperbolic matrix.
Let $\delta_\ve:=e^{-\ve n}<\ve$ for some $n>0$.

\medskip
\noindent
{\sc Parameter $q(\wh x)$:} For $\wh x\in{\rm NUH}$, let
$q(\wh x):=\delta_\ve\min\{e^{\ve|n|}Q(\wh f^n(\wh x)):n\in\Z\}$.

\medskip
\noindent
{\sc The non-uniformly hyperbolic locus ${\rm NUH}^\#$:} It is the set of $\wh x\in{\rm NUH}$ such
that $q(\wh x)>0$ and
$$
\limsup_{n\to+\infty}q(\wh f^n(\wh x))>0\text{ and }\limsup_{n\to-\infty}q(\wh f^n(\wh x))>0.
$$

\noindent
{\em Step 2: A first coding $\pi:\Sigma\to\wh M$.}

\medskip
Introduce {\em $\ve$--double charts} $\Psi_{\wh x}^{p^s,p^u}$, which consist
of a pair of Pesin charts both centered at $\wh x$ but with different sizes $p^s$ and $p^u$.
Let $v=\Psi_{\wh x}^{p^s,p^u}$ and $w=\Psi_{\wh y}^{q^s,q^u}$ be $\ve$--double charts.
Draw an edge $v\overset{\ve}{\rightarrow}w$ when some nearest neighbor conditions are satisfied.
These conditions, called (GPO1) and (GPO2) in \cite{ALP-23},
allow to define
a {\em stable graph transform} from graphs near $\wh y$ with size $q^s$ that are almost 
parallel to $E^s_{\wh y}$ to graphs near $\wh x$ with size $p^s$ that are almost parallel to $E^s_{\wh x}$;
and a {\em unstable graph transform} from graphs near $\wh y$ with size $p^u$
that are almost parallel to $E^u_{\wh x}$ to graphs near $\wh y$ with size $q^u$
that are almost parallel to $E^u_{\wh y}$. This allows to associate to each
sequence $\un v=(v_n)$ with $v_n\overset{\ve}{\rightarrow}v_{n+1}$ for every $n\in\Z$
a point $\wh x\in\wh M$ which is the unique point shadowed by $\un v$.

Construct a countable family $\mathfs A$ of $\ve$--double charts such that
for all $t>0$, the set $\{\Psi_{\wh x}^{p^s,p^u}\in\mathfs A:p^s,p^u>t\}$ is finite,
and every $\wh x\in {\rm NUH}^\#$ is shadowed by some regular sequence 
$\un v=\{\Psi_{\wh x_n}^{p^s_n,p^u_n}\}\in{\mathfs A}^{\Z}$ with $p^s_n\wedge p^u_n\approx q(\wh f^n(\wh x))$.
Let $\Sigma$ be the TMS defined by the graph with vertices $V=\mathfs A$ and edges 
$E=\{v\overset{\ve}{\rightarrow}w:v,w\in\mathfs A\}$.

\medskip
\noindent
{\sc First coding $\pi:\Sigma\to\wh M$:} The map $\pi:\Sigma\to\wh M$ where
$\pi(\un v)$ is the unique $\wh x\in\wh M$ shadowed by $\un v$.

\medskip
We have $\pi[\Sigma^\#]\supset\nuh^\#$, thus we get a cover of 
$\mathfs Z=\{Z(v):v\in V\}$ of $\nuh^\#$ where $Z(v)=\{\un v\in\Sigma^\#:v_0=v\}$.
This cover is locally finite: for every $Z\in \mathfs Z$, it holds $\#\{Z'\in\mathfs Z:Z\cap Z'\neq\emptyset\}<\infty$.

\medskip
\noindent
{\em Step 3: The second coding $\wh\pi:\whS\to\wh M$.}

\medskip
Applying a Bowen-Sina{\u\i} refinement to $\mathfs Z$, obtain a Markov partition $\mathfs R$
of $\nuh^\#$ that is locally finite with respect to $\mathfs Z$: for every $R\in\mathfs R$,
$\#{Z\in\mathfs Z: Z\supset R}<\infty$; for every $Z\in\mathfs Z$, $\#\{R\in\mathfs R:R\subset Z\}<\infty$.
Let $\whS$ be the TMS with vertices $\wh V=\mathfs R$ and edges 
$\wh E=\{R\to S:R,S\in \mathfs R\text{ such that }\wh f(R)\cap S\neq\emptyset\}$.

\medskip
\noindent
{\sc Second coding $\wh\pi:\whS\to\wh M$:} The map $\wh\pi:\whS\to\wh M$ where
$$
\wh\pi(\un R)=\bigcap_{n\geq 0}\overline{\wh f^n(R_{-n})\cap\cdots\cap f^{-n}(R_n)}.
$$

\medskip
Now we show how to obtain the properties listed in Theorem \ref{Thm-ALP}.

\medskip
\noindent
{\bf Property (C1) and (C5).} These conditions are proved in \cite{ALP-23}.
The symmetric binary relation is called {\em affiliation}, first
introduced in \cite{Sarig-JAMS}: call $R,S\in\wh V$ affiliated and write $R\sim S$
when there are $Z,Z'\in\mathfs Z$ such that $R\subset Z$, $S\subset Z'$ and $Z\cap Z'\neq\emptyset$.
For $R\in\mathfs R$, define
$N(R):=\{(S,w):\mathfs R\times \mathfs A: R\sim S\text{ and }Z(w)\supset S\}$, a finite number
by the local finiteness.
Then property (C1) is \cite[Theorem 7.6(3)]{ALP-23} with $N(R,S)=N(R)N(S)$, and property
(C5) is \cite[Lemma 7.5]{ALP-23}. 

\medskip
\noindent
{\bf Property ($\wh{\textrm C}$2).} By \cite[Main Theorem]{ALP-23}, we have
$\wh\pi(\whS)=\nuh^\#$. Let $\mu$ be adapted and $\chi$--hyperbolic.
By \cite[Lemma 3.6]{ALP-23}), we have $\wh\mu[\wh\pi(\whS^\#)]=\wh\mu[\nuh^\#]=1$.
Now, using (C1), we can lift $\wh\mu$ to 
$$
\nu=\int_{\wh M}\left(\frac{1}{\#(\wh\pi^{-1}(\wh x)\cap\whS^\#)}\sum_{\un R\in\wh\pi^{-1}(\wh x)\cap\whS^\#}\delta_{\un R}\right)d\wh\mu(\wh x),
$$ 
which satisfies part (a).

Now we prove part (b). Let $\nu$ be $\whs$--invariant, then $\nu(\whS^\#)=1$.
Again by \cite[Main Theorem]{ALP-23}, it follows that $\wh\mu=\nu\circ \wh\pi^{-1}$ 
satisfies $\wh\mu[\nuh^\#]=1$ and so it is hyperbolic. Finally, $h_{\wh f}(\wh\mu)=h_\nu(\whs)$
by the Abramov-Rohklin formula, since $\wh \pi\restriction_{\whS^\#}:\whS^\#\to \nuh^\#$ is finite-to-one. 

\medskip
\noindent
{\bf Property (C3).} Parts (i) and (ii) are proved in \cite[Prop. 4.11(1)]{ALP-23}.
The Hölder regularity of $E^{s/u}$ is \cite[Prop. 7.7]{ALP-23}.

\medskip
\noindent
{\bf Property (C4).} This is proved in \cite[Prop. 4.9(4)]{ALP-23}. 

\medskip
\noindent
{\bf Property (C6).} As stated in the proof of ($\wh{\textrm C}$2) above, we have 
$\wh\mu[\nuh^\#]=1$ and so $\wh\mu$ is $\chi'$--hyperbolic for every $0<\chi'<\chi$.

\medskip
\noindent
{\bf Property (C7).} This is \cite[Prop. 3.7]{BCS-Annals}. Its proof only
requires that $\wh\pi$ is continuous with $\wh\pi\circ\whs=\wh f\circ\wh\pi$ and property (C3).

\medskip
\noindent
{\bf Property (C8).} This follows from the assumption and (C5), as proved
in \cite[Proposition 3.8]{BCS-Annals}.

\medskip
\noindent
{\bf Property ($\wh{\textbf C}$9).} The proof is an adaptation of 
\cite[Prop. 3.9]{BCS-Annals}. Let $K\subset\wh M$ be transitive $\wh f$--invariant
compact and $\chi$--hyperbolic. 

\medskip
\noindent
{\sc Step 1:} There is $X_0\subset \wh\Sigma$ compact such that $\wh\pi(X_0)\supset K$.

\begin{proof}[Proof of Step $1$]
Each $\wh x\in\nuh^\#$ has a canonical coding $\un R(\wh x)=\{R_n(\wh x)\}_{n\in\Z}$
where $R_n(\wh x)$ is the unique rectangle of $\mathfs R$ containing $\wh f^n(\wh x)$.
Since $K$ is compact and $\chi$--hyperbolic, $\inf_{\wh x\in K}q(x)>0$
and so $K$ intersects finitely many rectangles of $\mathfs R$. Hence there is a finite
set $V_0\subset\mathfs R$ such that $R_0(\wh x)\in V_0$ for all $\wh x\in K$.
By invariance, the same holds for all $n\in\Z$, i.e. 
$R_n(\wh x)\in V_0$ for all $\wh x\in K$. Therefore 
the subshift $X_0$ induced by $V_0$, which is compact since $V_0$ is finite,
satisfies $\wh\pi(X_0)\supset K$. 
\end{proof}

\medskip
\noindent
{\sc Step 2:} There is $X\subset X_0$ transitive such that $\wh\pi(X)=K$. 

\begin{proof}[Proof of Step $2$] Among all compact $X\subset X_0$ with $\wh\pi(X)\supset K$,
consider a minimal set for the inclusion, which exists by Zorn's lemma. We claim that such $X$ satisfies
Step 2. To see that, let $\wh x\in K$ whose forward orbit is dense in $K$, and let
$\un R\in X$ be a lift of $\wh x$. We claim that the forward orbit of $\un R$ is dense in $X$.
Indeed, since $\wh\pi$ is continuous, we have
$$
\wh\pi\left(\omega(\whs,\un R)\right)=\omega(\wh f,\wh x)=K.
$$
Since $\omega(\whs,\un R)\subset X$ and $X$ is minimal for the inclusion, it follows that $X=\omega(\whs,\un R)$,
so $X$ is transitive and the above equality gives that $\wh\pi(X)=K$.
\end{proof}

This concludes the proof of Theorem \ref{Thm-ALP}.

\section{Symbolic dynamics of homoclinic classes of measures}\label{Section-homoclinic}

In this section, we prove Theorem \ref{thm.criterion}. For that, we prove
a version of Theorem \ref{Thm-ALP} for homoclinic classes of measures, which is Theorem \ref{Thm-Main} below.  To state this theorem, we first need to introduce the notion of homoclinic relations for hyperbolic measures
and obtain general properties of this relation.

\subsection{Invariant manifolds and invariant sets}\label{ss.invariant.sets}

Following Section 4.5 of \cite {ALP-23}, to each $\wh x\in \nuh^\#$ one associates
a {\em local stable manifold} $W^s(\wh x)\subset M$ 
and a {\em local unstable manifold} $W^u(\wh x)\subset M$. \footnote{In \cite{ALP-23},
each $\ve$--generalized pseudo-orbit $\un v$ is associated to local stable/unstable manifolds
$V^{s/u}[\un v]$. This defines, in particular, the local stable/unstable manifolds of every
$\wh x\in\nuh^\#$.} These sets are constructed away from the singular set $\mathfs S$.
Also, $W^s(\wh x)=W^s(x_0)$
only depends on $x_0$, while $W^u(\wh y)$ is defined in terms of inverse branches.
Section 4.6 of \cite{ALP-23} also defines {\em local invariant sets}, which are subsets of $\wh M$,
as follows.

\medskip
\noindent
{\sc Local invariant sets $V^{s/u}(\un x)$:} The {\em local stable set} of $\wh x\in\nuh^\#$
is defined by 
$$
V^s(\wh x)=\{\wh y\in\wh M:y_0\in W^s(\wh x)\},
$$
and its {\em local unstable set} is defined by
$$
V^u(\wh x)=\{\wh y\in\wh M:y_0\in W^u(\wh x)\text{ and }y_{-n}=f^{-1}_{x_{-n}}(y_{-n+1})\text{ for all }n\geq 1\}.
$$

\medskip
Recall that $f^{-1}_{x_{-n}}$ is the inverse branch of $f$ such that $f^{-1}_{x_{-n}}(x_{-n+1})=x_{-n}$,
see its definition in page \pageref{regularity-f}.
Alternatively, by the invariance of $W^{s/u}$, see \cite[Prop. 4.7(2)]{ALP-23}, we have 
\begin{align*}
V^s(\wh x)&=\{\wh y\in\wh M:y_n\in W^s(\wh f^n(\wh x)),\forall n\geq 0\}=
\{\wh y\in\wh M:y_n\in W^s(x_n),\forall n\geq 0\}\\
V^u(\wh x)&=\{\wh y\in\wh M:y_n\in W^u(\wh f^n(\wh x)),\forall n\leq 0\}.
\end{align*}
We introduce global versions of $V^{s/u}(\wh x)$, following the analogy of 
diffeomorphisms.

\medskip
\noindent
{\sc Global invariant sets $\mathfs V^{s/u}(\wh x)$:} The {\em global stable set} of $\wh x\in\nuh^\#$ is defined by
$$
\mathfs V^s(\wh x)=\bigcup_{n\geq 0}\wh f^{-n}[V^s(\wh f^n(\wh x))]=\{\wh y\in\wh M: \exists n\geq 0\text{ s.t. }
y_n\in W^s(x_n)\},
$$
and its {\em global unstable set} is defined by
$$
\mathfs V^u(\wh x)=\bigcup_{n\geq 0}\wh f^{n}[V^u(\wh f^{-n}(\wh x))]=\{\wh y\in\wh M:\exists n\leq 0\text{ s.t. }
y_m\in W^u(\wh f^m(\wh x)),\forall m\leq n\}.
$$
It is clear from the alternative characterization of $V^{s/u}$ given above that the sets in the unions defining $\mathfs V^{s/u}(\wh x)$ form increasing families,
e.g. $\wh f^{-n}[V^s(\wh f^n(\wh x))]\subset \wh f^{-m}[V^s(\wh f^m(\wh x))]$ for all $n\leq m$.
The next lemma translates intersection between global invariant sets in terms of local 
invariant sets.

\begin{lemma}\label{lemma.intersection}
Let $\wh x,\wh y\in\nuh^\#$. Then $[\mathfs V^u(\wh x)\cap \mathfs V^s(\wh y)] \setminus \wh{\mathfs S}\neq\emptyset$
if and only if there are $m\leq 0\leq n$ and $z_m\in W^u(\wh f^m(\wh x))$ such that:
\begin{enumerate}[{\rm (1)}]
\item $f^{n-m}(z_m)\in f^{n-m}[W^u(\wh f^m(\wh x))]\cap W^s(\wh f^n(\wh y))$;
\item $f^j(z_m)\notin \mathfs S$ for all $0<j<n-m$.
\end{enumerate}
In particular, $\mathfs V^u(\wh x)$ intersects $\bigcup_{\ell\in\Z}\wh f^\ell(\mathfs V^s(\wh y))$
in a point not belonging to $\wh{\mathfs S}$ if and only if there are $m,n\in\mathbb Z$, $k\geq 0$ 
and $z\in W^u(\wh f^m(\wh x))$ such that 
$f^k(z)\in f^k[W^u(\wh f^m(\wh x))]\cap W^s(\wh f^n(\wh y))$ and $f^j(z)\notin\mathfs S$ for
$0<j<k$.
\end{lemma}

\begin{proof}
Start assuming that there is
$\wh z\in \mathfs V^u(\wh x)\cap \mathfs V^s(\wh y) \setminus \wh{\mathfs S}$.
By definition, there are $m\leq 0\leq n$
such that $\wh f^m(\wh z)\in V^u(\wh f^m(\wh x))$ and $\wh f^n(\wh z)\in V^s(\wh f^n(\wh y))$.
In particular, $z_m\in W^u((\wh f^m(\wh x))$ and $z_n\in W^s(\wh f^n(\wh y))$. Since 
$f^{n-m}(z_m)=z_n$, it follows that $z_n\in f^{n-m}[W^u(\wh f^m(\wh x))]\cap W^s(\wh f^n(\wh y))$,
thus proving (1). Also, since $\wh z\not\in \wh{\mathfs S}$, condition (2) holds.

Now assume that $m\leq 0\leq n$ and $z_m\in W^u(\wh f^m(\wh x))$ with $f^{n-m}(z_m)\in W^s(\wh f^n(\wh y))$
and $f^j(z_m)\notin \mathfs S$ for all $0<j<n-m$. 
Define $\wh z=(z_k)$ by 
$$
z_k=\left\{
\begin{array}{ll}
(f_{x_k}^{-1}\circ \cdots\circ f_{x_{m-1}}^{-1})(z_m)& \text{, if }k< m\\
&\\
f^{k-m}(z_m)& \text{, if }k\geq m.
\end{array}
\right.
$$
Observe that $z_k$ is well-defined:
\begin{enumerate}[$\circ$]
\item $z_m\in W^u(\hf^m(\hx))$, hence the composition
$(f_{x_k}^{-1}\circ \cdots\circ f_{x_{m-1}}^{-1})(z_m)$ is well-defined and does not belong to $\mathfs S$
for all $k<m$;
\item $z_k=f^{k-m}(z_m)\notin \mathfs S$ for all $m\leq k<n$ by hypothesis.
\item $z_n\in W^s(\hf^n(\hx))$, hence $z_k=f^{k-n}(z_n)$ is well-defined and does not belong to $\mathfs S$
for all $k\geq n$. 
\end{enumerate}
It is clear that $\wh z\in \mathfs V^u(\wh x)\cap \mathfs V^s(\wh y) \setminus \wh{\mathfs S}$.
This completes the proof.
\end{proof}

\medskip
\noindent
{\sc Transversality of global invariant sets:} We say that $\mathfs V^u(\wh x)$ and $\mathfs V^s(\wh y)$
are {\em transversal}, and write $\mathfs V^u(\wh x)\pitchfork\mathfs V^s(\wh y)\neq\emptyset$,
if there are $m\leq 0\leq n$ and $z_m\in W^u(f^m(\hx))$ such that $f^{n-m}(z_m)\in f^{n-m}[W^u(\wh f^m(\wh x))]\pitchfork W^s(\wh f^n(\wh y))$ and $f^j(z_m)\notin \mathfs{S}$ for all $0<j<n-m$.

\medskip
When this happens, for the element $\wh z$ given by Lemma \ref{lemma.intersection} we will write 
that $\wh z\in \mathfs V^u(\wh x)\pitchfork\mathfs V^s(\wh y)$.
Note that $\wh z$ belongs to $\hf^{-m}(V^u(\hf^m(\hx)))\cap \hf^{-n}(V^s(\hf^n(\hy)))$, 
so we also write $\mathfs V^u(\wh x)\pitchfork\mathfs V^s(\wh y)\neq\emptyset$ 
in terms of local invariant sets by
$\hf^{-m}(V^u(\hf^m(\hx)))\pitchfork \hf^{-n}(V^s(\hf^n(\hy)))\neq\emptyset$ for some (any) $m\leq 0\leq n$.

As usual, a periodic orbit $\mathcal O$ of period $n$ for $\hf$ is called {\em hyperbolic} if
$df^n_{\vt[\wh x]}=\wh{df}^{(n)}_{\wh x}$ is a hyperbolic matrix for $\wh x\in\mathcal O$.
In this case, we define
$\mathfs V^{s/u}(\mathcal O)=\bigcup_{\hx\in\mathcal O}\mathfs V^{s/u}(\hx)$ denote
the global stable/unstable set of $\mathcal O$. 

\medskip
\noindent
{\sc Homoclinic class of hyperbolic periodic orbit:} The {\em homoclinic class}
of a hyperbolic periodic orbit $\mathcal O$ is the set
$$
{\rm HC}(\mathcal O)=\overline{\mathfs V^u(\mathcal O)\pitchfork \mathfs V^s(\mathcal O)}.
$$

\subsection{Homoclinic relation of measures }\label{ss.hom.relation}

Let $\mathbb P_{\rm h}(\wh f)$ denote the set of $\wh f$--invariant ergodic hyperbolic probability
measures.

\medskip
\noindent
{\sc Partial order of measures:} Let $\mu_1,\mu_2\in\mathbb P_{\rm h}(\wh f)$.
We say that $\mu_1\preceq\mu_2$ if there are measurable sets $A_1,A_2\subset\wh M$
with $\mu_i(A_i)>0$ such that $\mathfs V^u(\wh x)\pitchfork \mathfs V^s(\wh y)\neq\emptyset$
for all $(\wh x,\wh y)\in A_1\times A_2$.
 
\medskip
Observe that $\mathfs V^{s/u}$ depends on the choice of $\chi$, but given
$\mu_1,\mu_2\in\mathbb P_{\rm h}(\wh f)$ we can always choose
$\chi$ small enough so that $\mu_1,\mu_2$ are both $\chi$--hyperbolic.

\medskip
\noindent
{\sc Homoclinic relation of measures:} Given $\mu_1,\mu_2\in\mathbb P_{\rm h}(\wh f)$,
we say that $\mu_1$ and $\mu_2$ are {\em homoclinically related}
if $\mu_1\preceq\mu_2$ and  $\mu_2\preceq\mu_1$. When this happens, we write $\mu_1 \stackrel{h}{\sim}\mu_2$.

\medskip
We also define homoclinic relation between a set and a measure.\\

\noindent
{\sc Homoclinic relation of a set and a measure:} Given a transitive set $K\subset \hM$ and
$\mu\in \mathbb P_h(\hf)$, we say that $K$ and $\mu$ are {\em homoclinically related} and write
$K \stackrel{h}{\sim}\mu$ when there exists $\nu\in \mathbb P_h(\hf|_K)$ such that $\nu \stackrel{h}{\sim}\mu$.

\medskip
Now we can state the version of Theorem \ref{Thm-ALP} for homoclinic classes of measures.

\begin{theorem}\label{Thm-Main}
Let $M$ be a smooth Riemannian manifold with finite diameter, $f$ a map on $M$,
and assume that $M,f$ satisfy assumptions {\rm (A1)--(A7)}. For every adapted, ergodic and hyperbolic
$f$--invariant probability measure $\mu$,
there is a locally compact countable topological Markov shift $(\Sigma,\sigma)$ and
a H\"older continuous map $\pi:\Sigma\to\wh M$ such that $\pi\circ\sigma=\wh f\circ\pi$,
satisfying properties {\em (C1)}, {\em (C3)--(C8)} and:
\begin{enumerate}[{\rm (C1)}]
\item[{\rm (C0)}] $\Sigma$ is irreducible.
\item[{\rm (C2)}] 
\begin{enumerate}[{\rm (a)}]
\item If $\nu$ is adapted, $\chi$--hyperbolic and homoclinically related to $\mu$
then $\wh\nu[\pi(\Sigma^\#)]=1$, and there exists $\eta$ a $\sigma$--invariant probability measure on $\Sigma$
such that $\wh\nu=\eta\circ\pi^{-1}$ and $h_{\eta}(\sigma)=h_\nu(f)$.
\item If $\eta$ is a $\sigma$--invariant probability measure on $\Sigma$ then
$\wh\nu=\eta\circ\pi^{-1}$ is hyperbolic and homoclinically related to $\wh\mu$ 
with $h_{\wh\nu}(\wh f)=h_{\eta}(\sigma)$.
\end{enumerate}
\item[{\rm (C9)}] If $K\subset\wh M$ is a transitive $\hf$--invariant compact $\chi$--hyperbolic set
that is homoclinically related to $\wh\mu$, then there is a transitive $\sigma$--invariant 
compact set $X\subset\Sigma$ such that $\pi(X)=K$.
\end{enumerate}
\end{theorem}

Theorem \ref{Thm-Main} is a non-invertible version of \cite[Theorem 3.14]{BCS-Annals}.
The topological Markov shift $(\Sigma,\sigma)$ depends on $\mu$ and on $\chi$.
We emphasize that it is irreducible, a property that is important for applications
(the topological Markov shift obtained in Theorem \ref{Thm-ALP}
might not be irreducible).

The proof of Theorem \ref{Thm-Main} requires obtaining some basic properties
of homoclinic relations between hyperbolic measures, as follows.

\begin{proposition}\label{prop.equivalence.relation}
Let $\mu_1,\mu_2,\mu_3\in \mathbb P_{\rm h}(\widehat f)$.
If $\mu_1\preceq \mu_2$ and $\mu_2\preceq \mu_3$, then $\mu_1\preceq \mu_3$.
\end{proposition}

The proof of the above proposition requires a version of the Inclination Lemma for points
with some recurrence. For diffeomorphisms, this was obtained in \cite[Lemma 2.7]{BCS-Annals}, 
where recurrence was stated in terms of Pesin blocks. We follow the same strategy:
we define Pesin blocks and then state the Inclination Lemma (Proposition \ref{lambda.u-lemma}),
whose proof is in Appendix \ref{appendix}.

Let $d$ be the dimension of $M$. Recall that $\chi,\ve>0$. Let $\ell\in\{0,1,\ldots,d\}$.

\medskip
\noindent
{\sc Pesin sets $\Lambda_{\chi,\ve,C,\ell}$ and $\Lambda_{\chi,\ve,C}$:} 
We denote $\Lambda_{\chi,\ve,C,\ell}$ as the set
of $\hx\in\hM\backslash\wh{\mathfs S}$ such that there is a $\wh{df}$--invariant decomposition 
$\wh{TM}_{\hf^k(\hx)}=E^s_{\hf^k(\hx)}\oplus E^u_{\hf^k(\hx)}$, $k\in\Z$, such that:
\begin{enumerate}[(PS1)]
\item[(PS1)] $\norm{\wh{df}^{(n)}|_{E^s_{\hf^k(\hx)}}}\leq Ce^{-n\chi+|k|\ve}$ and
$\norm{\wh{df}^{(-n)}|_{E^u_{\hf^k(\hx)}}}\leq Ce^{-n\chi+|k|\ve}$ for all $n\geq 0$ and $k\in\Z$.
\item[(PS2)] $\angle(E^s_{\hf^k(\hx)},E^u_{\hf^k(\hx)})\geq C^{-1}e^{|k|\ve}$ for all $k\in\Z$.
\item[(PS3)] $d(\vt_k[\hx],\mathfs S)\geq C^{-1}e^{|k|\ve}$ for all $k\in\Z$.
\item[(PS4)] ${\rm dim}(E^s_{\hf^k(\hx)})=\ell$ for all $k\in\Z$.
\end{enumerate}
We then define 
$$
\Lambda_{\chi,\ve,C}:=\bigcup_{\ell=0}^d \Lambda_{\chi,\ve,C,\ell}.
$$

\medskip
These sets were defined in \cite{Pesin-Izvestia-1976} for diffeomorphisms and in \cite{KSLP} for billiards.
Following these references, each $\Lambda_{\chi,\ve,C,\ell}$ is a compact set such that
the maps $\hx\in \Lambda_{\chi,\ve,C,\ell}\mapsto E^{s/u}_{\hx}$ are continuous. Furthermore,
there are local invariant submanifolds $W^{s/u}(\hx)$, $\hx\in\Lambda_{\chi,\ve,C,\ell}$,
which vary continuously. Hence, the same properties hold for $\Lambda_{\chi,\ve,C}$.

Another property of these sets is that for each $m\in \Z$, $\hf^m(\Lambda_{\chi,\ve,C}) \subset \Lambda_{\chi,\ve, Ce^{|m| \ve}}$. Fix a sequence $(\chi_n)$ that decreases to zero. For each $n$, choose $\ve_n>0$
small enough.


\medskip
\noindent
{\sc Pesin blocks $K_n$:} We define {\em Pesin blocks} $(K_n)$ by $K_n=\Lambda_{\chi_n,\ve_n,1/\chi_n}$.

\medskip
We actually choose $(\chi_n)$ converging to zero fast enough to assure that
$\hf^{-1}(K_n)\cup K_n\cup \hf(K_n)\subset K_{n+1}$. Therefore
$Y:=\bigcup K_n$ is $\hf$--invariant.
We finally define the following set.

\medskip
\noindent
{\sc The set $Y'$:} It is the set of $\hx\in Y$
for which there are sequences $n_k,m_k\to \infty$ such that $\hf^{n_k}(\hx),\hf^{-m_k}(\hx)$ both
belong to a same Pesin block and which converge to $\hx$.

\medskip
We are now ready to state the Inclination Lemma.

\begin{proposition}[Inclination Lemma]\label{lambda.u-lemma}
Let $\wh y\in Y'$, and let $\Delta\subset M$ be a disc of same dimension of $W^u(\hy)$.
If $\Delta$ is transverse to $W^s(\wh f^m(\wh y))$ for some $m\in\Z$, then 
there are discs $D_k\subset \Delta$ and $n_k\to \infty$ such that $f^{n_k}(D_k)$
converges to $W^u(\wh y)$ in the $C^1$ topology.
\end{proposition}

The proof is in Appendix \ref{appendix}.

\begin{corollary}\label{cor.trans}
Let $\wh x,\wh y,\wh z\in Y'$. If $\mathfs V^u(\wh x)\pitchfork \mathfs V^s(\wh y)\neq \emptyset$ and
$\mathfs V^u(\wh y)\pitchfork \mathfs V^s(\wh z)\neq \emptyset$, then 
there is $n\in\Z$ such that $\mathfs V^u(\wh x)\pitchfork \wh f^n(\mathfs V^s(\wh z))\neq\emptyset$.
\end{corollary}

\begin{proof}
Using that $\mathfs V^u(\wh x)\pitchfork \mathfs V^s(\wh y)\neq \emptyset$, Lemma \ref{lemma.intersection}
implies the existence of $m\leq 0\leq \ell$ and $w\in W^u(\hf^m(\hx))$ such that $f^{\ell-m}(w)\in f^{\ell-m}[W^u(\hf^m(\hx)]\pitchfork W^s(\hf^\ell(\hy))$ and $\{w,f(w),\dots , f^{\ell-m}(w)\}\cap \mathfs S=\emptyset$.
Since $\mathfs S$ is closed, there is $D'\subset W^u(\hf^m(\hx))$ a disc of the same dimension of 
$W^u(\hf^m(\hx))$ containing $w$ such that $f^{\ell-m}(w)\in f^{\ell-m}(D')\pitchfork W^s(\hf^\ell(\hy))$ and
$\{D',f(D'),\dots , f^{\ell-m}(D')\}\cap \mathfs S=\emptyset$.

Using that $\mathfs V^u(\wh y)\pitchfork \mathfs V^s(\wh z)\neq \emptyset$, 
Lemma \ref{lemma.intersection} implies the existence of $i\leq 0\leq j$ 
and $v\in W^u(\hf^i(\wh y))$ such that $f^{j-i}(v)\in f^{j-i}[W^u(\hf^i(\hy)]\pitchfork W^s(\hf^j(\wh z))$ 
and $\{v,f(v),\dots , f^{j-i}(v)\}\cap \mathfs S=\emptyset$.
Write $\wh y=(y_k)_{k\in\Z}$, and let $g=f^{-1}_{y_i}\circ\cdots\circ f^{-1}_{y_{\ell-1}}$, which is well-defined 
on a neighborhood of $W^u(\wh f^\ell(\hy))$.  By Proposition \ref{lambda.u-lemma}, 
there are discs $D_k\subset D'$ and $n_k\to\infty$ such that $f^{n_k}(D_k)$
converges to $W^u(\wh f^\ell(\wh y))$ in the $C^1$ topology and $[D_k\cup f(D_k)\cup\cdots\cup f^{n_k}(D_k)]\cap\mathfs S=\emptyset$. Since $g$ is smooth, 
it follows that $g[f^{n_k}(D_k)]$ converges to $g[W^u(\wh f^\ell(\wh y))]$ in the $C^1$ topology. 
By \cite[Prop. 4.7(2)]{ALP-23}, this latter set is a subset of 
$W^u(\wh f^i(\wh y))$ containing $y_i$. Hence, for a fixed $k_0$ large enough,
$\wt\Delta:=g[f^{n_{k_0}}(D_{k_0})]$ is transverse to
$W^s(\wh f^i(\wh y))$. Applying again 
Proposition \ref{lambda.u-lemma} to $\wt\Delta$, there are discs $\wt D_k\subset\wt\Delta$ and $m_k\to\infty$
such that $f^{m_k}(\wt D_k)$ converges to $W^u(\wh f^i(\wh y))$ in the $C^1$ topology 
and $[\wt D_k\cup f(\wt D_k)\cup\cdots\cup f^{m_k}(\wt D_k)]\cap \mathfs S=\emptyset$.

Recalling that $v\in W^u(\hf^i(\wh y))$, for large $k$ there is $w_k\in \wt D_k$ such that 
$f^{m_k}(w_k)\to v$ and $w'_k\in \wt D_k$ such that $f^{j-i}(w'_k)\in f^{j-i}[f^{m_k}(\wt D_k)]\pitchfork W^s(\wh f^j(\wh z))$.
Choose $k_1$ large enough satisfying this latter transversality and so that $m_{k_1}>\ell-i$,
and write $\wt D_{k_1}=g[f^{n_{k_0}}(D)]$ for $D\subset D_{k_0}$. Then
$$
f^{m_{k_1}}(\wt D_{k_1})=f^{m_{k_1}}(g[f^{n_{k_0}}(D)])=f^{m_{k_1}-(\ell-i)}[f^{n_{k_0}}(D)]=
f^{m_{k_1}+n_{k_0}-(\ell-i)}(D).
$$
Since $D\subset D_{k_0}\subset W^u(\wh f^m(\wh x))$, it follows that
$$
f^{(j-i)+m_{k_1}+n_{k_0}-(\ell-i)}[W^u(\wh f^m(\wh x))]\pitchfork W^s(\wh f^j(\wh z))\neq\emptyset.
$$
Since $D$ does not intersect $\mathfs S$ up to iterate $(j-i)+m_{k_1}+n_{k_0}-(\ell-i)$,
the last part of Lemma \ref{lemma.intersection} implies there is $n\in\Z$ such that
$\mathfs V^u(\wh x)\pitchfork \wh f^n(\mathfs V^s(\wh z))\neq\emptyset$.
\end{proof}

\begin{proof}[Proof of Proposition \ref{prop.equivalence.relation}]
This result is a version of \cite[Proposition 2.11]{BCS-Annals} in our context, 
and we follow the same proof. By the standing assumption, there are sets $A_1,A_2,A'_2,A_3\subset \wh M$ 
with $\mu_1(A_1)>0$, $\mu_2(A_2)>0$, $\mu_2(A_2')>0$, $\mu_3(A_3)>0$ such that 
$\mathfs V^u(\wh x^{(1)})\pitchfork \mathfs V^s(\wh x^{(2)})\neq\emptyset$
for all $(\wh x^{(1)},\wh x^{(2)})\in A_1\times A_2$ and 
$\mathfs V^u(\wh x^{(2)})\pitchfork \mathfs V^s(\wh x^{(3)})\neq\emptyset$
for all $(\wh x^{(2)},\wh x^{(3)})\in A_2'\times A_3$. Since $\mu_2$ is ergodic, we can choose $A_2=A_2'$
and assume that:
\begin{enumerate}[$\circ$]
\item For every $(\wh x^{(1)},\wh x^{(2)})\in A_1\times A_2$ there is $n\in\Z$ such that
$$
\mathfs V^u(\wh x^{(1)})\pitchfork \wh f^n(\mathfs V^s(\wh x^{(2)}))\neq\emptyset.
$$
\item For every $(\wh x^{(2)},\wh x^{(3)})\in A_2\times A_3$ there is $n\in\Z$ such that
$$
\mathfs V^u(\wh x^{(2)})\pitchfork \wh f^n(\mathfs V^s(\wh x^{(3)}))\neq\emptyset.
$$
\end{enumerate}
Reducing $A_1,A_2,A_3$ if necessary, we can also assume that for $i=1,2,3$:
\begin{enumerate}[$\circ$]
\item $\mu_i(A_i)>0$.
\item $A_i$ is contained in a Pesin block of $\mu_i$.
\item Each element of $A_1\cup A_2\cup A_3$ satisfies the conditions of Proposition \ref{lambda.u-lemma}.
\item $A_i$ is contained in the support of $\mu_i|_{A_i}$.
\end{enumerate}
Fix $\wh x^{(1)}\in A_1$ that is backwards recurrent to $A_1$, $\wh x^{(2)}\in A_2$ arbitrary, and
$\wh x^{(3)}\in A_3$ that is forward recurrent to $A_3$. By Corollary \ref{cor.trans}, there is $n\in\Z$ such that
$\mathfs V^u(\wh x^{(1)})\pitchfork \wh f^n(\mathfs V^s(\wh x^{(3)}))\neq\emptyset$.
Translating this to local invariant sets, there are $n_1,n_3\geq 0$
such that
$$
\wh f^{n_1}(V^u(\wh f^{-n_1}(\wh x^{(1)})))\pitchfork \wh f^{n-n_3}(V^s(\wh f^{n_3}(\wh x^{(3)})))\neq\emptyset.
$$
The same transversality holds changing $n_1,n_3$ to $N_1,N_3$ with $N_i\geq n_i$.
Choose $N_1\geq n_1$ and $N_3\geq n_3$ such that $\wh f^{-N_1}(\wh x^{(1)})\in A_1$ and 
$\wh f^{N_3}(\wh x^{(3)})\in A_3$. Then 
$$
\wh f^{N_1}(V^u(\wh f^{-N_1}(\wh x^{(1)})))\pitchfork \wh f^{n-N_3}(V^s(\wh f^{N_3}(\wh x^{(3)})))\neq\emptyset.
$$
Since $W^{s/u}$ are continuous on Pesin blocks,
\begin{equation}\label{eq-transitivity}
\wh f^{N_1}(V^u(\wh y^{(1)}))\pitchfork \wh f^{n-N_3}(V^s(\wh y^{(3)}))\neq\emptyset
\end{equation}
for $\wh y^{(1)}\in A_1$ close to $\wh f^{-N_1}(\wh x^{(1)})$ and $\wh y^{(3)}\in A_3$
close to $\wh f^{N_3}(\wh x^{(3)})$. \footnote{More specifically, the transversality condition
for $\hx^{(1)},\hx^{(3)}$
means that $f^k(W^u(\hf^\ell(\wh x^{(1)})))\pitchfork W^s(\hf^m(\wh x^{(3)}))\neq\emptyset$
for some $k,\ell,m$. Since $W^{s/u}$ varies continuously on Pesin blocks,
this transversality remains true for points close to $\hx^{(1)}$ and $\hx^{(3)}$.}
Choosing a neighborhood $B_1\subset A_1$
of $\wh f^{-N_1}(\wh x^{(1)})$ with $\mu_1(B_1)>0$ and a neighborhood $B_3\subset A_3$
of $\wh f^{N_3}(\wh x^{(3)})$ with $\mu_3(B_3)>0$, it follows that (\ref{eq-transitivity})
holds for all $(\wh y^{(1)},\wh y^{(3)})\in B_1\times B_3$. Retranslating this back 
to global invariant sets, we conclude that 
$\mathfs V^u(\wh z^{(1)})\pitchfork \wh f^n(\mathfs V^s(\wh z^{(3)}))\neq\emptyset$
for all $(\wh z^{(1)},\wh z^{(3)})\in \wh f^{N_1}(B_1)\times \wh f^{-N_3}(B_3)$, and so
$\mu_1\preceq \mu_3$.
\end{proof}

\subsection{Proof of Theorem \ref{Thm-Main}}

Fix $\mu$ an adapted, ergodic and hyperbolic $f$--invariant probability measure, 
a real number $\chi>0$, and let $\wh\pi$ be the coding from Theorem \ref{Thm-ALP}. We will find
an irreducible component $\Sigma\subset\whS$ that satisfies (C2) and (C9).
The proof is divided into two steps.

\medskip
\noindent
{\sc Step 1:} Let $\{\mathcal O_i\}$ be the set of all $\chi$--hyperbolic periodic orbits
homoclinically related to $\wh\mu$. Then there is an irreducible component
$\Sigma \subset\wh\Sigma$ that lifts all $\{\mathcal O_i\}$.

\begin{proof}[Proof of Step $1$] For diffeomorphisms, this statement is 
\cite[Lemma 3.12]{BCS-Annals}, and we follow the same approach.
Fix $\wh x^{(i)}\in\mathcal O_i$. We start constructing, for each $n$,
a transitive compact invariant $\chi$--hyperbolic set $K_n\subset \wh M$ that contains $\mathcal O_1,\ldots,\mathcal O_n$.
For each $i,j$, let $\wh y^{(ij)}\in\mathfs V^u(\wh x^{(i)})\pitchfork \mathfs V^s(\wh f^{a_{ij}}(\wh x^{(j)}))$
with $0\leq a_{ij}\leq |\mathcal O_j|$, where $\hy^{(ii)}:=\hx^{(i)}$.
The set $L=\bigcup_{1\leq i,j\leq n}\mathcal O(\wh y^{(ij)})$
is $\chi$--hyperbolic, compact and invariant. Hence $L$ is uniformly hyperbolic, and so there are 
$\ve,\delta>0$ such that every $\ve$--pseudo-orbit of $L$ is $\delta$--shadowed.\footnote{Here we consider
the classical notions, i.e. $\{\wh z^{(i)}\}_{i\in\Z}$ is an $\ve$--pseudo-orbit if $d(\wh f(\wh z^{(i)}),\wh z^{(i+1)})<\ve$
and $d(\wh z^{(i)},\wh f^{-1}(\wh z^{(i+1)}))<\ve$ for all $i\in\Z$. Shadowing is obtained as in
the classical setting of uniformly hyperbolic diffeomorphisms, see e.g. \cite[Theorem IV.2.3]{Qian-Zhu-Xie}.} 
Let $\ve'>0$ with the following property: if $\wh z,\wh w\in L$ satisfy $\wh d(\wh z,\wh w)<\ve'$,
then $\wh d(\wh f(\wh z),\wh f(\wh w))<\tfrac{\ve}{2}$.  
Fix $m$ large such that the set $\{\wh f^k(\wh y^{(ij)}):1\leq i,j\leq n\text{ and }-m\leq k<m+a_{ij}\}$
is $\ve'$--dense in $L$. As done in \cite[Lemma 3.11]{BCS-Annals}, this finite set defines a {\em transitive} compact 
TMS, whose elements are $\ve$--pseudo-orbits. Let $\wt\pi$ be the shadowing map,
and $K_n$ be the image of $\wt\pi$. This set satisfies the claimed properties. 

Recall that $\wh\pi:\wh\Sigma\to\nuh^\#$ is the coding given by Theorem \ref{Thm-ALP}.
By property ($\wh{\textrm C}$9), for each $n$ there is an irreducible component $\Sigma_n\subset\wh\Sigma$
which lifts $K_n$. Since $\mathcal O_1$ has finitely many lifts to $\wh\Sigma^\#$, 
and every $\Sigma_n$ contains one such lift, there is an irreducible component $\Sigma$
such that $\Sigma_n=\Sigma$ for infinitely many $n$. Clearly, $\Sigma$ satisfies the conditions
of Step 1.
\end{proof}

\medskip
\noindent
{\sc Step 2:} Every $\nu\in\mathbb P_h(\wh f)$ that is homoclinically
related to $\wh\mu$ lifts to $\Sigma$.  

\begin{proof}[Proof of Step $2$]
For diffeomorphisms, this statement is \cite[Lemma 3.13]{BCS-Annals}, and we follow the same approach.
We will show that $\nu$--a.e. $\wh x$ has a lift to $\Sigma^\#$. Once this is proved,
we can lift $\nu$ to $\Sigma^\#$ as in \cite{Sarig-JAMS}, defining
$$
\eta=\int_{\wh M}\left(\frac{1}{\#(\wh\pi^{-1}(\wh x)\cap\Sigma^\#)}\sum_{\un R\in\wh\pi^{-1}(\wh x)\cap\Sigma^\#}\delta_{\un R}\right)d\nu(\wh x).
$$ 
By Theorem \ref{Thm-ALP}, $\nu$ has an ergodic lift $\overline\nu$ to $\wh\Sigma^\#$. 
Let $\un R\in\wh\Sigma^\#$ be a recurrent and generic point for $\overline\nu$, i.e.
$\tfrac{1}{n}\sum_{i=0}^{n-1}\delta_{\wh\sigma^i(\un R)}\to \overline\nu$ in the weak--$^*$  topology.
The point $\wh x=\wh\pi(\un R)$ is generic for $\nu$. We first show that $\wh x$ has a lift
to $\Sigma$. Since $\un R$ is recurrent, there is a sequence $\{\un q^{(i)}\}\subset\wh\Sigma$
of periodic points such that $\un q^{(i)}\to \un R$. Indeed, since $\un R$ is recurrent
there are $R\in\mathfs R$ and sequences $n_i\to+\infty$
and $m_i\to -\infty$ such that $R_{n_i}=R_{m_i}=R$ for all $i$, so we can define 
$\un q^{(i)}$ by repeating the pattern $(R_{m_i},R_{m_i+1},\ldots,R_{n_i})$.
Let $\wh x^{(i)}=\wh\pi(\un q^{(i)})$, which is a hyperbolic periodic point, since
it belongs to $\wh\pi[\wh\Sigma^\#]=\nuh^\#$.
Since $\un q^{(i)}$ is (symbolically) homoclinically related to $\overline\nu$, its
projection $\wh x^{(i)}$ is homoclinically related to $\nu$. By Proposition
\ref{prop.equivalence.relation}, $\wh x^{(i)}$ is homoclinically related to $\mu$.
By Step 1, there is $\un p^{(i)}\in\Sigma$ periodic such that $\wh\pi(\un p^{(i)})=\wh x^{(i)}$.
Now, observe that since $\{\un q^{(i)}\}$ is relatively compact ($q^{(i)}_0=R_0$ for all $i$ and
the vertices defining $\wh\Sigma$ have finite degrees), property (C8) of Theorem \ref{Thm-ALP}
implies that $\{\un p^{(i)}\}$ is relatively compact as well. Therefore, we can pass to a converging
subsequence $\un p^{(i_k)}\to \un p\in\Sigma$. By continuity, $\wh\pi(\un p)=\wh x$. Finally, we show as in 
\cite[Lemma 3.13]{BCS-Annals} that $\un p\in\Sigma^\#$, thus completing the proof of Step 2. 
\end{proof}

Defining $\pi=\wh\pi|_{\Sigma}$, we have proved (C0) and item (a) of (C2). Item (b) of (C2) is proved as in \cite[Prop. 3.6]{BCS-Annals}.\\

\medskip
\noindent
{\bf Property (C9).} By Step 1 in the proof of property ($\wh{\textrm C}$9) of
Theorem \ref{Thm-ALP}, there is $X_0\subset\whS$ compact such that $\wh\pi(X_0)\supset K$.
Therefore, given $\hx\in K$ there is $\un R\in X_0$ such that $\wh\pi(\un R)=\hx$. Since $X_0$ is compact,
$\un R$ is recurrent and so we can repeat the argument of Step 2 above to find $\un p\in \Sigma^\#$
such that $\wh\pi(\un p)=\hx$. We have thus lifted every element of $K$ to $\Sigma$. Now repeat
the proof of  property ($\wh{\textrm C}$9) of Theorem \ref{Thm-ALP} to $\pi$.
This gives $X\subset\Sigma$ compact such that $\pi(X)=K$.\\

We have thus concluded the proof of Theorem \ref{Thm-Main}.

\begin{proof}[Proof of Theorem~\ref{thm.criterion}]
Let $\mu$ be an ergodic adapted hyperbolic $f$--invariant measure, and fix $\chi>0$. By Theorem~\ref{Thm-Main},
there is $\Sigma$ an irreducible TMS and  $\pi:\Sigma\to\wh M$ a coding such that
every adapted $\chi$--hyperbolic $\nu\stackrel{h}{\sim}\mu$ lifts to a measure $\overline\nu$ on $\Sigma$.
Also, property (C2) implies that if $\nu$ is a measure of maximal entropy for $f$ then $\overline\nu$
is a measure of maximal entropy for $\sigma$.
By \cite{Gurevich-Topological-Entropy,Gurevich-Measures-Of-Maximal-Entropy}, 
$\sigma:\Sigma\to\Sigma$ has at most one measure of maximal entropy.
Hence $f$ has at most one adapted, $\chi$--hyperbolic measure of maximal entropy that is 
homoclinically related to $\mu$. Since $\chi>0$ is arbitrary,
we conclude that $f$ has at most one adapted hyperbolic measure of maximal entropy
homoclinically related to $\mu$.

Now assume that $\nu$ is an adapted measure of maximal entropy homoclinically related to $\mu$.
Let $\wh\nu$ be its lift to $\hM$ and $\overline\nu$ its lift to $\Sigma$.
It remains to show that $\wh\nu$ is Bernoulli times a finite rotation, and to identify
its support. We start studying the Bernoulli property.
The measure $\overline\nu$ is a measure of maximal entropy for $\sigma$, hence it is isomorphic
to the product of a Bernoulli shift and a finite rotation \cite{Sarig-Bernoulli-JMD}.
As in the proof of Theorem 1.1 of \cite{Sarig-Bernoulli-JMD}, the same happens to $\wh\nu$.

Let $\mathcal O$ be a hyperbolic periodic orbit for $\hf$ homoclinically related to $\mu$.  Finally, we show that the support of $\wh\nu$ equals ${\rm HC}(\mathcal O)$.  Observe that $\mathcal O$ is homoclinically related to $\nu$.
The measure $\overline\nu$ has full support in $\Sigma$ \cite{Buzzi-Sarig}, hence
${\rm supp}(\wh\nu)=\overline{\pi(\Sigma)}$. It is not hard to see that
${\rm supp}(\wh\nu)\subset{\rm HC}(\mathcal O)$. Indeed, given $\hx\in {\rm supp}(\wh\nu)$
and a neighborhood $U$ of $\hx$ with $\wh\nu(U)>0$, we can take $\hy\in U\cap Y'$ 
such that $\mathfs V^u(\mathcal{O})\pitchfork \mathfs V^s(\hy)\neq\emptyset$ and 
$\mathfs V^u(\hy)\pitchfork \mathfs V^s(\mathcal{O})\neq\emptyset$. By the inclination lemma
(Proposition \ref{lambda.u-lemma}),
it follows that $\hy\in {\rm HC}(\mathcal O)$. We thus have
$$
\overline{\pi(\Sigma)}={\rm supp}(\wh\nu)\subset  {\rm HC}(\mathcal O)
$$
and so it is enough to show that
$\pi(\Sigma)$ is dense in ${\rm HC}(\mathcal O)$. 
The proof of this fact is made as in \cite[Corollary 3.3]{BCS-Annals},
which works equality well in our context due to the properties (C1)--(C9).
\end{proof}

\begin{remark}
The above proof gives a uniqueness result for more general equilibrium states, as follows.
Assuming the setting of Theorem \ref{thm.criterion}, let $\varphi:M \to \mathbb{R}$ such that
its lift $\overline\vf=\varphi \circ \vt\circ\pi$ to $\Sigma$ is H\"older continuous. 
The potential $\overline\vf$ has at most one equilibrium state \cite{Buzzi-Sarig}, hence 
$\vf$ has at most one adapted hyperbolic equilibrium state that is homoclinically related to $\mu$.
\end{remark}

\subsection{The Bernoulli property} \label{subsection.bernoulli}
In this section, we collect two properties ensuring that, when it exists, the
measure of maximal entropy in Theorem \ref{thm.criterion} is Bernoulli. We will apply these criteria in the proofs of
Theorems \ref{thm.billiards}, \ref{thm.unique.ACS}, \ref{thm.ergodicity}.
Recall that, for a non-invertible map $f$, an $f$--invariant measure $\mu$ is \emph{Bernoulli} if its lift
$\wh\mu$ to $\hM$ is isomorphic to a Bernoulli shift.

In the sequel, we let $\wh{\pi}: \widehat{\Sigma}\to \hM$ be the map given by Theorem \ref{Thm-ALP}.
Call $\varphi:M\to \R$ an \emph{admissible potential} if its lift $\overline{\varphi}=\varphi\circ \vartheta\circ\wh{\pi}:\widehat{\Sigma}\to \R$ is H\"older continuous.
Let $\vf_n=\vf+\vf\circ f+\cdots+\vf\circ f^{n-1}$ be the
$n$--th Birkhoff sum of $\vf$. If $\vf$ is admissible then so is each
$\vf_n$. \footnote{For all $k\geq 0$, the lift
$\overline{\vf\circ f^k}=\vf\circ f^k\circ \vt\circ \hat{\pi} = \vf\circ \vt\circ \wh f^k\circ \hat{\pi}=
\vf\circ\vt\circ\hat{\pi}\circ\hat{\sigma}^k=\overline{\vf}\circ\hat{\sigma}^k$
is the composition of two Hölder maps.}

\begin{remark} Observe that $\varphi$ does not even need to be continuous
to be admissible;
we only require its lift to $\widehat{\Sigma}$ to be H\"older continuous. This condition
disregards the trajectories that intersect the singular set $\mathfs S$.
\end{remark}

Below, we consider equilibrium states for pairs $(f^n,\vf_n)$.

\begin{proposition}\label{thm.unique-bernoulli}
Under the assumptions of Theorem \ref{thm.criterion}, let $\varphi:M\to\R$ be an admissible potential
and $\mu$ an ergodic equilibrium state for $(f,\varphi)$ that is adapted and hyperbolic. If one of the two
conditions below hold:
\begin{enumerate}[{\rm (1)}]
\item For all $n>0$, there is at most one equilibrium state for $(f^n,\vf_n)$;
\item There exists $\nu\sim \mu$ adapted and hyperbolic which is ergodic for all $f^n$, $n>0$;
\end{enumerate}
then $\mu$ is Bernoulli. Additionally, if \emph{(2)} holds and $\nu$ is fully supported then $\mu$ is fully supported.
\end{proposition}

\begin{proof}
Let $\chi>0$ small such that $\mu$ is $\chi$--hyperbolic and, when (2) holds, $\nu$ is also 
$\chi$--hyperbolic. 
Applying Theorem \ref{thm.criterion} to $\mu$ and $\chi$, we know by \cite{Sarig-Bernoulli-JMD}
that $\wh\mu$ is Bernoulli times a finite rotation,
i.e. there are $p>0$
and disjoint sets $X_0,X_1,\ldots,X_{p-1}\subset \hM$ such that $\hf(X_i)=X_{i+1}$ and $(\hf^p,\hmu_i)$ a Bernoulli shift
for $i=0,\ldots,p-1$, where $\hmu_i=\hmu(\cdot |X_i)$. Let $\mu_i=\hmu_i\circ\vt^{-1}$.
Since $\mu$ is an equilibrium state for the pair $(f,\vf)$,
it is not hard to see that each $\mu_i$ is an equilibrium state for $(f^p,\vf_p)$. 
Note also that $\mu_i$ is adapted and $\chi$--hyperbolic.

Assume that (1) holds. Since $\mu_0,\ldots,\mu_{p-1}$ are distinct equilibrium states for $(f^p,\vf_p)$,
it follows that $p=1$, i.e. $\wh\mu$ is Bernoulli.

Now assume that (2) holds. Since $\nu\sim\mu$,
necessarily $\nu\sim\mu_i$ for $i=0,1,\ldots,p-1$. By Theorem \ref{Thm-Main},
$\mu_0,\ldots,\mu_{p-1}$ lift to $\overline{\mu}_0,\ldots,\overline{\mu}_{p-1}$. Each 
$\overline{\mu}_i$ is an equilibrium state for the pair $(\sigma^p,\overline{\vf}_p)$. Since 
$\overline{\vf}_p$ is H\"older continuous, \cite{Buzzi-Sarig} implies that there is at most one
such equilibrium measure, hence $p=1$. 
Finally, assume also that $\nu$ is fully supported. Since $\nu$ has a lift $\overline\nu$ to $\Sigma$,
we have $M={\rm supp}(\nu)\subset\overline{\pi(\Sigma)}$ and so $\overline{\pi(\Sigma)}=M$.
Also by  \cite{Buzzi-Sarig}, $\overline{\mu}$ is fully supported in $\Sigma$ and so, 
as in the end of the proof of Theorem \ref{thm.criterion}, we conclude that
${\rm supp}(\mu)=\overline{\pi(\Sigma)}=M$.
\end{proof}

\section{Finite horizon dispersing billiards}\label{sec.billiards}

In this section we prove Theorem \ref{thm.billiards}.
We show that finite horizon dipersing billiards have a single 
homoclinic class. The proof follows closely some arguments
of Sina{\u\i} as made in \cite{CDLZ-24}. For that, we refer the reader 
to nowadays classical results in the field, which may be found in
the book of Chernov and Markarian \cite{Chernov-Markarian}.

Recall that $\mathfs T:= \mathbb{T}^2 \setminus (\bigcup_{i=1}^\ell O_i)$ is a billiard table,
where $O_1,\ldots,O_\ell$ are pairwise disjoint closed, convex subsets of $\mathbb{T}^2$
such that each boundary $\partial O_i$ is a $C^3$ curve with strictly positive curvature, 
$M = \bigcup_{i=1}^\ell (\partial O_i \times [-\pi/2, \pi/2])$, and
$f:M\to M$ is the associated dispersing billiard map, which is a map with singularities 
$\mathfs S=\mathfs S_0\cup f^{-1}(\mathfs S_0)$ where $\mathfs S_0=\{(r,\vf)\in M:|\vf|=\tfrac{\pi}{2}\}$
are the grazing collision. Recall also that we are assuming that $f$ has finite horizon, i.e.
$\sup_{x\in M}\tau(x)<\infty$ where $\tau(x)$ is the flight time from $x\in M$ 
to $f(x)$. 

Parametrize $\partial O_i$ by arclength $r$ (oriented clockwise).
It is well-known that $f$ preserves a smooth probability measure $\musrb$, such that
$$
d\musrb(x) = \frac{1}{2|\partial Q|} \cos \vf \, dr d\vf
$$
which is adapted (see \cite[Section I.3]{KSLP}), hyperbolic and ergodic (even Bernoulli).
In particular, $\musrb$--a.e. $x\in M$ has local stable/unstable manifolds $W^{s/u}(x)$.

We show that points with local invariant manifolds are homoclinically related. For that, we need some notation.
Let $\mathcal K$ denote the curvature function of the obstacles. 
Let $\tau_{\min},\mathcal K_{\min}$ denote the minima of $\tau,\mathcal K$ respectively.

\medskip
\noindent
{\sc Invariant cones (see \cite[Section~4.5]{Chernov-Markarian}):}
The map $f$ has {\em stable/unstable invariant cones} 
\begin{align}\label{cones}
\begin{array}{l}
\mathcal C^s_x=\left\{(dr,d\vf)\in T_xM: -\mathcal K-\frac{\cos\vf}{\tau(x)}\leq d\vf/dr\leq -\mathcal K\right\}\\
\\
\mathcal C^u_x=\left\{(dr,d\vf)\in T_xM: \mathcal K\leq d\vf/dr\leq \mathcal K+\tfrac{\cos\vf}{\tau(f^{-1}(x))}\right\},
\end{array}
\end{align}
which are uniformly transverse and contract/expand uniformly with a rate at
least 
\begin{equation}\label{equation.Lambdadef}
\Lambda:=1+2\tau_{\min} \mathcal K_{\min}\textrm{  (see \cite[Eq. (4.19)]{Chernov-Markarian}).}
\end{equation}
We call a $C^1$ curve $W$ a {\em stable curve} if the tangent vector at each $x \in W$ lies in
$\mathcal C^s_x$. Unstable curves are defined similarly.

\medskip
\noindent
{\sc Solid rectangles and Cantor rectangles (see \cite[Section~7.11]{Chernov-Markarian}):}	
A closed subset $D\subset M$ is called a {\em solid rectangle} if $D$ is the closure of its interior and
$\partial D$ is the union of
four smooth curves, two local stable and two local unstable manifolds. 
Let $\mathfrak{S}^{s/u}(D)$ denote the set of local stable/unstable manifolds
of points in $D$ that do not terminate in the interior of $D$.
We then define the {\em Cantor rectangle} 
$$
R(D) = \mathfrak{S}^s(D) \cap \mathfrak{S}^u(D) \cap D.  
$$
Conversely, given a Cantor rectangle $R$, we denote by $D(R)$ the smallest solid
rectangle containing $R$.

\medskip
\noindent
{\sc $s/u$--subrectangles (see \cite[Section~7.11]{Chernov-Markarian}):}
Given a Cantor rectangle $R$, we call $S \subset R$ an {\em $s$--subrectangle}
of $R$ if for each $x \in S$ we have $W^s(x) \cap S = W^s(x) \cap R$.
A {\em $u$--subrectangle} is defined similarly. 

\medskip
\noindent
{\sc $s/u$--crossing (see \cite[Section~7.11]{Chernov-Markarian}):}	
Given a Cantor rectangle $R$, we say a local stable manifold $W^s$
{\em fully crosses} $R$ if $W^s \cap \mathring{D}(R) \ne \emptyset$ and $W^s$ does not terminate in the
interior of $D(R)$. A Cantor rectangle $R'$ is said to {\em $s$--cross} $R$
if every stable manifold $W^s \in \mathfrak{S}^s(R')$
fully crosses $R$. We define $u$--crossings similarly.

\medskip
Given a Cantor rectangle $R$, the set $f^n(R)$ is a finite union of (maximal) Cantor rectangles,
which we call $R_{n,i}$. It is clear that each $f^{-n}(R_{n,i})$ is an
$s$--subrectangle of $R$. Recall the set ${\rm NUH}^\#_\chi$ introduced in Section \ref{ss.proof.ALP}.

\begin{lemma}\label{lemma.homoclinic.billiards}
If $x,y\in {\rm NUH}^\#_\chi$ then there is $k>0$ such that $f^k(W^u(x))$ and
$W^s(y)$ intersect transversally in a point not belonging to $\bigcup_{n\in\Z}f^n(\mathfs S)$.
\end{lemma}

\begin{proof}
This is essentially the discussion in \cite[Section 3.3]{CDLZ-24}.
Fix $\delta>0$ small such that $W^u(x),W^s(y)$ have length at least $\delta$. 
By \cite[Lemma 7.87]{Chernov-Markarian}, there is a finite collection of positive $\musrb$--measure 
rectangles $R_1,\ldots,R_N$ such that any stable and unstable curve of length at least $\delta/2$ 
fully crosses at least one of the rectangles. Without loss of generality, we can assume that 
$W^u(x)$ fully crosses $R_1$ and $W^s(y)$ fully crosses $R_2$.

By  \cite[Lemma 7.90]{Chernov-Markarian}, there is a `magnet' rectangle $R^*$ of positive $\musrb$--measure, 
and a `high density' subset $\mathfrak{P}^* \subset R^*$, satisfying
the following property: if $R_{k,n,i}$ is a maximal rectangle in $f^n(R_k)$ and
$R_{k,n,i} \cap \mathfrak{P}^* \neq\emptyset$ where $n$ is large enough,
then $R_{k,n,i}$ $u$--crosses $R^*$. By the definition of maximal rectangle,
the iterates $f^{-n}(R_{k,n,i}),f^{-n+1}(R_{k,n,i}),\ldots,R_{k,n,i}$ do not intersect $\mathfs S$.
The analogous properties hold for maximal
rectangles $R_{k,-n,i}$ of $f^{-n}(R_k)$ and $s$--crossings of $R^*$.

Now, since $\musrb$ is mixing, there are $m,n>0$ such that $f^m(R_1)\cap \mathfrak{P}^*\neq\emptyset$
and $f^{-n}(R_2)\cap \mathfrak{P}^*\neq\emptyset$. Therefore there are an $s$-subrectangle $R'\subset R_1$
and a $u$-subrectangle $R''\subset R_2$ such that $f^m(R')$ $u$-crosses $R^*$ and 
$f^{-n}(R'')$ $s$-crosses $R^*$. 
This implies that $f^m(W^u(x))\pitchfork f^{-n}(W^s(y))\neq\emptyset$
in a point not belonging to $\bigcup_{n\in\Z}f^n(\mathfs S)$, which concludes the proof.
\end{proof}

Next, we recall the definition of the geometric potential and show how to guarantee that 
it is H\"older continuous with respect to the symbolic metric.
Every  $x\in M$ has a well-defined unstable direction $E^u_x$.

\medskip
\noindent
{\sc Geometric potential:} The {\em geometric potential} of $f$ is the map $\vf:M\to[-\infty,0]$
given by $\varphi(x) = -\log \|df_x|_{E^u_x}\|$.

\medskip
Note that $-\infty\leq \vf(x)\leq -\log\Lambda$. More precisely, \cite[Equation (4.20)]{Chernov-Markarian}
gives that $\vf(x)\approx \log d(f(x),\mathfs S_0)$ when $f(x)\not\in\mathfs S_0$, i.e. the ratio between these
two functions is bounded away from zero and infinity.

\begin{lemma}\label{lemma.SRB=equilibrium}
The potential $\vf$ has zero topological pressure, and $\musrb$ is an equilibrium state for $\vf$.
\end{lemma}

\begin{proof}
The Pesin entropy formula holds for $\musrb$ \cite{KSLP}, hence the pressure
of $\musrb$ is 
$$
P_{\musrb}(\vf)=h_{\musrb}(f)+\int\vf d\musrb=0.
$$
Let $\mu$ be an arbitrary $f$--invariant probability measure. We consider two cases.
If $\mu$ is adapted, then by \cite[Theorem 1.1]{Liao-Qiu} the Ruelle inequality holds, thus
$$
P_\mu(\vf)=h_\mu(f)+\int \vf d\mu\leq 0.
$$
If $\mu$ is not adapted, the estimate $\vf(x)\approx \log d(f(x),\mathfs S_0)$ implies
$$
\int \vf d\mu =\int (\vf\circ f^{-1})d\mu \approx \int\log d(x,\mathfs S_0)=-\infty.
$$
Since $h_\mu(f)\leq h_{\rm top}(f)<\infty$ by \cite[Theorem 2.3]{Baladi-Demers-MME}, it follows that
$$
P_\mu(\vf)=h_\mu(f)+\int\vf d\mu \leq h_{\rm top}(f)+\int\vf d\mu =-\infty.
$$
Therefore $P_{\rm top}(\vf)=0$, with equality for the measure $\musrb$.
\end{proof}

\begin{lemma}\label{lemma.holderregularitypotential}
There are a locally compact countable Markov shift $(\Sigma,\sigma)$ and a H\"older continuous
map $\pi:\Sigma\to M$ satisfying Theorem \ref{Thm-Main} for the measure $\musrb$
such that $\varphi \circ \pi$ is H\"older continuous. 
\end{lemma}

\begin{proof} 
We need some facts from the theory of dispersing billiards and from the construction in \cite{ALP-23}.
\begin{enumerate}[(1)]
\item There is $C_1>0$ such that $\|df_x\| \leq C_1 d(x,\mathfs S)^{-1}$ and $\|d^2f_x\| \leq C_1 d(x,\mathfs S)^{-3}$,
see \cite[Thm. 7.2]{KSLP}.
\item Since $\|df_x\mid_{E^u_x}\|\geq \Lambda>1$, the spectrum of $f$ is contained in
$\mathbb R\setminus (-\log\Lambda,\log\Lambda)$, hence we fix $0<\chi<\log\Lambda$. 
\item In the construction of \cite{ALP-23}, which was quickly summarized in Section \ref{ss.proof.ALP},
each $x\in\nuh$ is associated to a parameter $Q(x)>0$ and a {\em Pesin chart}
$\Psi_x:[-Q(x),Q(x)]^2\to M$, where $Q(x)< d(x,\mathfs S)^{96}$. \footnote{
The precise exponent in \cite{ALP-23} is $96a/\beta$; for billiards $a=\beta=1$, see \cite[Thm. 7.2]{KSLP}.}
The choice of 96 is arbitrary, and the construction works for any $L\geq 96$. We thus
choose $L:=\max\{96,14/\chi\}$.  
\end{enumerate}
The effect of choosing $L$ large is that $df$ becomes uniformly H\"older continuous inside images of Pesin charts:
if $y,z\in \Psi_x[-Q(x),Q(x)]^2$, then 
$$
\|df_y-df_z\|\leq 2C_1d(x,\mathfs S)^{-3}d(y,z)\leq C_1d(y,z)^{1/2}.
$$

Now recall that $\pi$ and $\un v\in\Sigma\mapsto E^u_{\pi(\un v)}$ are H\"older continuous,
see property (C3) of Theorem \ref{Thm-Main}. 
By \cite{Sarig-JAMS,ALP-23}, their H\"older exponents are at least $\chi/2$ and $\chi/7$
respectively.
Let $\un v,\un w\in \Sigma$ with $v_0=w_0$ and let $y=\pi(\un v),y=\pi(\un w)$,
which belong to the image of the Pesin chart defined by $v_0$.
Since map $\norm{df_x\mid_{E^u_x}}$ is uniformly bounded from below, there is $C_2>0$
such that 
$$
\begin{aligned}
&\ \abs{(\varphi\circ \pi)(\un v)-(\varphi\circ \pi)(\un w)}
\leq C_2\left|\|df_y\mid_{E^u_y}\|-\|df_z\mid_{E^u_z}\|\right|\\
&\leq C_2\left[\norm{df_y-df_z}+\norm{df_z}d(E^u_y,E^u_z)\right].
\end{aligned}
$$
The second term inside the brackets is at most
$$
{\rm const}\times C_1 d(z,\mathfs S)^{-1}d(\un v,\un w)^{\chi/7}\leq C_1 d(\un v,\un w)^{\chi/14}
$$
and so 
$$
\abs{(\varphi\circ \pi)(\un v)-(\varphi\circ \pi)(\un w)}\leq C_1 C_2[d(y,z)^{1/2}+d(\un v,\un w)^{\chi/14}]
\leq 2C_1 C_2 d(\un v,\un w)^{\theta}
$$
where $\theta=\min\{1/2,\chi/14\}$.
\end{proof}

\begin{proof}[Proof of Theorem \ref{thm.billiards}]
Any Lyapunov regular point has one positive and one negative Lyapunov exponent.
Therefore, every adapted measure is hyperbolic. 
If $\mu$ is also ergodic, then there is $\chi>0$ such that $\mu[{\rm NUH}^\#_\chi]=1$. 

Now let $\mu_1,\mu_2$ adapted and ergodic. Take $\chi>0$ small such that
$\mu_i[{\rm NUH}^\#_\chi]=1$ for $i=1,2$. By Lemma \ref{lemma.homoclinic.billiards},
$\mu_1,\mu_2$ are homoclinically related (moreover, the single homoclinic relation of
measures contains $\musrb$). By Theorem \ref{thm.criterion}, it follows that there
is at most one adapted measure of maximal entropy. 

Assume that there is an adapted measure of maximal entropy $\mu$.
Since $\mu\sim\musrb$ and $\musrb$ is mixing \cite{Sinai-billiards}, Proposition~\ref{thm.unique-bernoulli}(2) implies
that $\mu$ is Bernoulli and fully supported.

Now we characterize when $\musrb$ is the unique adapted measure of maximal entropy.
We wish to show that this occurs iff $\frac{1}{p} \log\|df^p_x|_{E^u_x}\|$
is constant for every non grazing periodic point $x$ of period $p$, and that in this case
the constant value is equal to $h_{\rm top}(f)$.

Recall that $\varphi$ is the geometric potential, and
let $\pi:\Sigma\to M$ be the map given by Lemma \ref{lemma.holderregularitypotential}.
Let $\overline{\mu}_{\mbox{\tiny SRB}}$ be the lift of $\musrb$ to $\Sigma$, and 
$\overline\vf=\varphi \circ \pi$ be the lift of $\vf$, which is a H\"older continuous potential. 
Since $\musrb$ is Bernoulli, $\Sigma$ is topologically mixing. Observe that 
$P_{\rm top}(\overline\vf)=0$ and that $\overline{\mu}_{\mbox{\tiny SRB}}$ is an equilibrium 
state for $\overline\vf$, since:
\begin{enumerate}[$\circ$]
\item If $\overline\mu$ is $\sigma$--invariant, letting $\mu=\overline\mu\circ\pi^{-1}$ then by 
property (C2)(b) of Theorem \ref{Thm-Main} we have 
$h_{\overline\mu}(\sigma)=h _{\mu}(f)$ and so by Lemma \ref{lemma.SRB=equilibrium}
it follows that $P_{\overline\mu}(\overline\vf)=P_{\mu}(\vf)\leq 0$.
\item $P_{\overline{\mu}_{\mbox{\tiny SRB}}}(\overline\vf)=P_{\musrb}(\vf)=0$.
\end{enumerate}

Assume that $\musrb$ is a measure of maximal entropy for $f$. Then 
$\overline{\mu}_{\mbox{\tiny SRB}}$ is a measure of maximal entropy for $\sigma$.
By \cite[Theorem 1.1]{Sarig-Lecturenotes}, $\overline\vf$ and $-h_{\rm top}(f)$ are cohomologous
and so $\overline\vf_p(x)=-ph_{\rm top}(f)$ for every $x\in\Sigma$ of period $p$. Therefore
$\frac{1}{p} \log\|df^p_x|_{E^u_x}\|=h_{\rm top}(f)$
for every non grazing periodic point $x$ of period $p$.

Reversely, assume that there is $c>0$ such that $\frac{1}{p} \log\|df^p_x|_{E^u_x}\|=c$
for every non grazing periodic point $x$ of period $p$.
Again by \cite[Theorem 1.1]{Sarig-Lecturenotes}, $\overline{\vf}$ is cohomologous to $-c$.
Since $\overline{\mu}_{\mbox{\tiny SRB}}$ is an equilibrium 
state for $\overline\vf$, it is also an equilibrium state for $-c$, i.e. 
it is a measure of maximal entropy for $\sigma$. This implies that $\musrb$ is a measure
of maximal entropy for $f$. Finally, observe that this necessarily implies that
$$
0=P_{\overline{\mu}_{\mbox{\tiny SRB}}}(\overline\vf)=h_{\overline{\mu}_{\mbox{\tiny SRB}}}(\sigma)+
\int\overline\vf d\overline{\mu}_{\mbox{\tiny SRB}}=h_{\rm top}(f)-c
$$
and so $c=h_{\rm top}(f)$.
\end{proof}

\section{Codimension one partially hyperbolic endomorphisms}\label{Section-ph}

In this section we prove Theorem \ref{thm-phcodone}. Suppose that $f$ satisfies the
hypothesis of Theorem \ref{thm-phcodone}. 
The Ledrappier-Young formula for non-invertible maps, proved in \cite{Lin.shu},
states that if $\mu$ is ergodic then
\begin{equation}\label{LY}
h_\mu(f)= F_\mu(f)-\sum_{\lambda_i<0}\lambda_i\gamma_i,
\end{equation}
where the sum is taken over all Lyapunov exponents $\lambda_i<0$, 
$\gamma_i$ is the measure dimension of $\mu$ along stable manifolds, and 
$F_\mu(f)$ is the folding entropy, defined as follows.
Let $\mathcal E$ be the point partition of $M$.

\medskip
\noindent
{\sc Folding entropy:}  
The {\em folding entropy} of $f$ with respect to $\mu$ is
$$
F_\mu(f)=H_\mu(\mathcal E\vert f^{-1}\mathcal E).
$$ 

\medskip
Clearly $F_\mu(f)\leq \log{\rm deg}(f)$.
Since $h_{\rm top}(f)>\log \operatorname{deg}(f)$, there exists $\chi>0$ such that 
every ergodic measure of maximal entropy has center Lyapunov exponent smaller than $-\chi$. 
Define the set
\[
\mathcal{Z}_{\chi}:=\{x\in M: \log\|df^n_x|_{E^c}\| < -\tfrac{n\chi}{2},\ \forall n\geq 0\}.
\]
We prove that $\mathcal Z_\chi$ carries a uniform measure for every measure of maximal entropy.

\begin{lemma}
\label{lemma-uniformzchi}
There exists $\delta>0$ such that $\mu(\mathcal{Z}_{\chi})> \delta$ for every measure
of maximal entropy $\mu$. 
\end{lemma}

The proof of Lemma \ref{lemma-uniformzchi} uses a version of the Pliss Lemma
given in \cite[Lemma 3.1]{crovisier-pujals-2018}, which we state below.

\begin{lemma}[\cite{crovisier-pujals-2018}]
\label{lemma-pliss}
For any $\varepsilon>0$, $\alpha_1 < \alpha_2$ and any sequence $(a_i) \in (\alpha_1, +\infty)^{\N}$
satisfying
$$
\displaystyle \limsup_{n\to +\infty} \frac{a_0 + \cdots + a_{n-1}}{n} \leq \alpha_2,
$$
there exists a sequence of integers $0 < n_1 < n_2 < \cdots $ such that: 
\begin{enumerate}[$\circ$]
\item for any $\ell\geq 1$ it holds
$$
\displaystyle \frac{a_{n_\ell} + \cdots + a_{n-1}}{(n-n_\ell)} \leq \alpha_2 + \varepsilon,
\ \ \forall n>n_\ell;
$$
\item the upper density $\displaystyle \limsup\limits_{\ell\to+\infty} \tfrac{\ell}{n_\ell}$ is larger than
$\delta=\displaystyle \tfrac{\varepsilon}{\alpha_2+ \varepsilon - \alpha_1}$.
\end{enumerate}
\end{lemma}

\begin{proof}[Proof of Lemma \ref{lemma-uniformzchi}]
Let $\mu$ be a measure of maximal entropy for $f$ and let $x\in M$ be generic for $\mu$.
For each $i\geq 0$, let  $a_i = \log \|df_{f^i(x)}|_{E^c}\|$, and let
$$\alpha_1 = \inf_{y\in M} \log\|df_y|_{E^c}\|,\ \ \alpha_2 = -\chi,\ \ \varepsilon = \tfrac{\chi}{2}.
$$
By Lemma \ref{lemma-pliss}, we get that $\mu(\mathcal{Z}_{\chi}) > \delta$.
Observe that $\delta$ does not depend on $\mu$.
\end{proof}

\begin{lemma}\label{lemma-uniformstable}
There exists $\ell_0>0$ such that every $x\in \mathcal{Z}_{\chi}$ has a local stable manifold of size larger
than $\ell_0$. 
\end{lemma}

\begin{proof}
Below, we denote the length of a curve $\gamma$ by $|\gamma|$.
By uniform continuity of $df$ in the projective tangent bundle, there exists a small center cone field
$\mathcal{C}^c$ and a constant $\delta_0>0$ such that:
\[
\left(\begin{array}{l}d(x,y)<\delta_0\\
v\in \mathcal{C}^c_y\textrm{ unitary}\end{array}\right)\ \Longrightarrow\
 \left| \log\|df_x|_{E^c}\| - \log\|df_y v\|\right| < \tfrac{\chi}{4}.
\]
The inequality above implies that  $\|df_y v\| < e^{\frac{\chi}{4}} \|df_x|_{E^c}\|$.
The bundle $E^c$ is locally integrable, i.e. there is $\delta_1>0$ and a continuous family
of curves $\gamma^c_x$ centered at $x$ and tangent to $E^c$ such that $|\gamma^c_x|>\delta_1$ for every $x$.
Take $\ell_0:= \min\{\delta_0, \delta_1\}$, and let $W^c(x,\ell_0)$ be the subset of $\gamma^c_x$ 
centered at $x$ with length $\ell_0$.

We claim that if $x\in \mathcal{Z}_{\chi}$ then $W^c(x,\ell_0)$ is contained in the stable manifold of $x$.
For that, fix $x\in \mathcal{Z}_{\chi}$ and let $\gamma(t)$ be a parametrization of $W^c(x,\ell_0)$. We have 
\[
\|df_{\gamma(t)}\gamma'(t)\| < e^{\frac{\chi}{4}} \|df_x|_{E^c}\| \|\gamma'(t)\|
\]
and so  $|f(W^c(x,\ell_0))|< e^{-\frac{\chi}{4}}\ell_0< \delta_0$. By induction, it follows that
$|f^n(W^c(x,\ell_0))|< e^{-\frac{n\chi}{4}}\ell_0$ for all $n\geq0$, thus proving that 
$W^c(x,r_0)$ contracts exponentially fast and so it is a local stable manifold.
\end{proof}

\begin{proof}[Proof of Theorem \ref{thm-phcodone}]

Let $f$ be a $C^{1+}$ codimension one partially hyperbolic endomorphism with
$h_{\rm top}(f)>\log \deg(f)$. Let $\chi>0$ and $\mathcal Z_\chi$ as above.

\medskip
\noindent
{\sc Claim:} There is $\ell_0>0$ such that every $x\in\mathcal{Z}_\chi$
has local stable and unstable manifold of size larger than $\ell_0$ and their angle at $x$
is larger than $1/\ell_0$. 

\begin{proof}[Proof of claim.]
By Lemma \ref{lemma-uniformstable}, every $x\in\mathcal{Z}_\chi$
has a local stable manifold of uniform size.
Since $f$ is partially hyperbolic, every $\hx\in \hM$ has a local unstable manifold of uniform size
(note: this unstable manifold might not be tangent to $E^u$).
By domination, the angle between center and any unstable direction is uniformly bounded
from below. The claim is proved.
\end{proof}

By the claim, $f$ has at most finitely many homoclinic classes that intersect $\mathcal{Z}_{\chi}$
(if not, a subsequence would accumulate and hence be the same homoclinic class for large indices).
By Lemma \ref{lemma-uniformzchi} and Theorem~\ref{thm.criterion}, it follows that
$f$ has at most finitely many measures of maximal entropy.

Suppose now that for any $C^1$ curve $\gamma$ tangent to the unstable cone 
the union $\bigcup_{n\geq 0} f^n(\gamma)$ is dense in $M$. We wish to show 
that $f$ has a unique measure of maximal entropy. Let $\mu_1,\mu_2$
be two ergodic measures of maximal entropy. Let $\hx_1$ be a generic point for $\hmu_1$
such that $\vt[\hx_1]\in\mathcal{Z}_\chi$, and $\hx_2$ a generic point for $\hmu_2$.
The local unstable manifold $W^u(\hx_2)$ is a $C^1$ curve tangent to the unstable cone, hence
$\bigcup_{n\geq 0} f^n(W^u(\hx_2))$ is dense in $M$ and so there is $k>0$ such that 
$ f^k(W^u(\hx_2))\pitchfork W^s(\hx_1)\neq\emptyset$. Letting $m=-k$, $n=0$ and
$\hx_3=\hf^k(\hx_2)$, this means that $f^{n-m}(W^u(\hf^m(\hx_3)))\pitchfork W^s(f^n(\hx_1))\neq\emptyset$,
which in turn means that $\mathfs V^u(\hx_3)$ and $\mathfs V^s(\hx_1)$ are transversal.
Thus $\hmu_2 \preceq \hmu_1$. By symmetry,  $\hmu_1\preceq \hmu_2$ and so
$\hmu_1,\hmu_2$ are homoclinically related. By Theorem~\ref{thm.criterion},
$\hmu_1=\hmu_2$, hence $\mu_1=\mu_2$.
\end{proof}

\begin{remark}
In an ongoing work, Buzzi, Crovisier and Sarig introduce the notion of
\emph{strong positive recurrence (SPR)} for diffeomorphisms.  Using coding techniques,
they prove that SPR implies that hyperbolic measures of maximal entropy are exponential
mixing for H\"older potentials, up to the period of the system. This notion is characterized
by the existence of a Pesin block with uniform measure for any ergodic measure with ``large enough'' entropy.
It seems a hard task to prove the existence of such uniform Pesin blocks.
For surface diffeomorphisms with positive topological entropy, Buzzi, Crovisier and Sarig 
prove a criterion for SPR in terms of the continuity of Lyapunov exponents, which 
was established in \cite{BCS-2}. 

We believe that the same technique of SPR can be used for endomorphisms.
In our partially hyperbolic setting, we also have continuity of the central exponent,
since the center bundle is one dimensional. In particular, it should not be too difficult to construct uniform Pesin blocks for measures having ``large enough'' entropy in this setting. Therefore,
we believe that the measures of maximal entropy in this setting are exponentially mixing
up to the period of the system. 
\end{remark}

\section{Non-uniformly hyperbolic volume-preserving endomorphisms}\label{Section-ACS}

In this section we prove Theorems~\ref{thm.unique.ACS} and \ref{thm.ergodicity}.
The condition \eqref{eq.pastexpansion} introduced in \cite{ACS} is similar to a property 
considered in random dynamics, see \cite{Dolgopyat-Krikorian,ZHANG}.
The ideas below related to estimating the sizes of unstable manifolds and how their angles
vary go back to the work of Dolgopyat and Krikorian for random dynamical systems~\cite[Corollary 4 and Section 11]{Dolgopyat-Krikorian}, see also \cite[Section 3.2]{ZHANG}. 
In our context, the pre-images correspond to the ``random'' part of the dynamics, and so we
apply the ideas of ~\cite{Dolgopyat-Krikorian} to the disintegrated measures of the Lebesgue measure
on the pre-images.

We start introducing some notation from \cite{ACS}. In this section, $S$ denotes a smooth closed surface. 
For $x\in S$, $v\in T_x S\setminus\{0\}$ and $N\in \mathbb{N}$ define 
$$
I(x,v,f^N)=\sum_{f^N(y)=x} \frac{\log \norm{(df^N_y)^{-1}v}}{\abs{\det(df^N_y)}}\,,
$$
and let 
$$
C(f)=\sup_{n\in \N}\frac{1}{n}\left(\inf_{x\in S\atop{v\in T_xS,\|v\|=1}} I(x,v,f^n)\right)=
\lim_{n\to\infty}\frac{1}{n}\left(\inf_{x\in S\atop{v\in T_xS,\|v\|=1}} I(x,v,f^n)\right),
$$
see \cite[Corollary 2.1]{ACS}. We fix a reference volume probability measure $\mu$ on $S$ and let
$$
\mc U=\{f:S\to S:f\text{ is a $C^1$ endomorphism that preserves }\mu\text{ and }C(f)>0\}.
$$
For every $f\in\mathcal U$ the measure $\mu$ is hyperbolic, with one positive and
one negative Lyapunov exponent \cite[Theorem A]{ACS}.
In the remaining of this section, we fix $f\in \mc U$ of class $C^{1+}$.

As introduced in Sections \ref{Section-natural-extension} and \ref{ss.proof.ALP},
let $\wh{\mu}$ be the lift of $\mu$ to the natural extension $\wh S$ and $(\wh{df}^{(n)}_{\hx})_{n\in\Z}$
be the invertible cocycle over $\wh f$ induced by $\left(df^n_x\right)_{n\geq 0}$.
By the Oseledets theorem, $\hmu$--a.e. $\hx\in\wh S$ has stable/unstable direction
$E^{s/u}_{\hx}$. 

If $\mathcal E$ is the point partition of $S$,
then $\vartheta^{-1}\mathcal E$ is a measurable partition on $\wh S$, and so 
by the Rokhlin disintegration theorem we can write
$$
\wh\mu=\int \mu_x^-d\mu(x)
$$ 
where $\mu_x^-$ is a probability measure on $\vartheta^{-1}(x)$, the set of pre-orbits
of $x$. By \cite[Proposition~2.1]{ACS}, $\mu^-_x$ is the unique measure on $\vt^{-1}(x)$
such that
$$
\mu^-_x\{\hx\in \vartheta^{-1}(x):x_{-n}=z\}=\abs{\det df^n_z}^{-1}.
$$
Observe that $x\in M\mapsto \mu_x^-$ is globally defined and continuous.

 With respect to the above disintegration, the condition $C(f)>0$ is equivalent to existing $N,c>0$ such that
$$
I(x,v,f^n)=\int \log \norm{\wh{df}^{-n}_{\hx} v}d\mu^-_x(\hx)\geq n c>0,\ \ \forall n\geq N,\forall x\in S,\forall v\in T_xS \text{ unitary},
$$
see \cite[Section 2.3]{ACS}.

\subsection{Uniform estimates of unstable manifolds}\label{ss.unstable.manifolds}

Observe that, as rewritten above, $C(f)>0$ means that we have
an expansion for the past, on the average, for every unitary direction $v$.
We will use this condition to obtain properties about the {\em unstable} directions/manifolds.
The first property is that for a generic $x\in S$ the unstable direction varies as we pre-iterate $x$
by the different inverse branches. Below, we denote the projective space of $TxS$ by $\mathbb P T_xS$.

\begin{theorem}\label{thm.unstable.angle}
There exist $A,\beta>0$ such that the following holds for $\mu$--a.e. $x\in S$:
for any $E\in \mathbb{P}T_xS$ and any $\eta>0$ it holds $\mu^-_x\lbrace \hx: \angle (E,E^u_{\hx})<\eta \rbrace\leq A \eta^\beta$.
\end{theorem}

This is a version of~\cite[Corollary 4, item (b)]{Dolgopyat-Krikorian} in our setting.
The second property is a uniform estimate on the sizes and geometry of unstable manifolds on measure, 
independent of $x$. Recall that $\hmu$--a.e. $\hx\in\wh S$ has stable/unstable local curves $W^{s/u}(\hx)$.
We wish to control the curvature of $W^u(\hx)$, but since it might not be $C^2$ we 
consider the following notion.
Given $x\in S$, recall that $\exp{x}:T_xS\to S$ is the exponential map of $S$ at $x$.
Given $v\in T_xS$, we can identify $T_xS=\mathbb R v\oplus \mathbb Rv^\perp$ with $\mathbb R^2$.
Let $L>0$.

\medskip
\noindent
{\sc $L$--Lipschitz graph:}
A curve $\gamma:[-a,a]\to M$ is called a {\em $L$--Lipchitz graph centered at $\gamma(0)$} 
if for the identification $T_xS=\mathbb R \gamma'(0)\oplus \mathbb R\gamma'(0)^\perp\cong \mathbb R^2$
there exists a $L$--Lipschitz function $F:[-b,b]\to \mathbb R$ such that
$\gamma[-a,a]=\exp{x}\{(t,F(t)):t\in [-b,b]\}$.

\medskip
We denote the length of $W^{s/u}(\hx)$ by $|W^{s/u}(\hx)|$.
We prove the following result.

\begin{theorem}\label{thm.big.unstable}
For every $\sigma>0$, there is $\ell_0>0$ such that for $\mu$--a.e. $x\in S$:
$$
\mu^-_x\{\hx:W^u(\hx)\text{\emph{ is a $1$--Lipschitz graph with }}|W^u(\hx)|>\ell_0\}>1-\sigma.
$$
\end{theorem}

To prove the above theorems, we start with some preliminary results.

\begin{lemma}\label{lemma.integral.bound}
For all $s>0$ small, there are $\chi,C>0$ depending on $s$
such that
$$
\int \norm{\wh{df}^{-n}_{\hx} v}^{-s} d\mu^-_x( \hx)<C e^{-n\chi},
\ \ \forall n>0,\forall x\in S,\forall v\in T_xS \text{\emph{ unitary}}.
$$
\end{lemma}

\begin{proof}
Take $m$ large enough so that 
$\displaystyle\int \log \norm{\wh{df}^{-m}_{\hx}v}^{-1}d\mu^-_x(\hx)<-1$ for all $x\in S$ and $v\in T_x S$ unitary.
Using the inequality $e^x\leq 1+x +\frac{x^2}{2} e^{\abs{x}}$, we have 
$$
\begin{aligned}
&\int \norm{\wh{df}^{-m}_{\hx} v}^{-s} d\mu^-_x( \hx)=\int e^{s \log \norm{\wh{df}^{-m}_{\hx}v}^{-1}} d\mu^-_x(\hx)\\
&\leq \int \left(1+s\log \norm{\wh{df}^{-m}_{\hx}v}^{-1}+
\frac{s^2}{2} e^{\left|s \log \norm{\wh{df}^{-m}_{\hx}v}\right|}\log^2 \norm{\wh{df}^{-m}_{\hx}v}^{-1}\right) 
 d\mu^-_x(\hx)\\
&\leq 1-s +K s^2
\end{aligned}
$$
for a constant $K>0$ that only depends on $m$ and $f$ (here, we use that
$\|\wh{df}^{-m}_{\hx}v\|$ is uniformly bounded from 0 and $\infty$).
Let $s_0>0$ such that 
$\kappa=\kappa(s):=1-s+Ks^2<1$ for all $s\in (0,s_0]$. In particular, 
for $s$ in this domain we have
\begin{equation}\label{eq.n=m}
\int \norm{\wh{df}^{-m}_{\hx} v}^{-s} d\mu^-_x( \hx)<\kappa,\ \ \forall x\in S,\forall v\in T_xS \text{ unitary}.
\end{equation}
Now we prove the analogous estimate for $n=km$. We do the case $n=2m$ (the general case is analogous).
We have
\begin{align*}
&\int \norm{\wh{df}^{-2m}_{\hx} v}^{-s} d\mu^-_x( \hx)
=\sum_{f^{2m}(y)=x}\frac{\norm{(df^{2m}_y)^{-1}v}^{-s}}{\abs{\det(df^{2m}_y)}}\\
&=\sum_{f^m(z)=x}\sum_{f^m(y)=z}\frac{\norm{(df^{m}_y)^{-1}(df^{m}_z)^{-1}v}^{-s}}{\abs{\det(df^{m}_y)\det(df^{m}_z)}}\\
&=\sum_{f^m(z)=x}\sum_{f^m(y)=z}\frac{\norm{(df^{m}_y)^{-1}v_z}^{-s}\norm{(df^{m}_z)^{-1}v}^{-s}}{\abs{\det(df^{m}_y)\det(df^{m}_z)}}\\
&=\sum_{f^m(z)=x}\frac{\norm{(df^{m}_z)^{-1}v}^{-s}}{\abs{\det(df^{m}_z)}} 
\sum_{f^m(y)=z}\frac{\norm{(df^{m}_y)^{-1}v_z}^{-s}}{\abs{\det(df^{m}_y)}},
\end{align*}
where $v_z=\tfrac{(df^{m}_z)^{-1}v}{\norm{(df^{m}_z)^{-1}v}}\in T_yS$ is unitary.
By estimate (\ref{eq.n=m}), it follows that
\begin{align*}
&\int \norm{\wh{df}^{-2m}_{\hx} v}^{-s} d\mu^-_x( \hx)
\leq \kappa\sum_{f^m(z)=x}\frac{\norm{(df^{m}_z)^{-1}v}^{-s}}{\abs{\det(df^{m}_z)}}\leq \kappa^2.
\end{align*}
Writing $\kappa=e^{-m\chi}$ and letting 
$C=[\inf m(df_x)]^{-ms}$, the result follows.
\end{proof}

As a direct consequence, we obtain the following.

\begin{corollary}\label{cor.large.dev} 
For any $0<\overline{\chi}<\chi$, $x\in S$ and $v\in T^1_x S$ unitary we have
$$
\mu^-_x\left\lbrace \hx: \norm{\wh{df}^{-n}_{\hx}v}^{-s}\geq e^{-n\overline{\chi}}\right\rbrace<Ce^{-n(\chi-\overline{\chi})}.
$$
Hence, for $\mu^-_x$--a.e. $\hx$ there is $n(\hx)>0$ such that
$\norm{\wh{df}^{-n}_{\hx}v}^{-s}<e^{-n\overline{\chi}}$, $\forall n\geq n(\hx)$.
\end{corollary}

\begin{proof}
The first claim follows from the Markov inequality and the second from the Borel-Cantelli lemma. 
\end{proof}

Now we are able to prove Theorem~\ref{thm.unstable.angle}.

\begin{proof}[Proof of Theorem~\ref{thm.unstable.angle}]
We consider $x\in S$ such that $\mu^-_x$--a.e. point $\hx\in \vartheta^{-1}(x)$ is Oseledets regular.
This defines a set of full $\mu$--measure. Fix one such $x$, and let $E\in \mathbb{P}T_x M$
be an arbitrary direction. Fix  a unit vector $v\in E$. 
We will prove the lemma in two steps. Firstly, we will define for $\mu^-_x$--a.e.
$\hx$ a parameter $\eta(\hx)>0$
such that
$$
\angle(w,v)\leq\eta(\hx)\ \Longrightarrow\ \norm{\wh{df}^{-n}_{\hx}w}\to\infty \text{ as }n\to\infty.
$$
By the Oseledets Theorem, the above condition implies that $\angle(E,E^u_{\hx})>\eta(\hx)$, since $E^u_{\hx}$ contracts exponentially fast for the past.  So
$$
\mu^-_x\lbrace \hx:\angle (E,E^u_{\hx})<\eta \rbrace\leq \mu^-_x\lbrace \hx:\eta(\hx)<\eta \rbrace.
$$
The second step consists of estimating this latter measure.

Fix $\hx\in \vartheta^{-1}(x)$. To define $\eta(\wh x)$, we will represent
$\wh{df}^{-n}_{\hx}$ in a suitable system of coordinates. Write $E_n:=\wh{df}^{-n}_{\hx}E$,
and consider the decomposition $E_n\oplus E_n^\perp$. 
In the sequel, write $x_{-j}=\hf^{-j}(\hx)$. Then
the derivative $\wh{df}^{-\ell}_{x_{-j}}:E_j\oplus E^\perp_j\to E_{j+\ell}\oplus E_{j+\ell}^\perp$
equals
$$
\wh{df}^{-\ell}_{x_{-j}}=\left[\begin{array}{cc}
\lambda^\ell_j & C^\ell_j\\
 & \\
0 & d^\ell_j/\lambda^\ell_j
\end{array}\right]
$$
where, $\lambda^\ell_j=\norm{\wh{df}^{-\ell}_{x_{-j}}\mid_{E_j}}$, $d^\ell_j= \det \wh{df}^{-\ell}_{x_{-j}}$
and $|C^\ell_j |\leq \sup\|df^{-1}\|^\ell$.
Decompose $\wh{df}^{-\ell}_{x_{-j}}=D^\ell_j+U^\ell_j$, where
$$
D^\ell_j=\left[\begin{array}{cc}
\lambda^\ell_j & 0\\
0 & d^\ell_j/\lambda^\ell_j
\end{array}\right]
\text{\ and \ }
U^\ell_j=\left[\begin{array}{cc}
0 & C^\ell_j\\
0 & 0
\end{array}\right].
$$
Noting that the composition of two matrices of the form $U_j^\ell$ is zero, we have
$$
\wh{df}^{-n}_{\hx}=D^n_0+\sum_{\ell=0}^{n-1}D^{n}_{\ell+1}U_\ell^{\ell+1} D^\ell_0=D_0^n+\tilde{U}_0^n,
$$
where  $\tilde{U}^n_0:=\sum_{\ell=0}^{n-1}D^{n}_{\ell+1}U_\ell^{\ell+1} D^\ell_0$.
Calculating the last sum, we obtain that $C_n:=C_0^n$ equals
$$
C_n=\sum^{n-1}_{\ell=0} \frac{\lambda^{n}_{\ell+1}C_\ell^{\ell+1} d^\ell_0}{\lambda^\ell_0}
$$
Let $L:=\sup_{x\in M}\norm{(df_x)^{\pm 1}}$ and note that as $f$ is volume preserving 
$|\det df^{-1}|<1$, then we have
$$
\abs{C_n}\leq L\sum^{n-1}_{\ell=0} \frac{\lambda^{n}_{\ell+1}}{\lambda^\ell_0}\leq 
\lambda_0^n L^2\sum^{n-1}_{\ell=0} \left(\frac{1}{\lambda^\ell_0}\right)^2.
$$
If $w=(1,\eta)\in E\oplus E^\perp$, then
$$
\wh{df}^{-n}_{\hx}w=\left[\begin{array}{cc}
\lambda^n_0 & C_n\\
 & \\
0 & d^n_0/\lambda^n_0
\end{array}\right]
\begin{bmatrix}1 \\ \\ \eta \end{bmatrix}=
(\lambda^n_0+\eta C_n,\eta d^n_0/\lambda^n_0).
$$ 
Define
$$
\eta(\hx):=\frac{1}{2L^2\sum_{\ell\geq 0}\left(\tfrac{1}{\lambda_0^\ell}\right)^2}\cdot
$$ 
Hence,
if $|\eta|<\eta(\hx)$ then $|\lambda^n_0+\eta C_n|\geq \lambda_0^n/2$ goes to infinity
as $n\to\infty$, and so $\norm{\wh{df}^{-n}_{\hx}w}\to\infty$ as $n\to\infty$. This concludes the first step of the proof.

Now we estimate $\eta(\hx)$ from below in terms of the hyperbolicity of $\hx$.
Fix $0<\overline{\chi}<\chi$. For each $\hx$,
let $B(\hx)=\{\ell\geq 0: \|\wh{df}^{-\ell}_{\hx}v\|^{-s}\geq e^{-\ell\overline\chi}\}$
and $\ell(\hx)=\sup B(\hx)$. We have:
\begin{enumerate}[$\circ$]
\item If $\ell\not\in B(\hx)$ then $\lambda_0^\ell > e^{\ell \overline{\chi}/s}$
\item If $\ell\in B(\hx)$ then $\lambda_0^\ell\geq L^{-\ell}$.
\end{enumerate}
Hence
\begin{align*}
&\sum_{\ell\geq 0}\left(\tfrac{1}{\lambda_0^\ell}\right)^2=
\sum_{\ell\in B(\hx)}\left(\tfrac{1}{\lambda_0^\ell}\right)^2+
\sum_{\ell\not\in B(\hx)}\left(\tfrac{1}{\lambda_0^\ell}\right)^2\leq
\sum_{\ell\in B(\hx)}L^{2\ell}+
\sum_{\ell\not\in B(\hx)}e^{-\ell\left(\frac{2\overline{\chi}}{s}\right)}\\
&\leq \sum_{\ell=0}^{\ell(\hx)}L^{2\ell}+S=\frac{L^{2\ell(x)+2}}{L^2-1}+S\leq \frac{L^{2\ell(x)+2}(S+1)}{L^2-1}
\end{align*}
where $S=\sum_{\ell\geq 0}e^{-\ell\left(\frac{2\overline{\chi}}{s}\right)}$. Therefore,
$$
\eta(\hx)\geq \frac{L^2-1}{2L^{2\ell(\hx)+4}(S+1)}=A_0 L^{-2\ell(x)},
$$
where $A_0=\tfrac{L^2-1}{2L^4(S+1)}$. Define $\ell_0:=\tfrac{\log\left(A_0/\eta\right)}{2\log L}$.
If $\ell(\hx)\leq \ell_0$, then $\eta(\hx)\geq \eta$, and so
$\mu^-_x\lbrace \hx:\eta(\hx)<\eta \rbrace\leq \mu^-_x\lbrace \hx:\ell(\hx)>\ell_0\rbrace$. By Corollary \ref{cor.large.dev},  the latter measure is at most
$$
\sum_{\ell>\ell_0}Ce^{-\ell(\chi-\overline{\chi})}<C\frac{e^{-\ell_0(\chi-\overline{\chi})}}{1-e^{-(\chi-\overline{\chi})}}
= \left(\tfrac{CA_0^{-\beta}}{1-e^{-(\chi-\overline{\chi})}}\right)\eta^\beta
$$
for $\beta=\tfrac{\chi-\overline{\chi}}{2\log L}>0$. Letting $A=\tfrac{CA_0^{-\beta}}{1-e^{-(\chi-\overline{\chi})}}$,
we conclude that
$$
\mu^-_x\lbrace \hx:\angle (E,E^u_{\hx})<\eta \rbrace\leq \mu^-_x\lbrace \hx:\eta(\hx)<\eta \rbrace\leq 
\mu^-_x\lbrace \hx:\ell(\hx)>\ell_0\rbrace\leq A\eta^\beta.
$$
The proof is complete.
\end{proof}

The next step is to prove Theorem~\ref{thm.big.unstable}.
For this, we follow the presentation of \cite[Lemma~4]{ZHANG}, which in turn is 
an analogue of \cite[Corollary 4(a),(c)]{Dolgopyat-Krikorian}.

In the sequel, we fix $0<\overline\chi<\chi$, $s>0$ satisfying Lemma \ref{lemma.integral.bound}
and write $\lambda=\overline{\chi}/s$. We also fix $x\in S$ such that $\mu^-_x$-a.e.
$\hx$ is Oseledets regular. 

We define a fibered version of Pesin sets at $x$ as follows. Let $C,\ve>0$ and $E\in \mathbb{P}T_x M$.
Writing $E_k({\hx})=\wh{df}^{-k}_{\hx}E$ and $E^u_k(\hx)=E^u_{\hf^{-k}(\hx)}$, let 
$\Lambda^-_{\varepsilon,C,E}(x)$ be the set of $\hx\in \vartheta^{-1}(x)$ such that for all $n,k\geq 0$ it holds:
\begin{enumerate}[(1)]
\item $\norm{\wh{df}^{-n}_{\hf^{-k}(\hx)}\mid_{E^u_k(\hx)}}\leq C e^{-n\lambda+k\varepsilon}$,
\item $\norm{\wh{df}^{-n}_{\hf^{-k}(\hx)}\mid_{E_k(\hx)}}\geq C^{-1} e^{n\lambda-k\varepsilon}$,
\item $\angle (E_k(\hx),E^u_k(\hx))\geq C^{-1}e^{-k\varepsilon}$.
\end{enumerate} 
Compare this with the definition of Pesin sets in Section~\ref{ss.hom.relation}: now, 
we control just the past behavior along $E^u$ and $E$ for pre-orbits of $x$.

\begin{lemma}\label{lemma.unstable.pesin}
For all $\sigma>0$, there are $C,\varepsilon>0$ such that for $\mu$--a.e. $x\in S$
the following holds: for every $E\in \mathbb PT_x S$ it holds 
$\mu^-_x(\Lambda^-_{\varepsilon,C,E}(x))>1-\sigma$.
\end{lemma}

\begin{proof}
We estimate the measure of points not satisfying the
properties (1)--(3).
Fix $x\in S$ satisfying Theorem~\ref{thm.unstable.angle} and $E\in \mathbb{P}T_x S$.
We start estimating the measure of points not satisfying property (3). Let $\hx\in \vt^{-1}(x)$. 
For $k>0$, the $\hf$--invariance of $\hmu$ implies that 
$\mu^-_{x_k}=\mu^-_x\circ \hf^{-k}$. Note that $E_k(\hx)$ only depends on $x_k$, so
we denote it by $E_k(x_k)$. Recalling that $\vt_k:\wh S\to S$ is the projection into
the $k$--th coordinate (see Section~\ref{Section-natural-extension}), for each $x_k$ such that $f^k(x_k)=x$ the invariance and
Theorem~\ref{thm.unstable.angle} imply that
\begin{align*}
&\ \mu^-_x\left\{\hx\in \vt^{-1}_k(x_k):\angle (E_k(\hx),E^u_k(\hx))\geq C^{-1}e^{-k\varepsilon}\right\}\\
&=\mu^-_{x_k}\left\{\hy\in \vt^{-1}(x_k):\angle (E_k(x_k),E^u(\hy))\geq C^{-1}e^{-k\varepsilon}\right\}\leq 
A C^{-\beta}e^{-k\beta \varepsilon}.
\end{align*}
Since, for each $k>0$, $\mu^-_k$ is an average of the measures $\mu^-_{x_k}$ with $f^k(x_k)=x$, we conclude that
$$
\mu^-_x\left\{\hx\in \vt^{-1}(x):\angle (E_k(\hx),E^u_k(\hx))\geq C^{-1}e^{-k\varepsilon}\right\}
\leq A C^{-\beta}e^{-k\beta \varepsilon} .
$$
Summing up in $k>0$ and choosing $C>0$ large, we obtain that the $\mu^-_x$--measure of points not
satisfying (3) is less than $\sigma/3$. Now we focus on the other properties.

\medskip
\noindent
{\sc Claim:} Let $\hx\in\vt^{-1}(x)$ and assume that
$\norm{\wh{df}^{-n}_{\hf^{-k}(\hx)}\mid_{E^u_k(\hx)}}>C e^{-n\lambda+k\varepsilon}$.
Then for any $V\in \mathbb{P}T_{x_k}M$ and any $\varepsilon_0\leq \varepsilon$ 
at least one of the two conditions below hold:
\begin{enumerate}[(1.1)]
\item[(1.1)] $\angle (E^u_{k+n}(\hx), \wh{df}^{-n}_{\hf^{-k}(\hx)}V)< 2C^{-1/2}e^{-(n+k)\varepsilon_0}$.
\item[(1.2)] $\norm{\wh{df}^{-n}_{\hf^{-k}(\hx)}\mid_V}^{-1} >
C^{1/2} e^{-n \lambda+k\varepsilon-(n+k)\varepsilon_0}$.
\end{enumerate}

\begin{proof}[Proof of claim]
Write $V_n=\wh{df}^{-n}_{\hf^{-k}(\hx)}V$. Let 
$P_E$ denote the projection onto a subspace $E$.
If the first condition fails, then for $v\in E^u_k(\hx)$ unitary we have
\begin{align*}
&\ \norm{P_{V_n^\perp} \wh{df}^{-n}_{\hf^{-k}(\hx)}P_{V^\perp}v}=\norm{P_{V_n^\perp} \wh{df}^{-n}_{\hf^{-k}(\hx)}v}
\geq C e^{- n \lambda+k\varepsilon} \sin \angle( E^u_{k+n}(\hx),V_n) \\
&\\
&\geq \tfrac{2}{\pi} Ce^{- n \lambda+k\varepsilon} \angle( E^u_{k+n}(\hx),V_n)
>  C^{1/2} e^{-n\lambda+k\varepsilon -(n+k)\varepsilon_0}.
\end{align*}
Since
$$
\norm{P_{V_n^\perp} \wh{df}^{-n}_{\hf^{-k}(\hx)}P_{V^\perp}}=
\norm{\wh{df}^{-n}_{\hf^{-k}(\hx)}\mid_V}^{-1} \abs{\det\left(\wh{df}^{-n}_{\hf^{-k}(\hx)}\right)}
$$
and $\abs{\det(df)}>1$, we conclude that the second condition holds.
\end{proof}

We apply the claim with $\varepsilon_0=\varepsilon/2$ to bound the measure
of points not satisfying condition (1).
By Theorem~\ref{thm.unstable.angle}, the $\mu^-_x$--measure of points satisfying (1.1) 
is at most $4^\beta A C^{-\beta/2}e^{-(n+k)\beta \varepsilon_0}$. 
By Lemma~\ref{lemma.integral.bound}
and the Markov inequality, the measure of points satisfying (1.2) is at most
$$
\frac{{\rm const}\times e^{-n\chi}}{(C^{1/2} e^{-n \lambda+k\varepsilon-(n+k)\varepsilon_0})^s}
= {\rm const}\times C^{-s/2}e^{-n (\chi-\overline{\chi}+s\ve_0)-ks\varepsilon_0}.
$$
Summing up these two estimates in $n,k\geq 0$, we obtain that the measure of points not satisfying (1)
is bounded by ${\rm const}\times(C^{-\beta/2}+C^{-s/2})$, where the constant does not depend on $C$.
Taking $C$ large enough, this number is less than $\sigma/3$.

Again by Lemma~\ref{lemma.integral.bound} and the Markov inequality, the measure of points not satisfying (2)
for fixed $n,k$ is at most 
$$
\frac{{\rm const}\times e^{-n\chi}}{(Ce^{-n \lambda+k\varepsilon})^s}
= {\rm const}\times C^{-s}e^{-n (\chi-\overline{\chi})-ks\varepsilon}.
$$
Summing up in $n,k$, the measure of points not satisfying (2) is bounded by
${\rm const}\times C^{-s}$, where the constant does not depend on $C$.
This number is smaller than $\sigma/3$ for large $C$. The proof of the lemma is complete.
\end{proof}

Recalling the definition of $L$--Lipschitz graph in the beginning of Section~\ref{ss.unstable.manifolds},
we now prove Theorem~\ref{thm.big.unstable}.

\begin{proof}[Proof of Theorem~\ref{thm.big.unstable}]
The conditions defining $\Lambda^-_{\varepsilon,C,E}(x)$ allow to apply graph transform methods
from \cite[Theorem~2.1.1]{Pesin-Izvestia-1976} and conclude that every $\hx\in\Lambda^-_{\varepsilon,C,E}(x)$
has a local unstable manifold that is a 1--Lipchitz graph with uniform size (depending only
on $\varepsilon,C,E$). By Lemma~\ref{lemma.unstable.pesin}, we conclude the proof.
\end{proof}

\subsection{Ergodicitity}

In this section we prove Theorem~\ref{thm.ergodicity}.
We begin with the following remark. Using the Hopf argument can be delicate in the non-invertible setting.
The stable lamination has absolutely continuous holonomies,  like in the invertible case,
but we cannot talk about unstable holonomies because they do not form a lamination:
for each $x\in S$, the unstable manifolds depend on the choice of a pre-orbit of $x$.
Fortunately, there is another form of absolute continuity \cite{qian2002srb}, which is 
the existence of a {\em $u$--subordinated partition} $\wh{\mathfs{P}}$ of $\wh{S}$, i.e.
a measurable partition such that for $\hmu$--a.e. $\hx$:
\begin{enumerate}[$\circ$]
\item $\wh{\mathfs{P}}(\hx)$ is an open subset of $V^u(\hx)$, and 
\item the projection $\vartheta:\wh{\mathfs P}(\hx)\to W^u(\hx)$ is injective.
\end{enumerate}
Letting $\hmu=\int \hmu_{\hx}^{\wh{\mathfs P}}$ be the decomposition given by the Rokhlin decomposition
theorem, it follows that the measure $\wh\mu^{\wh{\mathfs{P}}}_{\hx}\circ \vt^{-1}$
is absolutely continuous with respect to the Lebesgue measure of $W^u(\hx)$ for
$\hmu$--a.e. $\hx\in \wh{S}$. 

Let $\psi:S \to \mathbb{R}$ be a continuous function, and let $\wh\psi:\wh S\to \mathbb R$ 
be $\wh\psi=\psi\circ\vt$, which is also continuous. Denote by $\wh\psi^{+/-}(\hx)$
the forward/backward limit of the Birkhoff average of $\wh\psi$ at $\hx$ with respect to
$\hf$, when the limit exists.

\begin{lemma}\label{abs.cont}
Let $\psi:S\to \mathbb{R}$ be continuous. For $\hmu$--a.e. $\hx$ the following holds:
$\psi^+(x) = \psi^+(y)$ for Lebesgue almost every $y\in W^u(\hx)$.
\end{lemma}

Above, $x=\vt[\hx]$ and $y=\vt[\wh y]$ (observe that $\wh\psi^+(\hx)=\psi^+(x)$ only depends on $x$).

\begin{proof}
Let $B_1=\{\wh x\in \wh S: \exists \psi^{\pm}(\hx)\text{ and }\wh\psi^-(\hx)=\wh\psi^+(\hx)\}$, which has full 
$\hmu$--measure. Let $\wh{\mathfs P}$ be a $u$--subordinated measurable partition as above.
The set $B_2=\{\hx\in B_1:\hmu^{\wh{\mathfs{P}}}_{\hx}\textrm{--a.e. }\hy\in B_1\}$ also
has full $\hmu$--measure. We also
know that  $\wh\psi^-(\wh x)=\wh\psi^-(\wh y)$ for all $\hy\in W^u(\hx)$, since $\psi$ is continuous. 
Hence, for $\hx\in B_2$ it holds
\[
\wh\psi^+(\vartheta[\hx]) =\wh\psi^+(\hy) = \wh{\psi}^-(\hy)=\wh{\psi}^-(\hx)=\wh{\psi}^+(\hx) = \psi^+(\vartheta[\hx])
\ \text{ for }\hmu^{\wh{\mathfs{P}}}_{\hx}\textrm{--a.e. }\hy\in \wh{\mathfs P}(\hx).
\]
Now repeat the argument to the partitions $\hf^n\mathfs P$, $n\geq 0$. Since 
$W^u(\hx)$ is contained in $\bigcup_{n\geq 0}f^n(\vt[\wh{\mathfs P}(\hf^{-n}(\hx))])$, the result follows.
\end{proof}

\begin{proof}[Proof of Theorem~\ref{thm.ergodicity}]

Let $K \subset S$ be a Pesin block for the stable manifolds, i.e. every $x\in K$ has
a local stable manifold $W^s_{\loc}(x)$ of uniform size and $x\in K\mapsto W^s_{\loc}(x)$
is $C^1$. We may assume that $K$ has positive volume.
Let $x_0\in K$ be a density point of $K$, and fix $U$ a small neighborhood of $x_0$
such that every $W^s_{\loc}(x)$, $x\in U\cap K$, is large with respect to $U$.
Let $\mathfs L=\{W^s_{\loc}(x):x\in U\cap K\}$ be the stable lamination, which we identify
with the union $\bigcup_{x\in U\cap K} W^s_{\loc}(x)$.
Since $\mathfs L$ is absolutely continuous, if $D$ is a disc transverse to every leaf of $\mathfs L$
then $D \cap \mathfs L$ has positive volume inside $D$. 

Fix $\psi:S\to\mathbb R$ continuous, and let $G$ be the set of points $x\in S$ satisfying
Theorems~\ref{thm.unstable.angle} and \ref{thm.big.unstable} and such that 
$\mu^-_x$--a.e. $\hx$ satisfies Lemma \ref{abs.cont}. Clearly $\mu(G)=1$.
We wish to show that if $x,y\in U\cap G$ then $\psi^+(x)=\psi^+(y)$. 

By Theorems~\ref{thm.unstable.angle} and \ref{thm.big.unstable}, there are
$\hx\in \vt^{-1}[x]$ and $\hy\in\vt^{-1}[y]$ satisfying Lemma \ref{abs.cont} such that
$W^u(\hx),W^u(\hy)$ are large and intersect every leaf of $\mathfs L$ transversaly. We thus obtain a {\em holonomy map} $H:W^u(\hx)\to W^u(\hy)$.
This map is absolutely continuous, hence there are $z\in W^u(\hx)$
with $\psi^+(x)=\psi^+(z)$ and $w\in W^u(\hy)$ with $\psi^+(y)=\psi^+(w)$ such that 
$H(y)=z$. This latter equality gives $\psi^+(z)=\psi^+(w)$, and so
$$
\psi^+(x)=\psi^+(z)=\psi^+(w)=\psi^+(y).
$$

This implies that  $x,y$ belong to the same ergodic component of $\mu$. Since we can take $K$
with measure arbitrarily close to 1, it follows that $\mu$ has an ergodic component of full volume in $U$.
This implies that $\mu$ has at most countably many ergodic components.
By contradiction, suppose that $\mu$ is not ergodic. Then there are disjoint open sets $U_1,U_2\subset S$
and distinct ergodic components $\mu_1,\mu_2$ such that Lebesgue--a.e. $x\in U_i$ is typical for $\mu_i$, $i=1,2$.
By transitivity, there is $n>0$ such that $V=f^{-n}(U_1)\cap U_2$ is a non-empty open set.
Then Lebesgue--a.e. $x\in V$ is typical for $\mu_1$ and $\mu_2$, a contradiction. Hence $\mu$ is ergodic.

Now assume that $\pm1$ is not an eigenvalue of the linear part of $f$. Since the same
holds for $f^n$, it follows from \cite{Andersson-trans} that $f^n$ is transitive for every $n\geq 1$.
Applying the same proof of ergodicity above, we obtain that $\mu$ is ergodic for $f^n$ for all $n\geq 1$.

The geometrical potential $\vf(\hx)=\log \|df_{\hx}|_{E^u_{\hx}}\|$ is admissible\footnote{The potential
$\vf$ is actually defined on $\hM$ and not on $M$, but the same proof of Proposition~\ref{thm.unique-bernoulli}(2)
works in this case, since ergodicity is preserved in the natural extension.} in the sense 
of Proposition~\ref{thm.unique-bernoulli}, by property (C3). By the Pesin entropy
formula, $\hmu$ is an equilibrium state for the pair $(\hf,\vf)$. Applying Proposition~\ref{thm.unique-bernoulli}(2)
with $\wh\nu=\hmu$, it follows that $\hmu$ is Bernoulli.
\end{proof}

\subsection{Uniqueness of the measure of maximal entropy}

Now we prove Theorem~\ref{thm.unique.ACS}. The idea is to construct
$\mathcal U_t$ such that every $f\in \mathcal U_t$ of class $C^{1+}$
has large stable manifolds, and so the lift of every measure of maximal entropy is
homoclinically related to the lift $\hmu$ of the Lebesgue measure. The proof of the homoclinic relation
uses the ``dynamical'' Sard's theorem of \cite{BCS-Annals}.

Recall the definition of $L$--Lipschitz graph in the beginning of Section~\ref{ss.unstable.manifolds}.
We first prove the following result.

\begin{theorem}\label{thm.large.uniqueness}
Let $f\in \mathcal{U}$ of class $C^{1+}$, and assume there are $\ell_0,L>0$
such that for $\mu$--a.e. $x$ the local stable manifold $W^s(x)$ contains a $L$--Lipschitz graph
centered at $x$ of size larger than $\ell_0$. Then $f$ has at most one hyperbolic measure
of maximal entropy.
\end{theorem}

\begin{proof}
Assume that $f$ has a hyperbolic measure of maximal. Let $\nu$ be one such measure,
and assume it is ergodic.
We prove that $\hmu,\wh\nu$ are homoclinically related.

\medskip
\noindent
{\sc Step 1:} $\hmu \preceq \wh\nu$.

\begin{proof}[Proof of Step $1$]
Let $x\in S$ be a generic point for $\nu$. Since $\mu$ is fully supported, by
Theorems~\ref{thm.unstable.angle} and \ref{thm.big.unstable} there is a set $A\subset\wh S$
of positive $\hmu$--measure such that $W^u(\hy)\pitchfork W^s(x)\neq\emptyset$
for every $\hy\in A$. This proves that $\hmu \preceq \wh\nu$.
\end{proof}

\medskip
\noindent
{\sc Step 2:} $\wh\nu \preceq \hmu$.

\begin{proof}[Proof of Step $2$]
We call $Q\subset S$ a $su$--quadrilateral if it is an open disc
such that its boundary $\partial Q=\partial^s_1Q\cup \partial^u_1Q\cup \partial^s_2Q\cup \partial^u_2Q$
is the union of four connected curves with $\partial^sQ=\partial^s_1Q\cup\partial^s_2Q\subset\vt[\mathfs V^s(\hp)]$
and $\partial^uQ=\partial^u_1Q\cup\partial^u_2Q\subset \vt[\mathfs V^u(\hp)]$
(recall the definition of $\mathfs V^{s/u}$ in Section \ref{ss.invariant.sets}).

Recall the homoclinic relation of a set and a measure introduced in Section \ref{ss.hom.relation}.
Let $\hp$ be a hyperbolic periodic point for $\hf$ homoclinically related to $\wh\nu$,
and fix a $su$--quadrilateral $Q$ associated to $\hp$.
By the inclination lemmas (Propositions~\ref{lambda.u-lemma} and \ref{lambda.s-lemma}),
we can choose $Q$ with diameter
$\ll\ell_0$ such that every $L$--Lipschitz graph centered at $x\in Q$ of size larger than $\ell_0$
intersects $\partial Q$. Therefore, $W^s(x)\cap \partial Q\neq\emptyset$ for $\mu$--a.e. $x\in Q$.
Since stable manifolds either coincide or are disjoint, we conclude that 
$W^s(x)\cap \partial^uQ\neq\emptyset$ for $\mu$--a.e. $x\in Q$.

Let $\wh{K}\subset \wh{S}$ be a Pesin block with positive $\hmu$--measure
such that  $K:=\vartheta(\wh{K})\subset Q$ has positive $\mu$--measure.
Around a density point $x$ of $K$, consider a subset $K'\subset K$ such that
$\mathfs{L}:=\{W^s(x):x\in K'\}$ is a continuous lamination with $C^{1+}$ leaves, Lipschitz holonomies
and transverse dimension equal to one. 
By \cite[Theorem~4.2]{BCS-Annals}, the set $\mathfs T=\{x\in K': W^s(x)\textrm{ is tangent to }\partial^u Q\}$
has transverse Hausdorff dimension smaller than one, and so there is $B\subset K'$ with $\mu(B)>0$
such that $\partial^uQ\pitchfork W^s(x)\neq\emptyset$ for all $x\in B$. Since $\hp$ is homoclinically related
to $\wh\nu$, it follows that $\wh\nu\preceq \hmu$.
\end{proof}

Therefore $\hmu$ and $\wh\nu$ are homoclinically related. By Theorem~\ref{thm.criterion},
$\nu$ is unique.
\end{proof}

\begin{proof}[Proof of Theorem~\ref{thm.unique.ACS}]
Recall that $f_t=E\circ P\circ h_t\circ P^{-1}\in\mathcal U$ for large $t$ \cite[Prop.~6.1]{ACS}, see also
Section \ref{ss.acs}. Fix one such $t$ and let $\lambda^-_\mu(f_t)<0$ be the negative Lyapunov exponent of $f_t$
with respect to $\mu$. Since $h_t$ is a volume-preserving diffeomorphism, 
we have $\operatorname{deg}(f_t)=\det(E)$. By the Pesin entropy formula for endomorphisms
\cite{Liu-entropy-endo}, $h_{\mu}(f_t)=\log \operatorname{deg}(f_t)+\abs{\lambda_\mu^-(f_t)}>\log\abs{\det(E)}$.
Since Lyapunov exponents are continuous inside $\mathcal{U}$ \cite[Theorem~B]{ACS},
if $f$ is $C^1$ close to $f_t$ then
$$
h_\mu(f)=\log \operatorname{deg}(f)+\abs{\lambda_\mu^-(f)}>\log\abs{\det(E)}
$$
and so $h_{\rm top}(f)>\log\abs{\det(E)}$. Given an invariant measure $\nu$, let $F_\nu(f)$
denote its folding entropy, see Section \ref{Section-ph} for the definition. Recall that
this number is $\leq \log \operatorname{deg}(f)=\log\abs{\det(E)}$
and satisfies the equality (\ref{LY}), see \cite{Lin.shu}.
Therefore, if $\nu$ is a measure of maximal entropy for $f$ then
$$
\log\abs{\det(E)}+\abs{\lambda_{\nu}^-(f)}\geq F_\nu(f)+\abs{\lambda_{\nu}^-(f)}>\log\abs{\det(E)}\ 
\Longrightarrow \ \lambda_\nu^-(f)<0
$$
and so $\nu$ is hyperbolic.

By \cite[Proposition~6.1]{ACS}, if $t$ is large then every $f$ of class $C^{1+}$ that is $C^1$ close to 
$f_t$ has the following property: $W^s(x)$ contains a
{\em $v$--segment}\footnote{This is a curve tangent to a vertical cone with size larger than some $\ell_0$.}
for $\mu$--a.e. $x\in S$.
This notion implies that $W^s(x)$ contains a $L$--Lipschitz graph centered at $x$ of size larger
than $\ell_0$, as required in
Theorem~\ref{thm.large.uniqueness}. Hence, $f_t$ has a $C^1$ neighborhood $\mathcal U_t\subset U$
such that every $f\in \mathcal U_t$ of class $C^{1+}$ has at most one measure of maximal entropy.
If $f$ is $C^\infty$, then the existence of such measure follows from \cite{Newhouse-Entropy}.

Finally, we prove that, when it exists, the unique measure of maximal entropy $\nu$ is Bernoulli and fully supported.
We have that $\wh\nu$ is homoclinically related to $\hmu$, and that this latter measure
is Bernoulli \cite[Theorem D]{ACS}. By Theorem~\ref{thm.unique-bernoulli}(2),
it follows that $\nu$ is Bernoulli and fully supported.
\end{proof}

\section{Expanding measures}\label{Section-exp}

In this section we prove Theorem \ref{thm.viana.maps}. Recall from Section \ref{Section-introduction} that for $a_0\in (1,2)$ a parameter such that $t=0$
is pre-periodic for $t\mapsto a_0-t^2$, $d\geq 2$ and $\alpha>0$ small enough, the associated
Viana map is $f=f_{a_0,d,\alpha}$ defined by 
$f(\theta,t)=(d\theta,a_0+\alpha \sin(2\pi \theta)-t^2)$ on $\mathbb S^1\times \mathbb{R}$. 

By normal hyperbolicity,  for any $g$ that is $C^3$--close to $f$, there exists a $g$--invariant center foliation
and a projection $\pi_g:\mathbb S^1 \times \mathbb{R} \to \mathbb S^1$ by this foliation, 
see \cite[Section 2.3]{Viana-maps}.
This defines a quotient map $\overline{g}: \mathbb S^1 \to\mathbb  S^1$ given
by $\overline{g}(\theta) = [\pi_g \circ g \circ \pi_g^{-1}](\theta)$.
The theory of normal hyperbolicity also gives that $g$ is leaf conjugated to $f$,
which is equivalent to $\overline{g}$ being $C^0$--conjugated to $\theta \mapsto d\theta$,
see e.g. \cite{HPS}. 

 Observe that the critical points for $f$ are the points with $t=0$.  Moreover, the second derivative
 in $t$ along the fibers is $-2$. By the Implicit Function Theorem,
 for any $g$ sufficiently $C^3$--close to $f$, the set of critical points is contained in
 a $C^2$ curve near the curve $t=0$. 

Recall that $\mathfs{S}\subset M$ is the set of singularities 
of $f$, which in this case coincides with the set of critical points,  and that $\wh{\mathfs{S}}=\bigcup_{n\in \Z}\hf^n(\vartheta^{-1}[\mathfs{S}])$, see Section~\ref{Section-ALP}.
Recall also the definition of the set $\nuh=\nuh_\chi$ in page \pageref{Def-NUH}.
For each $\chi>0$, let
$${\rm NUE}_\chi=\{\hx\in\nuh_\chi:E^s_{\hx}=\{0\}\}.
$$ 
We consider the set of non-uniformly expanding points of $f$.

\medskip
\noindent
{\sc The set ${\rm NUE}$:} It is defined as the union ${\rm NUE}=\bigcup_{\chi>0}{\rm NUE}_\chi$.

\medskip
Note that if $\hx\in{\rm NUE}$ then:
\begin{enumerate}[$\circ$]
\item $W^u(\hx)$ is an open subset of $M$;
\item $W^s(\hx)=\{\vt[\hx]\}$.
\end{enumerate}

\begin{proof}[Proof of Theorem~\ref{thm.viana.maps}]

Let $f=f_{a_0,d,\alpha}$ with $\alpha$ small.  
Applying \cite[Propositions~12.1 and 12.2]{ALP-23} and by \cite[Section 6]{Alves-Viana-2002},
we find a $C^3$ neighborhood $\mathfs U$ of
$f$ such that every $g\in\mathfs U$:
\begin{enumerate}[$\circ$]
\item satisfies conditions (A1)--(A7),
\item only has adapted expanding measures of maximal entropy, and
\item for any open set $U$, there exists $n = n(U)$ such that $g^n(U)$ contains the maximal invariant set of $g$. 
\end{enumerate}
Fix $g\in \mathfs U$. Suppose that $\mu$ is a measure of maximal entropy for $g$.
Consider $\overline{g}:\mathbb S^1 \to \mathbb S^1$ the quotient map defined above.  Since $\overline{g}$ is conjugated to $\theta \mapsto d\theta$, the entropy of $\overline g$ is $\log d$. 

By \cite{Alves-Viana-2002}, up to reducing the size of $\mathfs U$, the entropy of the unique SRB measure for $g$ is strictly greater than $\log d$ (see also \cite[Prop. 12.2]{ALP-23}).  In particular, $h_{\rm top}(g) > h_{\rm top}(\overline{g})$.  By the Abramov-Rokhlin entropy formula (see \cite{abramov-non-inv}), 
for $\mu$--almost every point the fiberwise entropy is positive. Thus, the measure $\mu_\theta$,
equal to the disintegrated measure on $\pi_g^{-1}(\theta)$, is non atomic for a.e. $\theta$. 

Suppose, by contradiction, the existence of another measure of maximal entropy $\mu'$ for $g$.
Let $\hx \in \widehat{\mathbb S^1 \times I}$ be a $\hmu'$--typical point and $W^u(\hx)$ be its local unstable manifold. Since $\mu'$ is expanding, $W^u(\hx)$ is an open set, and then there is $n>0$ such that 
$g^n(W^u(\hx))$ contains the maximal invariant set.  

Observe that $ \bigcup_{j=0}^n g^j(\mathfs{S})$ is compact and that
$\pi_g^{-1}(\theta) \cap \bigcup_{j=0}^n g^j(\mathfs{S})$ is a finite set for every $\theta$.
Since the measure $\mu_\theta$ is non atomic for a.e. $\theta$, we can find a point
$(\theta, t) \in W^u(\hx)$ such that $g^n(\theta,t)$ is $\mu$--typical and $g^j(\theta,t) \notin \mathfs{S}$
for $0<j<n$.  This shows that $\mu' \preceq \mu$. The same argument with $\mu,\mu'$ 
interchanged shows that $\mu \preceq \mu'$.  Therefore, any two measures of maximal entropy
are homoclinically related. By Theorem \ref{thm.criterion}, we conclude that there exists at most
one measure of maximal entropy.
The same argument shows that $g^n$ has at most one measure of maximal entropy
for every $n>0$. By Proposition \ref{thm.unique-bernoulli},
if $\mu$ is the measure of maximal entropy, then $(g,\mu)$ is Bernoulli.  \qedhere
\end{proof}

\begin{remark}
The above theorem holds in more generality, assuming that $g$ verifies:
\begin{enumerate}[$\circ$]
\item  for every open set $U$ there exists $n = n(U)$ such that $g^n(U) \supset \Omega(g)$,
\item for every measure of maximal entropy $\mu$, it holds
$\mu \left[\bigcup_{j=0}^ng^j(\mathfs{S}) \right]<1$ for every $n>0$, and
\item every measure of maximal entropy is adapted and expanding.
\end{enumerate}
\end{remark}

\appendix

\section{Inclination lemma}\label{appendix}

\renewcommand\thetheorem{\Alph{section}.\arabic{theorem}}

Here we prove the Inclination Lemma stated in Section \ref{ss.hom.relation}.

\begin{customprop}{4.4}[Inclination Lemma]
Let $\wh y\in Y'$, and let $\Delta\subset M$ be a disc of same dimension of $W^u(\hy)$.
If $\Delta$ is transverse to $W^s(\wh f^m(\wh y))$ for some $m\in\Z$, then 
there are discs $D_k\subset \Delta$ and $n_k\to \infty$ such that $f^{n_k}(D_k)$
converges to $W^u(\wh y)$ in the $C^1$ topology.
\end{customprop}

\begin{proof}
We first assume that $m=0$, i.e. that $\Delta\pitchfork W^s(\wh y)\neq\emptyset$.
By the definition of $Y'$, there is a sequence $n_k\to \infty$ such that $\wh f^{-n_k}(\wh y)\to \wh y$ and
$\wh f^{-n_k}(\wh y),\wh y$ belong to the same Pesin block. Since invariant manifolds
are continuous inside Pesin blocks, we have that $W^s(\wh f^{-n_k}(\wh y))$ converges to $W^s(\wh y)$ 
in the $C^1$ topology, hence there are $z_{-n_k}\in \Delta \pitchfork W^s(\wh f^{-n_k}(\wh y))$.

We wish to apply unstable graph transforms to $\Delta$, but a priori $\Delta$ does not define admissible manifolds.
In order to have this, we first need to iterate $\Delta$, as follows.  
Consider the Pesin charts $\Psi_{\wh f^i(\wh y)}^{q^s(\wh f^i(\wh y)),q^u(\wh f^i(\wh y))}$
for $i=-n_k,\ldots,0$. For simplicity, write $\Psi_i=\Psi_{\wh f^i(\wh y)}^{q^s(\wh f^i(\wh y)),q^u(\wh f^i(\wh y))}$.
By \cite[Theorem 3.3]{ALP-23}, the map $F_i=\Psi^{-1}_{i+1}\circ f \circ \Psi_{i}$
is well defined in $R[Q(\wh f^i(\wh y))]$ and has the form 
\begin{equation}\label{eq.diagonal}
F_i(u,v)=\left[\begin{array}{cc} A_i & 0 \\ 0 & B_i\end{array}\right]\left[\begin{array}{c}u\\ v\end{array}\right]+\left[\begin{array}{c}h_{i,1}(u,v)\\ h_{i,2}(u,v)\end{array}\right]
\end{equation}
for $(u,v)\in \mathbb{R}^{d_s}\times \R^{d_u}$, where $\norm{A_i},\norm{B_i^{-1}}^{-1}\leq e^{-\chi}$,
$h_{i,j}(0,0)=(0,0)$, $(dh_{i,j})_{(0,0)}=0$, and $\norm{dh_{i,j}}_{C^0}< \ve$ for $j=1,2$.

To ease the calculations, we can change $\Psi_i$ by composing it with a map 
with small $C^1$ norm and such that $\Psi_i(B^{d_s}[q^s(\wh f^i(\wh y))]\times \{0\})=W^s(\wh f^i(\wh y))$. 
Indeed, if the representing function of $W^s(\wh f^i(\wh y))$ is $F:B^{d_s}[q^s(\wh f^i(\wh y))]\to \R^{d_u}$,
then $\Phi(u,v)=(u,v+F(u))$ is a map with $\Phi(0,0)=(0,0)$, $(d\Phi)_{(0,0)}={\rm Id}$ and $\|\Phi\|_{C^1}<2$,
and $\Psi_i\circ\Phi$ clearly satisfies the required assumption. Hence, we can assume $\Psi_i$ has the form 
(\ref{eq.diagonal}) with $\norm{A_i},\norm{B_i^{-1}}^{-1}\leq e^{-\chi}$,
$h_{i,j}(0,0)=(0,0)$, $(dh_{i,j})_{(0,0)}=0$, and $\norm{dh_{i,j}}_{C^0}< 2\ve$ for $j=1,2$. Under this assumption,
the invariance of $W^s$ implies that $F_i(B^{d_s}[q^s(\wh f^i(\wh y))]\times \{0\})\subset
B^{d_s}[q^s(\wh f^{i+1}(\wh y))]\times \{0\}$ and so $h_{i,2}(x,0)=0$ for every $x\in B^{d_s}[q^s(\wh f^i(\wh y))]$.
Hence 
$$
(dF_i)_{(x,0)}=\left[\begin{array}{cc} A_i+\left(\frac{dh_{i,1}}{du}\right)_{(x,0)} & \left(\frac{dh_{i,1}}{dv}\right)_{(x,0)} \\
& \\
0 & B_i+\left(\frac{dh_{i,2}}{dv}\right)_{(x,0)}\end{array}\right].
$$

Write $z_{-n_k}=\Psi_{-n_k}(w_{-n_k})$ and let $(u,v)\in\R^n$ such that $(d\Psi_{-n_k})_{w_{-n_k}}(u,v)\in T_{z_{-n_k}}\Delta$. The assumption that $z_{-n_k}\in \Delta\pitchfork W^s(\wh f^{-n_k}(\wh y))$ implies that $v\neq 0$.
Let $z_{-n_k+i}=f^i(z_{-n_k})$ and write $z_{-n_k+i}=\Psi_{-n_k+i}(w_{-n_k+i})$, for $i=0,\ldots,n_k$.
Define also the sequence $\{(u_i,v_i)\}$ by
$$
(u_i,v_i)=(dF_{-n_k+i-1})_{w_{-n_k+i-1}}\circ \cdots \circ (dF_{-n_k})_{w_{-n_k}}(u,v),\ \ \ i=0,\ldots,n_k.
$$

\medskip
\noindent
{\sc Claim:} The ratio $\tfrac{\|u_{n_k}\|}{\|v_{n_k}\|}$ goes to zero as $n_k\to\infty$.

\begin{proof}[Proof of Claim.]
By definition, we have
\begin{align*}
\left\{
\begin{array}{l}
u_i=\left(A_{-n_k+i-1}+\left(\tfrac{dh_{-n_k+i-1,1}}{du}\right)_{w_{-n_k+i-1}}\right)u_{i-1}+\left(\tfrac{dh_{-n_k+i-1,1}}{dv}\right)_{w_{-n_k+i-1}}v_{i-1}\\
 \\
v_i=\left(B_{-n_k+i-1}+\left(\tfrac{dh_{-n_k+i-1,2}}{dv}\right)_{w_{-n_k+i-1}}\right)v_{i-1}
\end{array}
\right.
\end{align*}
and so, letting $\lambda=e^{-\chi}+2\ve$ and $\sigma=e^{\chi}-2\ve$, by induction it follows that
\begin{align*}
\left\{
\begin{array}{l}
\norm{u_i} \leq \lambda^i \norm{u}+
\displaystyle\sum_{\ell=0}^{i-1}\lambda^{i-1-\ell}\norm{\left(\tfrac{dh_{-n_k+\ell,1}}{dv}\right)_{w_{-n_k+\ell}}}\norm {v_\ell}\\
\\
\norm{v_i} \geq \sigma^{i-\ell} \norm{v_{\ell}},\ \ \ \ell=0,1,\ldots,i.
\end{array}
\right.
\end{align*}
Therefore
$$
\frac{\norm{u_i}}{\norm{v_i}}\leq \left(\frac{\lambda}{\sigma}\right)^i\frac{\|u\|}{\|v\|}+
\frac{1}{\sigma}\sum_{\ell=0}^{i-1}\left(\frac{\lambda}{\sigma}\right)^{i-1-\ell}\norm{\left(\tfrac{dh_{-n_k+\ell,1}}{dv}\right)_{w_{-n_k+\ell}}}.
$$
Since $\ve>0$ is small, we have $\lambda<1<\sigma$, hence the first term goes to zero when $i\to \infty$.
The second term has the form $\sum_{\ell=0}^i a_{i-\ell} \theta^\ell$ with $\theta=\tfrac{\lambda}{\sigma}<1$ and
$a_\ell\to 0$ as $\ell\to\infty$, since
$$
\norm{\left(\tfrac{dh_{-n_k+\ell,1}}{dv}\right)_{w_{-n_k+\ell}}}=
\norm{\left(\tfrac{dh_{-n_k+\ell,1}}{dv}\right)_{w_{-n_k+\ell}}-
\left(\tfrac{dh_{-n_k+\ell,1}}{dv}\right)_{(0,0)}}\leq C\|w_{-n_k+\ell}\|
$$
and $\|w_{-n_k+\ell}\|\leq {\rm const}\cdot e^{-\frac{\chi}{2}\ell}$, by \cite[Prop. 4.7(4)]{ALP-23}. Letting
$i=n_k$, we conclude that $\tfrac{\|u_{n_k}\|}{\|v_{n_k}\|}\to 0$ as $n_k\to\infty$. 
\end{proof}

Therefore, for $n_k$ large enough, $f^{n_k}(\Delta)$
contains a disc $\Delta_0$ of the same dimension of $W^u(\hy)$ that is contained on a $u$--admissible manifold $\wt\Delta_0$
at $\Psi_0$.\footnote{Following \cite{Sarig-JAMS}, the work \cite{ALP-23} defines $u$--admissible manifolds
requiring that a H\"older norm is at most 1/2. This constant is arbitrary and can be changed to a large constant
controlling the respective H\"older norm of $\wt\Delta_0$, so that $\wt\Delta_0$ becomes
$u$--admissible. Then for $\ve>0$ small enough the graph transforms are well-defined
and are contractions as in \cite{ALP-23}.}
Fix one such $n_{k_0}$, $\Delta_0$ and $\wt\Delta_0$. 
Since Pesin charts vary continuously on Pesin blocks, $\wt\Delta_0$ is also $u$--admissible
at $\Psi_{-n_k}$ for large $k$. For that, we just have to adjust the sizes of the representing functions, dividing the
windows parameters of $\Psi_i$ by a constant $a<1$. By the definition of unstable manifold, 
if $\mathfs F_i$ represents the unstable graph transform associated to $\Psi_i\to \Psi_{i+1}$, then  
$$
V_k:=(\mathfs F_{-1}\circ\mathfs F_{-2}\circ\cdots\circ \mathfs F_{-n_k})(\wt\Delta_0)
$$
defines a sequence $\{V_k\}$ of $u$--admissible manifolds at $\Psi_0$ that converges to
$W^u(\wh y)$ in the $C^1$ topology. If $k$ is large enough, then $V_k=f^{n_k}(D_k)$ for some disc
$D_k\subset \Delta_0$. This concludes the proof when $m=0$.

Now we assume that $\Delta$ is transverse to $W^s(\wh f^m(\wh y))$ for some $m$.
We claim that the proof will be complete once we prove that some forward iterate of $\Delta$ is transverse to
$W^s(\wh y)$. Indeed, assume there is $\Delta'\subset\Delta$ such $f^\ell(\Delta')$ is transverse
to $W^s(\wh y)$ for some $\ell\geq 0$. By the first part of the proof, there are discs $\Delta_k\subset f^\ell(\Delta')$
and $n_k\to\infty$ such that $f^{n_k}(\Delta_k)$ converges to $W^u(\wh y)$ in the $C^1$ topology.
Letting $g$ be the inverse branch of $f^\ell$ such that $g(f^\ell(\Delta'))=\Delta'$,
we have that $D_k=g(\Delta_k)\subset\Delta'\subset\Delta$ are discs such that 
$f^{n_k+\ell}(D_k)=f^{n_k}(\Delta_k)$ converges to $W^u(\wh y)$ in the $C^1$ topology, proving the claim.
To prove that some forward iterate of $\Delta$ is transverse to
$W^s(\wh y)$, we also use the first part of the proof: there are discs $\Delta_k\subset\Delta$ and $n_k\to \infty$ such that
$f^{n_k}(\Delta_k)$ converges to $W^u(\wh f^m(\wh y))$ in the $C^1$ topology. Let
$\wt g=f^{-1}_{y_0}\circ\cdots\circ f^{-1}_{y_{m-1}}$.
Then $\wt g(f^{n_k}(\Delta_k))$ converges to $\wt g(W^u(\wh f^m(\wh y)))$
in the $C^1$ topology. By \cite[Prop. 4.7(2)]{ALP-23}, this latter set is a subset of 
$W^u(\wh y)$ containing $y_0$. Hence, if $k$ is large enough, $\wt\Delta:=\wt g(f^{n_k}(\Delta_k))$
is transverse to $W^s(\wh y)$. For $k$ large we have
$\wt\Delta=f^{n_k-m}(\Delta_k)$, thus proving that
the $(n_k-m)$--th forward image of $\Delta$ is transverse to
$W^s(\wh y)$. This concludes the proof for arbitrary $m$.
\end{proof}

The same proof, applying the proper inverse branchs, gives a stable version of the
inclination lemma. Given $\hy=(y_n)_{n\in\Z}$, write $f^{-n}_{\hy}=f^{-1}_{y_{-n}}\circ\cdots\circ f^{-1}_{y_{-1}}$.

\begin{proposition}[Inclination lemma - stable version]\label{lambda.s-lemma}
Let $\wh y\in Y'$, and let $\Delta\subset M$ be a disc of same dimension of $W^s(\hy)$.
If $\Delta$ is transverse to $W^u(\wh f^m(\wh y))$ for some $m\in\Z$, then 
there are discs $D_k\subset \Delta$ and $n_k\to \infty$ such that $f^{-n_k}_{\hy}(D_k)$
converges to $W^s(\wh y)$ in the $C^1$ topology.
\end{proposition}

\bibliographystyle{alpha}
\bibliography{bibliography}{}

\end{document}